\documentclass[a4paper, reqno]{amsart} 

\usepackage{amsthm, amssymb, amsmath, latexsym, upgreek,accents}
\usepackage{amsfonts, mathrsfs, color}
\usepackage{a4wide}
\usepackage{hyperref}
\usepackage{graphicx,color}
\usepackage{pinlabel-test}
\usepackage{booktabs}
\usepackage{textcomp}
\usepackage{subfigure}

\numberwithin{equation}{section}

\theoremstyle{plain}                     
\begingroup
\newtheorem{theorem}{Theorem}[section]       
\newtheorem{prop}[theorem]{Proposition}     
        
\newtheorem{lemma}[theorem]{Lemma}      
\endgroup
      
\theoremstyle{definition}                
\begingroup
\newtheorem{defin}[theorem]{Definition}
\newtheorem{remark}[theorem]{Remark}
\endgroup

\newcommand{\R}{\mathbb R}
\newcommand{\N}{\mathbb N}
\newcommand{\Z}{\mathbb Z}
\newcommand{\C}{\mathbb C}
\newcommand{\e}{\varepsilon}

\newcommand{\mtwo}{\mathbb R^{2\times2}}
\newcommand{\HH}{{\mathcal H}}
\newcommand{\Gmu}{\beta_{\mu}}
\newcommand{\umu}{u_{\mu}}

\def\P{\mathcal P}

\def\HH{\mathcal H}
\def\A{\mathcal A}

\def\n{n}
\def\introf{\mathfrak f}

\DeclareMathOperator\dist{dist}

\DeclareMathOperator\supp{supp}

\DeclareMathOperator\skw{skew}
\DeclareMathOperator\sym{sym}
\DeclareMathOperator\Div{div}
\DeclareMathOperator\curl{Curl}

\def\weakto{\rightharpoonup}

\begin{document}

\title[Evolution of dislocations with Wasserstein dissipation]{Convergence of interaction-driven evolutions of dislocations with Wasserstein dissipation and slip-plane confinement}
\author{Maria Giovanna Mora \and Mark Peletier \and Lucia Scardia}

\address[M.G. Mora]{Dipartimento di Matematica, Universit\`a di Pavia, via Ferrata 1, 27100 Pavia, Italy}
\email{mariagiovanna.mora@unipv.it}

\address[M. Peletier]{Department of Mathematics and Computer Science and Institute for Complex Molecular Systems, TU Eindhoven, The Netherlands}
\email{M.A.Peletier@tue.nl}

\address[L. Scardia]{School of Mathematics and Statistics, University of Glasgow, 15 University Gardens, G12 8QW, Glasgow, United Kingdom}
\email{Lucia.Scardia@glasgow.ac.uk}

\begin{abstract}
We consider systems of $n$ parallel edge dislocations in a single slip system, represented by points in a two-dimensional domain; the elastic medium is modelled as a continuum. We formulate the energy of this system in terms of the empirical measure of the dislocations, and prove several convergence results in the limit $n\to\infty$.

The main aim of the paper is to study the convergence of the evolution of the empirical measure as $n\to\infty$. We consider rate-independent, quasi-static evolutions, in which the motion of the dislocations is restricted to the same slip plane. This leads to a formulation of the quasi-static evolution problem in terms of a modified Wasserstein distance, which is only finite when the transport plan is slip-plane-confined. 

Since the focus is on interaction between dislocations, we renormalize the elastic energy to remove the potentially large self- or core energy. We prove Gamma-convergence of this renormalized energy, and we construct joint recovery sequences for which both the energies and the modified distances converge. With this augmented Gamma-convergence we prove the convergence of the quasi-static evolutions as $n\to\infty$.
\end{abstract}

\maketitle

\section{Introduction}
It is well known that plastic, or permanent, deformation in metals is caused by the concerted movement of many curve-like defects in the crystal lattice, called \emph{dislocations}. What is not yet known is how to use this insight to predict behaviour at continuum scales. It would be natural to take a sequence of systems with increasing numbers of dislocations, and derive an effective description in terms of dislocation \emph{densities}. In various cases  formal and rigorous convergence results have been proved of the elastic energies generated by the dislocations (see the discussion of the literature below), and this provides a good starting point.

However, macroscopic plasticity heavily depends on \emph{dynamic} properties of the dislocation curves. The most basic aspect of the motion of dislocation curves is the confinement to \emph{slip planes}: each curve can move only in the plane spanned by its \emph{Burgers vector} and the tangent to the curve. Other essential dynamic phenomena are \emph{creation} and \emph{annihilation} of dislocations, and their behaviour at \emph{obstacles}; especially in small systems, these latter phenomena are believed to be the main determining factors in the plastic behaviour of metals~\cite{DeshpandeNeedlemanVan-der-Giessen05,ChakravarthyCurtin10}. 

Although models exist that describe the motion, creation, annihilation, and obstacle behaviour of dislocations at the level of densities (see e.g.~\cite{GromaBalogh99,GromaCsikorZaiser03,YefimovGromaVanderGiessen04,YefimovGiessen05,YefimovVanderGiessen05a,BakoGromaGyorgyiZimanyi06,GromaGyorgyiKocsis07,GromaGyorgyiIspanovity10}), these are phenomenological in nature, and the connection between these models and more microscopic descriptions of dislocation motion is tenuous. Many different models exist, and at this moment no good method is available to compare these or choose between them. For a better understanding of the complex motion of dislocations it is therefore essential to understand the upscaling of discrete-dislocation models to descriptions at the level of densities.

A mathematical theory of creation, annihilation, and obstacles can only be formulated in the context of dislocation motion that is confined to the slip plane. In this paper we therefore make a first step in the direction of the dynamics of dislocations, by proving a rigorous upscaling of a system of moving edge dislocations in two dimensions \emph{with slip-plane confinement}. Although this confinement may seem a minor restriction, it actually makes the proof of the upscaling far more complex.

\subsection{Setup}
We restrict ourselves to straight and parallel \emph{edge} dislocations in plane strain, and we assume that only one slip system is active, with Burgers vector $b = e_1\in \R^2$, and that all dislocations have the same sign. Since we will take the many-dislocation limit, we describe the positions of the dislocations by a measure on a bounded open set $\Omega\subset \R^2$ (the cross section) of the form
$$
\mu = \frac 1n\sum_{i=1}^n \delta_{z_i},
\qquad \{z_i\}_{i=1}^n \subset \Omega. 
$$
The dislocations interact with each other through the elasticity of the medium, which we assume to be a homogeneous, isotropic, linearly elastic continuum; its properties are characterized by the fourth-order stress-strain tensor $\mathbb C$, which satisfies $\mathbb{C} F=\lambda\, {\rm tr}({\rm sym}\,F) {\rm Id} +2\mu \, {\rm sym}\,F$ with Lam\'e constants $\lambda+\mu\geq0$ and $\mu>0$. 

Since the elastic energy density is not integrable close to a dislocation, we employ the \emph{core-region approach} of removing disks of radius $\e_n\to 0$ around the dislocations, leading to an effective domain $\Omega_n(\mu):=\Omega\setminus \bigcup_{j=1}^n \overline B_{\e_n}(z_j)$. In addition we assume that the dislocations are separated from each other by a distance $r_n\to0$ and that they can only live in a closed rectangle $\mathcal R$ in $\Omega$ whose sides are parallel to the coordinate axes and which has a distance $\ell>0$ from the boundary $\partial \Omega$. We assume that $\e_n$ and $r_n$ satisfy 
\begin{equation}\label{hyp:enrn}
\e_n\to 0 ,\quad r_n \to 0, \quad  \e_n /r_n^{3}\to0, \quad r_n n \to 0 \quad \textrm{as} \,\, n\to \infty.
\end{equation}
See Section~\ref{cond-separation} for comments on this choice.

At finite $n$, admissible measures $\mu$ belong to the set 
\begin{equation}\label{def:Xn}
X_n := \Bigl\{ \mu = \frac1n \sum_{i=1}^n \delta_{z_i} \text{ such that } \{z_i\}_{i=1}^n \subset \mathcal R \text{ and } |z_i-z_j|\geq r_n \text{ for }i\not= j \Bigr\}.
\end{equation}
The elastic energy of a measure $\mu\in X_n$, defined for convenience on the set $\mathcal{P}(\Omega)$ of all probability measures, is 
\begin{equation}
\label{def:Fn}
F_n(\mu) := \begin{cases}
\displaystyle
\inf_{\beta\in \mathcal A_n(\mu)} E_n(\mu,\beta) &\text{if }\mu\in X_n,
\\
+\infty&\text{if } \mu\in \mathcal{P}(\Omega)\setminus X_n,
\end{cases}
\end{equation}
where
\begin{equation}
\label{def:En}
E_n(\mu,\beta) := \frac12 \int_{\Omega_n(\mu)} \mathbb C\beta:\beta \,dx.
\end{equation}
The tensor-valued field $\beta$ has the interpretation of the elastic part of the strain. The admissibility class $\A_n(\mu)$ characterizes the incompatibility conditions on $\beta$ generated by the dislocations: 
\begin{multline}
\A_n(\mu):=\Big\{\beta \in L^2(\Omega;\mtwo): \ \beta=0 \text{ in } \Omega\setminus\Omega_n(\mu), \quad 
\curl \beta = 0  \mbox{ in }\Omega_n(\mu),\\
\int_{\partial B_{\e_n}(z_i)}\beta\, \tau\, d\HH^1 = \frac bn \text{ for every }i=1,\dots, n
\Big\},
\label{def:Aemu}
\end{multline}
where $\tau$ is the tangent to $\partial B_{\e_n}(z_i)$ and the integrand $\beta\,\tau$ is understood in the sense of traces (see \cite{CermelliLeoni06} for details).
The minimization problem~\eqref{def:Fn} is a reformulation of the standard elasticity problem, in terms of the elastic strain $\beta$, with stress-free boundary conditions at $\partial\Omega_n(\mu)$.

Note that since the dislocation density $\mu\in X_n$ is rescaled to have mass one, the incompatibility of the strain $\beta$ at every dislocation is of order $1/n$. We also observe that in alternative to the integral incompatibility condition in \eqref{def:Aemu} one could require the more familiar condition on the circulation of the strain:
$$
\curl\beta = \frac1n\frac{b}{2\pi \varepsilon_n}\sum_{i=1}^n \mathcal{H}^{1}\llcorner{\partial B_{\varepsilon_n}}(z_i)
\quad \text{in }\Omega.
$$
For more background on dislocations in general, see e.g.~\cite{Kroener81,HirthLothe82,Callister07}; a more detailed derivation of this model can be found in~\cite{VanderGiessenNeedleman95,CermelliLeoni06,GarroniLeoniPonsiglione10}.

\subsection{$\Gamma$-convergence of the renormalized energy}

The discrete evolutionary system at finite~$n$ is defined by the energy functional $ F_n$ and a dissipation distance that we introduce below. We first  study the behaviour of $F_n$ as $n\to\infty$.

Garroni, Ponsiglione, and co-workers~\cite{GarroniLeoniPonsiglione10,DeLucaGarroniPonsiglione12} show that as $n\to\infty$ and $\e_n\to0$ the energy $F_n$ has contributions of order $1$ and of order $n^{-1}|\log\e_n|$. The  contributions of order $1$ stem from the interaction between pairs of distant dislocations, of which there are $n^2$, each with strength of order $n^{-2}$ by the scaling in~\eqref{def:Aemu}. 
The contributions of order $n^{-1}|\log \e_n|$ arise from the energy of order $n^{-2}|\log \e_n|$ contained in a neighbourhood of each of the $n$ dislocations. Depending on the relative size of $n$ and $|\log \e_n|$, one or the other will dominate.

In this paper we consider an evolution that conserves the total number of dislocations, and therefore the self-energy of order $n^{-1}|\log\e_n|$ is of no relevance, even when it is large. To extract the interaction energy, we renormalize $F_n$ by defining 
\begin{equation}
\label{def:curly-F-intro}
\mathcal F_n (\mu) := \begin{cases}
\displaystyle
F_n(\mu) - \frac1{2\n^2}\sum_{i=1}^{\n} \int_{\Omega_n(\mu)}\mathbb CK^n_{z_i}:K^n_{z_i} \,dx
& \quad\text{if } \mu\in X_n, \ \displaystyle\mu=\frac1n\sum_{i=1}^n\delta_{z_i},
\\
+\infty 
& \quad\text{if } \mu\in \mathcal{P}(\Omega)\setminus X_n.
\end{cases}
\end{equation}
Here $K^n_z$ is a small correction of the canonical strain field $K_z$ generated by a single dislocation at $z$ in $\R^2$, as defined in Section~\ref{subsec:aux-functions}. The second term has the same scaling $n^{-1}|\log \e_n|$ as the self-energy, and approximately cancels the self-energy in the first term. 
The aim of this renormalization is to extract that part of the energy that characterizes the interaction, and our first main result (Theorem~\ref{ups:density}) makes clear in which sense this is indeed the case:

\begin{theorem}
\label{th:main1-intro}
Under conditions \eqref{hyp:enrn} on $\e_n$ and $r_n$, the functionals $\mathcal F_n$ $\Gamma$-converge in the space of probability measures $\P(\Omega)$ endowed with the narrow topology to the limit functional
\begin{multline}
\label{def:limitF}
\mathcal F(\mu) := \frac12 \iint_{\Omega\times\Omega} V(y,z)\,d\mu(y) d\mu(z) \\
  + \inf_v \Bigl\{\frac12 \int_\Omega \mathbb C\nabla v:\nabla v \,dx
          + \int_\Omega\int_{\partial\Omega} \C K(x;y)\nu(x)\cdot v(x)\,d\HH^1(x)\,d\mu(y)\Bigr\},
\end{multline}
where
\begin{equation}\label{def:V-intro}
V(y,z):=\int_{\Omega}\C K_y(x):K_z(x)\,dx.
\end{equation}
\end{theorem}

The first term in the limit $\mathcal F$ is a two-point interaction functional, with interaction potential $V(y,z)$. If we  replace the integration over $\Omega$ in~\eqref{def:V-intro} by integration over $\R^2$ we  obtain the formula for the interaction between segments of dislocations in infinite space that is widely used in the engineering literature (see e.g.~\cite[(5--16)]{HirthLothe82}). The distinction between this formula and $V$, and the second term in~\eqref{def:limitF}, is generated by the boundedness of $\Omega$.

\subsection{Quasi-static evolution of dislocations}

In the second part of this paper we use the convergence result above to pass to the limit in a \emph{rate-independent} or \emph{quasi-static} evolution that is driven by $\mathcal F_n$.

This rate-independent evolution arises from the model underlying dislocation motion, and in this section we describe the background of this model. Since dislocations are defects in the atomic lattice, their movement can be likened to that of a point particle in a periodic potential: the wells of this potential correspond to energetically favourable positions of the defect, and their spacing is exactly one Burgers vector. This locally periodic potential is tilted by a global driving force that varies on larger scales and arises from the bulk elasticity. 

Motion in this tilted periodic potential is assumed to arise from thermal fluctuations, leading to a jump process between wells with a rate that is determined by the rate of escape from a well. Orowan~\cite{Orowan40} first formulated this concept of dislocation motion and the Arrhenius rate expression that follows from it. In the one-dimensional case, with a periodic potential with wells of depth $e$ and spacing $\ell$, and a global force $\introf$, the rates of jumping right and left are
\[
r_+ = \upalpha \mathrm e^{-\upbeta (e - \introf\ell)}
\qquad r_- = \upalpha \mathrm e^{-\upbeta (e + \introf\ell)},
\]
where $\upalpha>0$ is a fixed constant and $\upbeta = 1/kT$ is inverse temperature. The expected velocity therefore is equal  to $2\upalpha\,\exp(-\upbeta e) \sinh(\upbeta \introf\ell)$.

`Rate-independence' appears in this expression for the rate of motion when we take the low-temperature limit $\upbeta\to\infty$:
\begin{equation}
\label{eq:RIlaw}
2\upalpha \,\mathrm e^{-\upbeta e} \sinh(\upbeta \introf\ell) \quad\xrightarrow {\upbeta\to\infty}\quad m\Bigl(\frac{\introf\ell}e\Bigr), 
\qquad m(s) := 
\begin{cases}
\emptyset  & \text{if }s<-1\\
(-\infty,0] & \text{if } s=-1\\
\{0\} & \text{if } -1 < s<1\\
[0,\infty) & \text{if } s=1\\
\emptyset & \text{if } s>1.
\end{cases}
\end{equation}
Here the convergence is in the sense of graphs, and the limit $m$ is a graph, i.e., a set-valued function.
In words: at low temperatures, motion only takes place when $\introf = \pm e/\ell$, and for those two values of $\introf$ any velocity is possible; for smaller forces $\introf$ the motion is arrested, and larger forces should never appear (see Figure~\ref{rate-independent}).   What constitutes low temperatures can be understood in terms of an energy comparison: low temperatures are those in which the thermal energy $kT = \upbeta^{-1}$ is small with respect to the activation energy $e$.\footnote{This description is simplified from many points of view; for instance, dislocations are generally curved, not straight, and micro-obstacles play an important role in the motion of dislocations. See e.g.~\cite{TeodosiuSidoroff76,OrtizPopov82} for more detailed treatments of three-dimensional dislocations.}

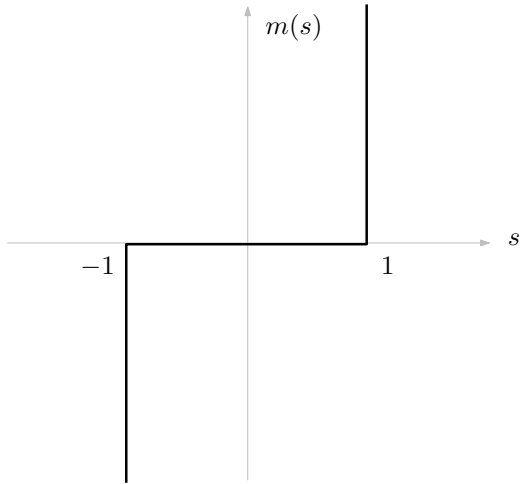
\begin{figure}[h]
\labellist
\pinlabel $s$ [l] <3pt,0pt> at 483 259
\pinlabel $m(s)$ at 325 430
\pinlabel $-1$ [tl] at 150 250
\pinlabel $1$ [tl] at 390 250
\endlabellist
\begin{center}
\includegraphics[width=2.5in]{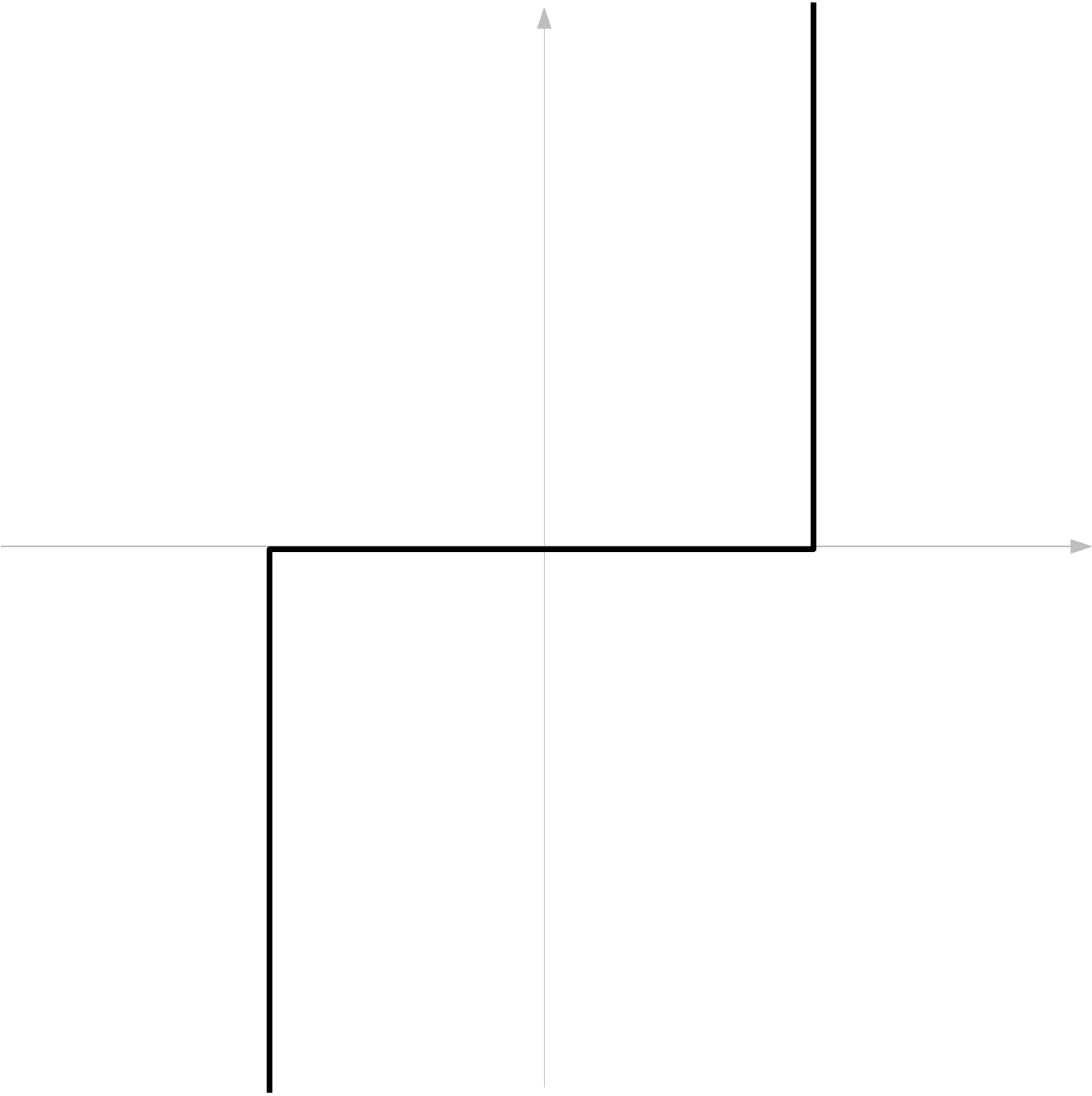}
\end{center}
\caption{The rate-independent limit.}
\label{rate-independent}
\end{figure}
\bigskip

Inspired by these arguments we choose a rate-independent evolution for the discrete system. In order to force a non-trivial evolution, we add a smoothly time-varying external load of the form
\begin{equation}\label{loading-term}
\int_{\Omega} f(x,t)\,d\mu(x).
\end{equation}
The function $f$ can be interpreted as a time-varying potential for dislocations, in the sense that the derivative $\partial_{x_1}f(x,t)$ acts as an additional horizontal force on a dislocation at $x$. For instance, applying an appropriate time-varying boundary traction, one can achieve a spatially uniform shear stress $\sigma(t)$ in the bulk, which can be represented with the potential $f(x,t) = \sigma(t)x_1$.
The total energy for this system is then
\begin{equation}\label{def:tildeFn}
\widetilde {\mathcal F}_n (\mu,t) = \mathcal  F_n(\mu) - \int_{\Omega} f(x,t) \,d\mu(x).
\end{equation}

We assume that each of the dislocations $z_i$ evolves by a rate-independent law of the type~\eqref{eq:RIlaw}. The driving force $\introf$  for each of the dislocations is the derivative of the energy with respect to the position $z_i$, $-\nabla_{z_i} \widetilde{\mathcal F}_n\bigl(\frac1n \sum_{j=1}^n\delta_{z_j},t\bigr)$. Since our edge dislocations are restricted to movement in the direction of the Burgers vector $b=e_1$, the motion is one-dimensional, and only the horizontal or first component of this force leads to motion: 
\begin{equation}
\label{eq:QSE-RN}
\dot z_i(t) \in e_1 \, m\Bigl(-e_1\cdot \nabla_{z_i} \widetilde{\mathcal F}_n\Bigl(\frac1n \sum_{j=1}^n\delta_{z_j(t)},t\Bigr)\Bigr).
\end{equation}
Note how the `equation' above is an inclusion, since the right-hand side is multivalued.
In~\eqref{eq:QSE-RN} we have normalized the constant $\ell/e$ to $1$ for simplicity. 

The flow rule~\eqref{eq:QSE-RN} can be interpreted as follows: for each $i$ separately, and at each $t$, \emph{either} the velocity $\dot z_i(t)$ is zero, \emph{or} the force $-e_1\cdot \nabla_{z_i} \widetilde{\mathcal F}_n$ equals $\pm1$; in the latter case, $\dot z_i$ is parallel to $e_1$ and points in the same direction as the force.

\subsection{Measure-valued formulation}
Instead of considering solutions of~\eqref{eq:QSE-RN} in the classical sense, we focus on solutions in the \emph{energetic} or \emph{quasi-static} sense~\cite{Mielke05a}, given by Definition~\ref{def:quasistat} below. For this we recast the problem as a rate-independent system in the space of probability measures $\P(\Omega)$, driven by~$\widetilde{\mathcal F}_n$, with a Wasserstein-type dissipation, which we now introduce.

To motivate the definition, consider a smooth curve $\mu(t) = \frac1n \sum_{i=1}^n \delta_{z_i(t)}$. Since dislocations move parallel to $b=e_1$, $\partial_t z_i(t)\cdot e_2 = 0$; therefore
\[
\partial_t \int_\Omega \varphi(x_2) \, d\mu(t)(x) = \frac1n \partial_t \sum_{i=1}^n \varphi(z_i(t)\cdot e_2) = 0
\qquad\text{for any }\varphi\in C^1(\R).
\]
Along the curve the integral on the left-hand side is preserved, implying that if $\mu$ and $\nu$ can be transported into each other, then $\int \varphi(x_2) \,d(\mu-\nu)(x) =0$; or equivalently, $(\pi_2)_\#\mu = (\pi_2)_{\#}\nu$, where $\pi_2:\Omega\to\R$ is the coordinate mapping $\pi_2(x):= x_2$. This leads us to define the distance function
\begin{equation}
\label{def:d}
d(\mu,\nu) := \begin{cases}
\inf_{\gamma\in \Gamma(\mu,\nu)} \iint_{\Omega\times\Omega} |x-y|\, d\gamma(x,y)
  & \text{if } (\pi_2)_\#\mu = (\pi_2)_\#\nu,\\
+\infty & \text{otherwise,}
\end{cases}
\end{equation}
where $\Gamma(\mu,\nu)$ is a restricted set of couplings of $\mu$ and $\nu$,
\begin{multline}
\label{def:Gamma}
\Gamma(\mu,\nu) := \bigl\{\gamma\in \P(\Omega\times\Omega): \gamma(A\times \Omega) = \mu(A),\ 
  \gamma(\Omega\times A) = \nu(A)\text{ for all Borel sets }A\subset \Omega, \\
  \text{ and }
 \pi_2(x)=\pi_2(y) \text{ for $\gamma$-a.e.\ } (x,y)\in\Omega\times\Omega \bigr\}.
\end{multline}
This is the usual $1$-Wasserstein or Monge-Kantorovich transport distance on $\P(\Omega)$~\cite{Villani03}, except for the additional restriction that $\pi_2(x)=\pi_2(y)$ for $\gamma$-a.e.\ $(x,y)$; this restriction forces the transport to move parallel to $b=e_1$. We describe some properties of this distance in Lemma~\ref{lemma:props-d}.\medskip

Now we have defined the driving functional $\widetilde{\mathcal F}_n$ and the dissipation distance $d$, a solution of the quasi-static evolution is defined as follows.
\begin{defin}\label{def:quasistat}
A solution of the quasi-static evolution associated with $\widetilde{\mathcal F}_n$ is a curve $t\mapsto \mu(t)$ from $[0,T]$ into $X_n$  that satisfies the two following conditions: 
\begin{itemize}
\item {\em global stability:} for every $t\in [0,T]$ we have 
$$
{\widetilde{\mathcal{F}}}_n(\mu(t),t) \ \leq\   d(\nu, \mu(t)) + \widetilde{\mathcal{F}}_n(\nu,t) 
$$
for every $\nu \in X_n$;

\item {\em energy balance:} for every $t\in [0,T]$
$$
\widetilde{\mathcal{F}}_n(\mu(t),t) + \mathcal{D}(\mu,[0,t]) 
\ =\ \widetilde{\mathcal{F}}_n(\mu(0),0)  - 
\int_0^{t}\int_{\Omega}\dot{f}(s)\, d\mu(s) ds,
$$
where $\mathcal{D}(\mu,[0,t])$ is the total dissipation of $\mu$ on $[0,t]$, 
\begin{equation}\label{e:var-diss}
\mathcal{D}(\mu,[0,t]):= \sup\Big\{\sum_{i=1}^Md(\mu(t_i), \mu(t_{i-1})): \ 0=t_0<\dots<t_M=t, \ M\in\N\Big\}.
\end{equation}
\end{itemize}
\end{defin}

In Section~\ref{Sect:quasistatic} we prove the existence of a solution for this evolution at fixed $n\in \N$ (Theorem~\ref{th:existence-finite-n}).

\subsection{Convergence of the quasi-static evolutions}

Now we have defined the energy of the discrete system and its evolution in time, we turn to the convergence of this evolutionary system as $n\to\infty$.

The renormalized $\Gamma$-convergence result in terms of measures (Theorem~\ref{th:main1-intro}) is the first ingredient in passing to the limit, as $n\to\infty$, in the quasi-static evolution that we formulated above. The second ingredient is a stronger recovery result, the existence of  a \emph{joint} or \emph{mutual} recovery sequence~\cite{MielkeRoubicekStefanelli08}, for which both $\mathcal F_n$ and $d$ converge. Specifically, if $\mu_n\weakto\mu$, and if $\nu$ is given, we need to construct $\nu_n$ such that 
$$
\mathcal F_n(\nu_n) + d(\mu_n,\nu_n) \to \mathcal F(\mu) + d(\mu,\nu).
$$
This is a stronger property than Theorem~\ref{th:main1-intro} in a non-trivial way. To start with, the topology of $d$ is stronger than the narrow topology, implying that the recovery sequence generated by Theorem~\ref{th:main1-intro} does not necessarily converge in $d$. The situation is worse, however: since $d(\mu,\nu)$ is only finite on measures with equal vertical marginals ($(\pi_2)_\#\mu=(\pi_2)_\#\nu$), the finiteness of $d(\mu_n,\nu_n)$ requires $\nu_n$ to be an exact horizontal transport of $\mu_n$, i.e., the slip planes of $\mu_n$ and $\nu_n$ have to coincide, and the numbers of dislocations on each of the slip planes have to be the same for the two. Note that the recovery sequence $\nu_n$ necessarily depends on the whole sequence $\mu_n$.

Because of this tight binding between $\mu_n$ and $\nu_n$, we prove the existence of such a joint recovery sequence under a restriction on $\mu_n$: we require the minimal slip-plane spacing in $\mu_n$ to be bounded from below, and the numbers of dislocations on each slip plane to be bounded from above, both by powers of~$n$ (see Definition~\ref{def:classYng} and Theorem~\ref{thm:joint}). In the context of convergence in the evolutions this can relatively easily be achieved by imposing the same restrictions on the approximations of the initial datum $\mu^0$, since the evolution preserves these restrictions. The set of measures that can be approximated this way includes all measures with bounded two-dimensional Lebesgue density, but also some concentrated measures; see Lemma~\ref{lemma:char-Pinfty} for details.\medskip

Armed with this joint recovery sequence and setting
$$
\widetilde{\mathcal F}(\mu, t):=\mathcal F(\mu)-\int_\Omega f(t,x)\mu(t,x)\, dx,
$$
in Section~\ref{conv:evolution} we prove the following theorem (Theorem~\ref{th:conv_evolutions}):

\begin{theorem}
\label{th:plasticlimit}
Consider $\mu^0\in \P(\Omega)$ such that there exists an approximating sequence $\mu_n^0$ according to Theorem~\ref{thm:joint}, with $\mu_n^0\weakto \mu^0$  and 
$\mathcal{F}_n(\mu^0_n)\to \mathcal F(\mu^0)$. Assume that $f\in W^{1,1}(0,T;C(\Omega))$ is given. 
For every $n\in\N$ let $t\mapsto\mu_n(t)$ be a quasi-static evolution on $[0,T]$ associated with $\widetilde{\mathcal F}_n$ and with initial value $\mu^0_n$, according to Definition~\ref{def:quasistat}. Then 
\begin{enumerate}
\item \textbf{Compactness:} There exists a subsequence $\mu_n$ (without change in notation) and a limit curve $\mu:[0,T]\to\P(\Omega)$ such that $\mu_n(t)\weakto \mu(t)$ for all $t\in [0,T]$; 
\item \textbf{Convergence:} The curve $\mu$ is a quasi-static evolution associated with $\widetilde{\mathcal F}$ and with initial value $\mu^0$ (see \eqref{0min-1} and \eqref{0e-bal-1}). 
\end{enumerate}
\end{theorem}

\subsection{Strong formulations}
In Section~\ref{sec:strong-solutions} we show that, as we claimed above, the quasi-static formulation of Definition~\ref{def:quasistat} is equivalent to the strong formulation of~\eqref{eq:QSE-RN} when the solution is smooth. In that section, we also give a formal argument suggesting that the quasi-static evolution satisfied by the limiting measures $\mu$, as given by Theorem~\ref{th:plasticlimit}, has a similar structure, which we now describe.

We consider evolving measures $t\mapsto \mu(t)$ that are smooth in space and time, and which have the property that there exists a smooth scalar field $\phi = \phi(t,x)$ such that 
$$
\partial_t\mu +\partial_{x_1}( \phi\mu )= 0 \qquad\text{in the sense of distributions.}
$$
The field $\phi$ has the interpretation of the horizontal velocity of the dislocations.
Such a curve of measures solves the quasi-static evolution problem associated with $\widetilde{\mathcal F}$, i.e., conditions \eqref{0min-1} and \eqref{0e-bal-1}, if and only if
$$
-\phi(t,x) \;\partial_{x_1} \frac{\delta\widetilde{\mathcal F}}{\delta\mu}(\mu(t),t)
 = |\phi(t,x)| \qquad
 \text{for $\mu(t)$-a.e. $x$.}
$$
Here $\delta\widetilde{\mathcal F}/{\delta\mu}$ is the variational derivative of $\widetilde{\mathcal F}$, and its $x_1$-derivative at a point $x$ has the interpretation of the $x_1$-component of the force acting on a dislocation at $x$. 
Similarly to the discrete case, this implies that: 
\begin{quote}
In $\mu(t)$-a.e.\ $x$, \emph{either} $\phi(x,t)=0$, \emph{or} the total force $-[\partial_{x_1}\delta\widetilde{\mathcal F}/{\delta\mu}](\mu(t),t)(x)$ equals $\pm 1$; in the latter case $\phi$ has the same sign as the force.
\end{quote}\medskip

We briefly mention that the evolution systems of this paper, both at finite $n$ and in the limit $n\to\infty$, have formal counterparts as solutions of the so-called Energy-Dissipation Inequality, see e.g.\ \cite{MielkeRossiSavare09,MielkeRossiSavare12a}. However, the approach followed in these works generates a different class of solutions than Definition~\ref{def:quasistat}, since solutions of the Energy-Dissipation Inequality only jump when they lose  their \emph{local} minimality, while solutions of Definition~\ref{def:quasistat} jump as soon as they lose \emph{global} minimality.

\subsection{Discussion}\label{subsec:discussion}

The research of this paper was driven by several aims. To start with, we wanted to make a first step in connecting dynamic discrete-dislocation models with their upscaled dislocation-density counterparts in a mathematically rigorous way; the slip-plane confinement is an important aspect of this. Secondly, we wanted to formulate a clear $\Gamma$-convergence statement for edge dislocations \emph{in terms of dislocation densities} (or measures), which is more convenient for dislocation dynamics than the results of~\cite{GarroniLeoniPonsiglione10}, which are formulated in terms of the matrix field $\beta$. Finally, we wanted to bring the mathematical and the engineering literature on dislocations closer together; one striking difference is the preference in engineering to formulate problems in terms of the interaction energy $V$, while this interaction energy is nearly absent from the mathematical literature. Hence the explicit characterization of the $\Gamma$-limit in terms of $V$ (Theorem~\ref{th:main1-intro}).

These aims suggested a number of differences with the current literature. For instance, we chose to renormalize the elastic energy by subtracting the self-energy, which amounts to disregarding that energy contribution; in the context of the evolution of preserved dislocation numbers this makes sense, and leads to a simpler Gamma-convergence result (one single scaling regime instead of three regimes, as in~\cite{GarroniLeoniPonsiglione10}). We also restrict ourselves to the single-Burgers-vector case, since the treatment of cancellation of dislocations of opposite sign should arise from a separate modelling of the creation and annihilation of dislocation pairs.

The main methodological contributions of this paper lie in the details of the construction of the recovery sequences in the two Gamma-convergence results (Theorems~\ref{th:main1-intro} and~\ref{thm:joint}). Even without the slip-plane restrictions (Theorem~\ref{th:main1-intro}),  avoiding the singularity in the interaction potential requires careful placing of the dislocations. In the case of slip-plane confinement, the restrictions are much more severe, and subtle management of the different scales is required.\medskip

We now comment more in detail on some issues. 

\subsubsection{Conditions on the separation of dislocations} \label{cond-separation}
For Theorem~\ref{th:main1-intro} the sequence $\mu_n$ of empirical measures is required to satisfy two separation requirements: the defects should be separated from each other by at least $r_n\to0$ and from the boundary by at least the fixed distance $\ell>0$. 

The condition \eqref{hyp:enrn} on the distance $r_n$ arises from the difference between $K^n$ and $K$, which in turn arises from the difference between imposing the circulation condition via an integrated boundary condition as in~\eqref{EL-Ke} or a Dirac delta as in~\eqref{curl:K}. The interaction energy difference between the two is bounded by $O(\sqrt{\e_n/r_n^{3}})$, and condition \eqref{hyp:enrn} makes this difference small. Without this separation we believe the limiting energy to be the same, but we have no proof. 

The separation of the dislocations from the boundary $\partial \Omega$ is related to the choice of boundary conditions. With the no-stress boundary conditions that we impose at $\partial\Omega$, dislocations can reduce their energy by moving to the boundary (in essence they `vanish' at the boundary). Taking this possibility into account would require a modification of the renormalization term in~\eqref{def:curly-F-intro}, and we leave this to a future publication. Note that similar conditions are present in other work; for instance, the authors of~\cite{AlicandroDe-LucaGarroniPonsiglione13} prove the liminf inequality only for the case that no dislocations are `lost' in the limit. 

The condition that dislocations can only live in a rectangle $\mathcal R \subset \Omega$ is a simplification that proves to be very useful in the approximation results in Section~\ref{conv:evolution} and in the construction of the recovery sequence in Theorem~\ref{thm:joint}. A crucial point in our constructions is the possibility to spread out the dislocations horizontally and our assumption on $\mathcal R$ guarantees that the modified positions are still in the domain.

In addition, as pointed out before, for the joint recovery sequence in Theorem~\ref{thm:joint} we require the admissible measures 
to satisfy an upper bound on the maximum number of dislocations per slip plane and a lower bound on the minimum distance between slip planes, both in terms of powers of the number $n$ of dislocations (see Definition \ref{def:classYng}). This condition allows concentration on a slip plane or on a vertical line, but does not allow dislocations to be too close both horizontally and vertically at the same time.

\subsubsection{Related Work}

The asymptotic behaviour of the quadratic dislocation energy for edge dislocations was already studied by Garroni, Leoni and Ponsiglione in \cite{GarroniLeoniPonsiglione10}. One of the main difference with our work is that we consider the reduced energy \eqref{def:Fn} instead of \eqref{def:En}, and use as main variable the dislocation density rather than the strain. Moreover, while in \cite{GarroniLeoniPonsiglione10} the authors focus on the self-energy term of \eqref{def:En} in the case of edge dislocations with multiple Burgers vectors, we instead focus on the next term in the expansion, namely the interaction energy, and simplify by restricting to one Burgers vector. A similar analysis as in \cite{GarroniLeoniPonsiglione10} has been done in~\cite{DeLucaGarroniPonsiglione12} without the well-separation assumption that we make (see~\eqref{hyp:enrn}).

The focus on the interaction term of the energy for edge dislocations was already present in the work of Cermelli and Leoni \cite{CermelliLeoni06}. In \cite{CermelliLeoni06} the authors define the renormalised energy starting from a quadratic dislocation energy, and focus on the interaction term of the energy. They however consider the case of a finite number of dislocations with different Burgers vectors, and do not phrase their result in terms of $\Gamma$-convergence, but instead they keep $n$ fixed, and therefore express their renormalised energy in terms of the positions of dislocations, rather than densities.

For screw dislocations the interaction energy was derived from discrete models in \cite{AlicandroDe-LucaGarroniPonsiglione13}. Moreover in \cite{AlicandroDe-LucaGarroniPonsiglione13} the authors also considered the time-dependent case. More precisely they proved the convergence of the discrete gradient flow for the discrete dislocation energy with flat dissipation to the gradient flow of the renormalised energy.

As for the slip-plane-confined motion, the only related work in the mathematical domain that we know of is \cite{BlassFonsecaLeoniMorandotti}, where
the authors consider screw dislocations that may move along a finite set of directions. Such a system presents different mathematical difficulties, since each dislocation can, in theory, reach each point in the plane, in contrast to the single-slip-plane confinement of this paper.

Finally, there is an intriguing question that arises from the comparison with current continuum-scale modelling of plasticity (as in e.g.~\cite{GurtinAnand05,GurtinAnandLele07}). The limiting energy of Theorem~\ref{th:main1-intro} is non-local, with an interaction kernel that has no intrinsic length scale. However, `defect energies' in the continuum-level modelling are usually assumed to be local (see e.g.~\cite[Eq.~(8.8)]{GurtinAnandLele07} or \cite[Eq.~(6.16)]{GurtinAnand05}). It is unclear to us how these two descriptions can be reconciled.

\subsubsection{Extensions and open questions}
Various extensions of the present work would be relatively straightforward. The isotropy of $\C$ is only assumed for convenience, and without this property we expect similar results to hold. The type of loading in the quasi-static evolution (the form of~\eqref{loading-term}) is also chosen for convenience; we expect that other types of loading can be treated with minor changes, although we expect that the elasticity  problem should then be formulated in terms of displacements $u$ rather than elastic strains $\beta$. 

The current restriction on the set of admissible initial data (see Definition~\ref{def:classYng}) is not fully satisfactory. It would be very interesting, and useful, to understand wich class of quasi-static evolutions can be approximated using discrete systems.

In this work we restricted our attention to the single slip case, namely to the case of parallel slip planes, and with no loss of generality we considered $b=e_1$. We moreover assumed that all the dislocations are \textit{positive} dislocations. This further assumption excludes the case of annihilation and the presence of \textit{dipoles}. Natural extensions would be to allow for both positive and negative dislocations, and to consider the multiple Burgers vector case. It would also be very interesting to include creation and annihilation in such a model, allowing the total variation of the density of dislocations to change in time.

Another possible direction of investigation would be to consider different dissipations, e.g. the flat dissipation considered in \cite{AlicandroDe-LucaGarroniPonsiglione13} (but still keeping our slip-confinement condition); another example of this would be considering gradient-flow (quadratic-dissipation) time evolution rather than quasi-static evolution. 

One of the most interesting directions of extension is towards the three-dimensional case, where dislocations are three-dimensional curves, and entanglement between dislocations has a crucial role in evolution. For static Gamma-convergence very first results have recently been proved by Conti, Garroni, and Massaccesi~\cite{ContiGarroniMassaccesi13}, but the mathematical understanding of the three-dimensional situation is still very much in its infancy.

\subsection{Notation}
Here we list some symbols and abbreviations that are going to be used throughout the paper. Since the problem is in the setting of planar elasticity, all functions and 
vectors below are defined on a two-dimensional domain.

\begin{minipage}{15cm}
\begin{small}
\bigskip
\begin{tabular}{lll}
$a^\perp$ & vector orthogonal to $a\in \R^2$ obtained by a counterclockwise rotation by $\pi/2$\\
$\lfloor r\rfloor$ & integer part of $r\in \R$\\
$\skw F$ & skew-symmetric part of a matrix $F$, $\skw F = (F-F^T)/2$\\
$\sym F$ & symmetric part of a matrix $F$, $\sym F = (F+F^T)/2$\\
$Eu$ & symmetric gradient, $Eu = \sym\nabla u$\\
$\mathcal{L}^1$, $\mathcal{L}^2$ & one- and two-dimensional Lebesgue measure\\
$\mathcal{H}^1$ & one-dimensional Hausdorff measure\\
$\pi_i$ & projection onto the coordinate $e_i$\\
$\mathcal{P}(\Omega)$ &  non-negative Borel measures on $\Omega$ of mass $1$ \\
$|\mu|(A)$ & total variation of $\mu\in\mathcal P(\Omega)$ in $A\subseteq \Omega$\\
$f_{\#}\mu$ & push forward of $\mu$ by $f$\\
$\beta_\mu$  & {elastic strain associated with $\mu$} & Sec.~\ref{subsec:aux-functions}\\
$K_z^n, K_z$ & {singular strains associated with $\mu$} & Sec.~\ref{subsec:aux-functions}\\
$u_\mu, v_{\mu}$ & displacements associated with $\mu$ & Sec.~\ref{subsec:aux-functions}\\
$I_{n,\mu}$, $I_{\mu}$ & auxiliary functionals associated with $\mu$ & Sec.~\ref{subsec:aux-functions}\\
$d$ & slip-plane-confined Wasserstein distance & \eqref{def:d}\\
$d_1$ & Monge-Kantorovich distance  with cost $c(x,y) = |x-y|$  & \cite{Villani03}\\
$\Gamma(\mu,\nu)$   & set of couplings of $\mu$ and $\nu$ with only horizontal transport & \eqref{def:Gamma}\\
$\Gamma_1(\mu,\nu)$ & set of couplings of $\mu$ and $\nu$\hphantom{obtained by a counterclockwise rotation by $\pi/2$}& \cite{Villani03}\\
$\mathcal F_n$ & renormalized energy & \eqref{renorm:F}\\
$\widetilde{\mathcal F}_n$ & total energy & \eqref{def:tildeFn}\\
$X_n$ & set of admissible discrete dislocation measures & \eqref{def:Xn}\\
$Y_n(\gamma,c)$ & set of admissible recovery measures & \eqref{def:Yng}\\
$\mathcal{P}_{\gamma,c}^\infty(\Omega)$ & set of admissible limit measures & Lemma~\ref{lemma:char-Pinfty}\\
\end{tabular}
\end{small}
\end{minipage}

\section{Definitions and preliminaries}

In the following $\Omega\subset\R^2$ is a simply connected domain with Lipschitz boundary.  
We moreover assume that it contains a closed rectangle $\mathcal R$ whose sides are parallel to the coordinate axes and with $\dist(\mathcal R,\partial \Omega)= \ell>0$.

Before turning to the main theorems of this paper, starting from Section~\ref{sec:Gamma-conv-1}, we introduce five auxiliary functions (Section~\ref{subsec:aux-functions}), and use these to rewrite the renormalized energy (Section~\ref{subsec:renormalizedEnergy}).

\subsection{Auxiliary functions}
\label{subsec:aux-functions}
In the arguments of this paper, five auxiliary functions play an important role: $K$, $K^n$, $\beta_\mu$, $\umu$, and $v_\mu$. We now introduce these. For the notation we follow~\cite{CermelliLeoni06}, and we set the Burgers vector $b=e_1$.\medskip

\textbf{The matrix-valued function \boldmath$K$} is defined as
$$
K(x;z):=\frac{1}{2\pi |x-z|^2} e_1\otimes (x-z)^\perp +\nabla v(x-z) \quad
\text{for every }x\neq z,
$$
where 
$$
v(x):= -\frac{\mu\log|x|}{2\pi(\lambda+2\mu)} e_2 -\frac{\lambda+\mu}{4\pi(\lambda+2\mu)|x|^2}
\big[ (e_1\cdot x^\perp) x + (e_1\cdot x) x^\perp\big]
$$
and $\lambda, \mu$ are the Lam\'e constants of the stress-strain tensor $\C$.
The function $K(\cdot;z)$ is the strain field in $\R^2$ generated by a single dislocation at $z$, with Burgers vector $e_1$, and is a distributional solution of 
\begin{equation}\label{curl:K}
\begin{cases}
\textrm{div} \,\mathbb{C}K(\cdot;z)=0 \quad &\textrm{in } \R^2,\\
\textrm{Curl}\,K(\cdot;z) = e_1\delta_z \quad &\textrm{in } \R^2.
\end{cases}
\end{equation}
For brevity we often write $K_z$ for $K(\cdot;z)$.\medskip

\textbf{The matrix-valued function \boldmath$K^n$} is a perturbation of $K$, defined by
\begin{equation}\label{def:Ke}
K^n(x;z):= K(x;z)+\e_n^2\nabla w(x-z),
\end{equation}
with
$$
w(x):=\frac{(\lambda+\mu)}{2\pi(\lambda+2\mu)|x|^4}\big[ (e_1\cdot x^\perp) x + (e_1\cdot x) x^\perp\big].
$$
As before, we will often write $K^n_z(\cdot)$ and $w_z(\cdot)$ instead of $K^n(\cdot;z)$ and $w(\cdot-z)$.
The function $K^n$ is also the strain generated by a dislocation at $z$, but in the context of the exterior domain $\R^2\setminus B_{\e_n}(z)$; the mismatch associated with the dislocation is enforced by a combination of a stress-free condition and a circulation condition on the boundary:
\begin{equation}\label{EL-Ke}
\begin{cases}
{\rm div}\,\C K^n(\cdot\,; z) = 0 & \text{ in } \R^2\setminus B_{\e_n}(z), \\
\C K^n(\cdot\,;z)\nu =0  &\text{ on }\partial B_{\e_n}(z),\\
\displaystyle\int_{\partial B_{\e_n}(z)}K^n(x;z)\, \tau(x)\, d\HH^1(x)=e_1.
\end{cases}
\end{equation}
The functions $K$ and $K^n$ have the same far-field behaviour, and $K^n$ converges pointwise to $K$ as $\e_n\to0$.\medskip

For given $\mu\in X_n$, \textbf{the matrix-valued function \boldmath$\beta_\mu$} is the unique (up to skew-symmetric matrices) minimizer of the energy $E_n$ in~\eqref{def:En} subject to the circulation constraint~\eqref{def:Aemu}; it is therefore the strain field generated by the dislocations that $\mu$ represents.
It satisfies the equations
\begin{equation}\label{EL-Ge}
\begin{cases}
{\rm div}\,\C \Gmu = 0 & \text{ in } \Omega_n(\mu),\\
\C \Gmu\nu = 0 \quad &\text{ on } \partial\Omega_n(\mu).\\
\end{cases}
\end{equation}
\medskip

For given $\mu\in X_n$, \textbf{the vector-valued function \boldmath$\umu$} is a displacement generated by the dislocations, as follows: since the left-hand side of
\begin{equation}
\label{Gmu-dec}
\Gmu(x)- \frac1n \sum_{i=1}^{\n} K^n(x,z_i) =\nabla \umu(x) \qquad \text{for } x\in \Omega_n(\mu)
\end{equation}
is curl-free in $\Omega_n(\mu)$, with zero-circulation boundary conditions on each of the $\partial B_{\e_n}(z_i)$, it is the gradient of a function $\umu$. It can be interpreted as a corrector displacement field, which cancels the non-stress-free boundary values of $\tfrac1n\sum_i K_{z_i}^n$.

We choose a fixed ball $B$ contained in $\Omega$ such that $\dist(x,\partial\Omega)<\ell/2$
for every $x\in B$ (and thus, contained in $\Omega_n(\mu)$). We will require that 
\begin{equation}\label{normal}
\int_B \umu(x)\, dx =0, \quad \int_B {\rm skew}\, \nabla\umu(x)\, dx=0.
\end{equation}

The function $u_{\mu}$ can be alternatively characterised in terms of a minimisation problem for the functional $I_{n,\mu}$ defined as  
\begin{equation}
\label{def:Inmu}
I_{n,\mu}(u) := \frac12 \int_{\Omega_n(\mu)}\C\nabla u(x):\nabla u(x)\, dx
+ \frac1n \sum_{i=1}^{\n}\int_{\partial\Omega_n(\mu)}\C K^n(x;z_i)\nu(x)\cdot u(x)\, d\HH^1(x).
\end{equation}
This will be proved in Lemma~\ref{lemma:u-est} at the end of this section.\medskip

\textbf{The vector-valued function \boldmath$v_\mu$} is very similar to $\umu$: for given $\mu\in\P(\Omega)$, it is the unique minimizer of 
\begin{equation}
\label{def:Imu}
I_\mu(v):= \frac12\int_\Omega \C\nabla v(x):\nabla v(x)\,dx+ \int_\Omega\int_{\partial\Omega} \C K(x;y)\nu(x)\cdot v(x)\,d\HH^1(x)\,d\mu(y)
\end{equation}
on the class 
\begin{equation}
\label{def:Iclass}
\Big\{v\in H^1(\Omega;\R^2): \ \int_B v\, dx =0, \ \int_B {\rm skew}\, \nabla v\, dx=0
\Big\}.
\end{equation}
Note the similarity between $I_{n,\mu}$, which defines $\umu$, and $I_\mu$: if we  take $\mu = \frac1n \sum_i \delta_{z_i}$, then  $I_{n,\mu}$ and $I_\mu$ are very similar, differing only in the distinction between $K^n$ and $K$ and in the domain of integration. Indeed we will see below (proof of Theorem~\ref{ups:density}) that if $\mu_n\weakto\mu$, then $u_{\mu_n}$ converges to $v_\mu$ and $I_{n,\mu_n}(u_{\mu_n})\to I_\mu(v_\mu)$.

\subsection{Rewriting the energy}
\label{subsec:renormalizedEnergy}

Recall that the energies $F_n$ and $\mathcal F_n$ are defined in~\eqref{def:Fn} and~\eqref{def:curly-F-intro} in terms of the minimiser $\beta_\mu$ of the energy $E_n$ in~\eqref{def:En} subject to the circulation constraint~\eqref{def:Aemu}. 

Using~\eqref{Gmu-dec} we rewrite $F_n(\mu)$ for $\mu\in X_n$ as 
\begin{align}
F_{n}(\mu) =  & \, \frac12\int_{\Omega_n(\mu)}\C \Gmu:\Gmu\,dx 
\nonumber
\\
 = & \,\frac1{2n^2}\sum_{i=1}^{\n}\int_{\Omega_n(\mu)}\C K^n(x;z_i):K^n(x;z_i)\, dx + \frac1{2n^2} \sum_{i=1}^{\n}\sum_{j\neq i}
\int_{\Omega_n(\mu)}\C K^n(x;z_i):K^n(x;z_j)\, dx
\nonumber
\\
&\,+ \frac1n \sum_{i=1}^{\n}\int_{\Omega_n(\mu)}\C K^n(x;z_i):\nabla \umu(x)\, dx
+\frac12\int_{\Omega_n(\mu)}\C \nabla \umu(x):\nabla \umu(x)\, dx.
\label{energy-dec}
\end{align}
Note that the first term in the above expression represents the \textit{self-energy} contributions of the dislocations, while the second term 
represents the \textit{mesoscopic} pairwise interaction energy between two dislocations located at $z_i$ and $z_j$. 
Concerning the last two terms, integrating by parts and applying~\eqref{EL-Ke} yield
\begin{multline*}
\frac1n\sum_{i=1}^{\n}\int_{\Omega_n(\mu)}\C K^n(x;z_i):\nabla \umu(x)\, dx
\\
= \frac1n\sum_{i=1}^{\n}\int_{\partial\Omega}\C K^n(x;z_i)\nu\cdot \umu(x)\, d\HH^1(x)
- \frac1n\sum_{i=1}^{\n}\sum_{j\neq i}\int_{\partial B_{\e_n}(z_j)}\C K^n(x;z_i)\nu\cdot \umu(x)\, d\HH^1(x).
\end{multline*}
Similarly, using \eqref{EL-Ke}--\eqref{Gmu-dec} we obtain
\begin{multline}
\label{Ke-bdary2}
\int_{\Omega_n(\mu)}\C \nabla \umu(x):\nabla \umu(x)\, dx
\\
=-\frac1n\sum_{i=1}^{\n}\int_{\partial\Omega}\C K^n(x;z_i)\nu\cdot \umu(x)\, d\HH^1(x)
+\frac1n\sum_{i=1}^{\n}\sum_{j\neq i}\int_{\partial B_{\e_n}(z_j)}\C K^n(x;z_i)\nu\cdot \umu(x)\, d\HH^1(x).
\end{multline}
By these computations we have that for $\mu\in X_n$
\begin{align}
\nonumber
{\mathcal F}_n(\mu)&= \frac1{2n^2} \sum_{i=1}^{\n}\sum_{j\neq i}
\int_{\Omega_n(\mu)}\C K^n_{z_i}:K^n_{z_j}\\
&\qquad {}+\frac1{2n} \sum_{i=1}^{\n}\int_{\partial\Omega}\C K^n_{z_i}\nu\cdot \umu\, d\HH^1
-\frac1{2n} \sum_{i=1}^{\n}\sum_{j\neq i}\int_{\partial B_{\e_n}(z_j)}\C K^n_{z_i}\nu\cdot \umu\, d\HH^1\label{renorm:F-}\\
&= \frac1{2n^2} \sum_{i=1}^{\n}\sum_{j\neq i}\int_{\Omega_n(\mu)}\C K^n_{z_i}:K^n_{z_j} 
 +\frac1{2n} \sum_{i=1}^{\n}\int_{\partial\Omega_n(\mu)}\C K^n_{z_i}\nu\cdot \umu\, d\HH^1,
\label{renorm:F}
\end{align}
where in the last step we used the boundary condition~\eqref{EL-Ke} to add the missing term in the double sum.

The two terms above will result in the two terms in the limiting energy $\mathcal F$ (see~\eqref{def:limitF}). To recognize the second term, note that by the same calculation~\eqref{Ke-bdary2}, the second term is equal to $I_{n,\mu }(u_\mu)$, which will converge in the limit to $I_\mu(v_\mu)$.\medskip

The energy decomposition above proves useful for the characterisation of $u_\mu$ in the following lemma.

\begin{lemma}\label{lemma:u-est}
Assume \eqref{hyp:enrn}
and let $\mu=\frac1{\n} \sum_{i=1}^{\n} \delta_{z_i}\in X_n$ for some $n\in\N$. Then there exists a unique $\umu\in H^1(\Omega_n(\mu);\R^2)$ that satisfies \eqref{Gmu-dec} and \eqref{normal}. Moreover, there exist a constant $C>0$, independent of $n$, $\mu$, and $\umu$, and an extension $\tilde u_\mu\in H^1(\Omega;\R^2)$ of $\umu$
such that 
\begin{equation}\label{ue-nb}
\|\tilde u_\mu\|_{H^1(\Omega)}\leq C \|E\umu\|_{L^2(\Omega_n(\mu))}\leq C.
\end{equation}
In addition, the function $\umu$ minimizes the functional $I_{n,\mu}$ defined in \eqref{def:Inmu} on $H^1(\Omega_n(\mu);\R^2)$.
\end{lemma}

\begin{proof}
As mentioned before, the curl-free property of the left-hand side of~\eqref{Gmu-dec} implies that there exists  some $\umu\in H^1(\Omega_n(\mu);\R^2)$ satisfying~\eqref{Gmu-dec}. Since the minimizer of $E_n(\mu,\cdot)$ in $\mathcal A_n(\mu)$
is unique up to addition of skew-symmetric matrices, condition \eqref{normal} guarantees the uniqueness of~$\umu$.

We now prove \eqref{ue-nb}.
For every $u\in H^1(\Omega_n(\mu);\R^2)$ the function
$$
\beta(x):=\chi_{\Omega_n(\mu)}(x)\Big(\frac1n \sum_{i=1}^{\n} K^n(x;z_i)+\nabla u(x)\Big)
$$
belongs to ${\mathcal A}_n(\mu)$. Thus, by \eqref{def:Fn} we have
$$
F_n(\mu)\leq E_n(\mu,\beta)=\frac12 \int_{\Omega_n(\mu)}\C\beta:\beta\, dx.
$$
Developing the quadratic form on the right-hand side and taking into account the decomposition \eqref{energy-dec} and 
the first equation in \eqref{EL-Ke}, one easily deduces that the function $\umu$
minimizes the functional $I_{n,\mu}$ on $H^1(\Omega_n(\mu);\R^2)$.
In particular, considering $u\equiv0$ as a competitor and using the coercivity of $\C$
on symmetric matrices, we deduce
\begin{multline}\label{ue-bound1}
C_1\| E\umu\|^2_{L^2(\Omega_n(\mu))}
\\
\leq \frac1n \sum_{i=1}^{\n} \|\C K^n(\cdot\,;z_i)\|_{L^2(\partial\Omega)}\|\umu\|_{L^2(\partial\Omega)}
+ \frac1n \sum_{i=1}^{\n}\sum_{j\neq i}
\|\C K^n(\cdot\,;z_i)\|_{L^2(\partial B_{\e_n}(z_j))}\|\umu\|_{L^2(\partial B_{\e_n}(z_j))}.
\end{multline}
From the definition \eqref{def:Ke} of $K^n$, the definition \eqref{def:Xn} of the admissible class, and assumption \eqref{hyp:enrn} it follows that
$$
\sup_{x\in\partial\Omega}|\C K^n(x;z_i)|\leq C\sup_{x\in\partial\Omega} \Big(\frac{1}{|x-z_i|}
+\e_n^2 \frac{1}{|x-z_i|^3}\Big) \leq C\Big(\frac1\ell+\frac{\e_n^2}{\ell^3}\Big)\leq C,
$$
and, analogously, for $i\neq j$,
\begin{equation}\label{supKe-Be}
\sup_{x\in\partial B_{\e_n}(z_j)}|\C K^n(x;z_i)|
\leq C\sup_{x\in\partial B_{\e_n}(z_j)} \Big(\frac{1}{|x-z_i|}
+\e_n^2 \frac{1}{|x-z_i|^3}\Big)
\leq C\Big(\frac{1}{r_n}+\frac{\e_n^2}{r_n^3}\Big)\leq \frac{C}{r_n}.
\end{equation}
Combining the two inequalities above with \eqref{ue-bound1}, we have
\begin{equation}\label{ue-bound2}
\| E\umu\|^2_{L^2(\Omega_n(\mu))}
\leq C  \|\umu\|_{L^2(\partial\Omega)}
+ C\frac{\sqrt{\e_n} }{r_n}\sum_{j=1}^{\n}
\|\umu\|_{L^2(\partial B_{\e_n}(z_j))}.
\end{equation}

To control the norms of $\umu$ on the right-hand side, we proceed as follows.
By Theorem~\ref{thm:extension} in the Appendix we can extend $\umu$ to a function $\tilde u_\mu\in H^1(\Omega;\R^2)$
such that
$$
\|E\tilde u_\mu\|_{L^2(\Omega)}\leq C\|Eu_\mu\|_{L^2(\Omega_n(\mu))},
$$
where $C$ is a constant independent of $n$, $\mu$, and $\umu$. Moreover, since $\tilde u_\mu=\umu$ on the ball $B$ where
\eqref{normal} is satisfied, by the Korn and Poincar\'e inequalities we have
\begin{equation}\label{KP-utilde}
\|\tilde u_\mu\|_{H^1(\Omega)}\leq C \|E\tilde u_\mu\|_{L^2(\Omega)}\leq C\|E\umu\|_{L^2(\Omega_n(\mu))},
\end{equation}
and by the continuity of the trace operator on $\partial\Omega$
\begin{equation}\label{trace-utilde}
\|\umu\|_{L^2(\partial\Omega)}= \|\tilde u_\mu\|_{L^2(\partial\Omega)}\leq C
\|\tilde u_\mu\|_{H^1(\Omega)}\leq C\|E\umu\|_{L^2(\Omega_n(\mu))}.
\end{equation}

To estimate the norm of $\umu$ on the boundaries $\partial B_{\e_n}(z_j)$ we use polar coordinates centered
at the point $z_j$ on the annulus $B_{r_n}(z_j)\setminus B_{\e_n}(z_j)$.
For every $\theta\in[0,2\pi]$ and every $\e_n\leq s\leq r\leq r_n$ we may write
$$
\umu(s,\theta)=\umu(r,\theta)-\int_s^r\frac{\partial \umu}{\partial\rho}(\rho,\theta)\,d\rho.
$$
Thus,
$$
|\umu(s,\theta)|^2\leq 2|\umu(r,\theta)|^2+ 2r_n\int_s^{r_n}\Big|\frac{\partial \umu}{\partial\rho}(\rho,\theta)\Big|^2d\rho.
$$
Integrating over $[0,2\pi]$ yields
$$
\int_0^{2\pi} \e_n |\umu(s,\theta)|^2\, d\theta
\leq 2\int_0^{2\pi}r|\umu(r,\theta)|^2\, d\theta+ 2 r_n\int_{B_{r_n}(z_j)\setminus B_{\e_n}(z_j)}
|\nabla \umu|^2\, dx.
$$
We now integrate over $r\in[\e_n,r_n]$ and divide by $r_n-\e_n$:
$$
\int_{\partial B_{\e_n}(z_j)} |\umu|^2\, d\HH^1
\leq \frac{2}{r_n-\e_n}\int_{B_{r_n}(z_j)\setminus B_{\e_n}(z_j)}
|\umu|^2\, dx+ 2r_n\int_{B_{r_n}(z_j)\setminus B_{\e_n}(z_j)}
|\nabla \umu|^2\, dx.
$$
Hence, by \eqref{KP-utilde} we conclude that
\begin{equation}\label{tr-Be}
\sum_{j=1}^{\n}\| \umu\|_{L^2(\partial B_{\e_n}(z_j))} 
\leq \frac{C}{\sqrt{r_n}} \|\umu\|_{H^1(\Omega_n(\mu))}\leq
\frac{C}{\sqrt{r_n}} \|\tilde u_\mu\|_{H^1(\Omega)}
\leq  \frac{C}{\sqrt{r_n}} \|E \umu\|_{L^2(\Omega_n(\mu))}.
\end{equation}

Combining \eqref{ue-bound2}, \eqref{trace-utilde}, and  \eqref{tr-Be}, and taking into account \eqref{hyp:enrn}, we obtain
$$
\| E \umu\|^2_{L^2(\Omega_n(\mu))}
\leq C \|E\umu\|_{L^2(\Omega_n(\mu))}
+ \frac{C\sqrt\e_n }{r_n^{3/2}}\|E\umu\|_{L^2(\Omega_n(\mu))}
\leq C \|E\umu\|_{L^2(\Omega_n(\mu))}.
$$
Therefore, using also \eqref{KP-utilde}, we deduce \eqref{ue-nb}.
\end{proof}


\section{$\Gamma$-convergence of the renormalized energy}
\label{sec:Gamma-conv-1}
In this section we prove Theorem~\ref{th:main1-intro}, the Gamma-convergence of the renormalised energy defined in \eqref{def:curly-F-intro}. 
Before stating and proving the main theorem we state some useful properties of the interaction potential $V$ appearing in the definition of the limit energy.

\begin{lemma}\label{lemma:Vcont}
The function $V:\Omega\times\Omega\to\R\cup\{+\infty\}$ 
defined as 
\begin{equation}\label{def:V}
V(y,z):=\begin{cases}
\displaystyle \int_\Omega \C K(x;y):K(x;z)\, dx & \text{ if } y\neq z,
\smallskip
\\
+\infty & \text{ if } y=z,
\end{cases}
\end{equation}
is well defined, symmetric, and continuous. Moreover, we have the following estimates:
\begin{itemize}
\item there exist two constants $C>0$ and $L>0$ such that for every $y,z\in\Omega$ with $y\neq z$
\begin{equation}\label{imp-est}
|V(y,z)|\leq C \Big(1 - \log \frac{|y-z|}{L}\Big);
\end{equation}
\item for every open set $\Omega'$ compactly contained in $\Omega$ there exist $C(\Omega')>0$ and $\overline R>0$ such that for every $y,z\in\Omega'$ with $0<|y-z|\leq \overline R$
\begin{equation}\label{below-est}
C(\Omega')(1-\log|y-z|)\leq V(y,z).
\end{equation}
\end{itemize}
\end{lemma}

\begin{proof}
From the definition of $K$ it follows that there exists a constant $C>0$ such that
\begin{equation}\label{usual}
|K(x;y)|\leq \frac{C}{|x-y|}
\end{equation}
for every $x\neq y$. This implies that the integral in the definition of $V(y,z)$ for $y\neq z$ is finite.
Clearly $V$ is symmetric, that is, $V(y,z)=V(z,y)$.

We now prove \eqref{imp-est}. Let $y,z\in\Omega$ with $y\neq z$. From the definition of $V$ and from \eqref{usual} we have
\begin{equation}\label{b_V}
|V(y,z) |\leq C\int_{\Omega}\frac{dx}{|x-y||x-z|} \leq C\int_{B_L(z)}\frac{dx}{|x-y||x-z|},
\end{equation}
because,  $\Omega$ being bounded, $\Omega\subset B_L(z)$ for some finite $L>0$ and every $z\in\Omega$. With no loss of generality we can assume that $L\geq 2|y-z|$. For convenience we set $\tilde{x}:=x-z$ and $\tilde{y}:=y-z$; then for the last integral in \eqref{b_V} we have
\begin{equation}\label{b_V1}
\int_{B_L(z)}\frac{dx}{|x-y||x-z|} = \int_{B_L(0)}\frac{d\tilde{x}}{|\tilde{x}||\tilde{x}-\tilde{y}|} = \int_{B_{2|\tilde{y}|}(0)}\frac{d\tilde{x}}{|\tilde{x}||\tilde{x}-\tilde{y}|} + \int_{B_L(0)\setminus B_{2|\tilde{y}|}(0)}\frac{d\tilde{x}}{|\tilde{x}||\tilde{x}-\tilde{y}|}.
\end{equation}
If $\tilde{x}\in B_L(0)\setminus B_{2|\tilde{y}|}(0)$, then $|\tilde{x}-\tilde{y}|\geq ||\tilde{x}|-|\tilde{y}|| = |\tilde{x}|-|\tilde{y}|$; hence
\begin{align}\label{b_V2}
\int_{B_L(0)\setminus B_{2|\tilde{y}|}(0)}\frac{d\tilde{x}}{|\tilde{x}||\tilde{x}-\tilde{y}|} &\leq \int_{B_L(0)\setminus B_{2|\tilde{y}|}(0)}\frac{d\tilde{x}}{|\tilde{x}|\big(|\tilde{x}|-|\tilde{y}|\big)} = 2\pi \int_{2|\tilde{y}|}^L \frac{dr}{r-|\tilde{y}|}\nonumber\\
& = 2\pi \log \left(\frac{L}{|\tilde{y}|}-1\right) \leq -2\pi \log \left(\frac{|\tilde{y}|}{L}\right).
\end{align}
We now split the integral on $B_{2|\tilde{y}|}(0)$ in \eqref{b_V1} into three terms:
\begin{align}\label{b_V3}
 \int_{B_{2|\tilde{y}|}(0)}\frac{d\tilde{x}}{|\tilde{x}||\tilde{x}-\tilde{y}|} =  \bigg\{\int_{B_{|\tilde{y}|/2}(0)} + \int_{B_{|\tilde{y}|/2}(\tilde{y})} +
 \int_{B_{2|\tilde{y}|}(0)\setminus\big(B_{|\tilde{y}|/2}(0)\cup B_{|\tilde{y}|/2}(\tilde{y})\big)}\bigg\}\frac{d\tilde{x}}{|\tilde{x}||\tilde{x}-\tilde{y}|}.
\end{align}
For the first integral in the right-hand side of \eqref{b_V3} we have
\begin{equation*}
\int_{B_{|\tilde{y}|/2}(0)} \frac{d\tilde{x}}{|\tilde{x}||\tilde{x}-\tilde{y}|} \leq \frac{2}{|\tilde{y}|} \int_{B_{|\tilde{y}|/2}(0)} \frac{d\tilde{x}}{|\tilde{x}|} = \frac{4\pi}{|\tilde{y}|}\int_{0}^{|\tilde{y}|/2}dr = 2\pi,
\end{equation*}
since $|\tilde{x}-\tilde{y}| \geq \frac{|\tilde{y}|}{2}$ for $\tilde{x} \in B_{|\tilde{y}|/2}(0)$. Analogously,
\begin{equation*}
\int_{B_{|\tilde{y}|/2}(\tilde y)} \frac{d\tilde{x}}{|\tilde{x}||\tilde{x}-\tilde{y}|} \leq  2\pi.
\end{equation*}
Moreover, for the last integral in the right-hand side of \eqref{b_V3}  we have
\begin{equation*}
 \int_{B_{2|\tilde{y}|}(0)\setminus\big(B_{|\tilde{y}|/2}(0)\cup B_{|\tilde{y}|/2}(\tilde{y})\big)}\frac{d\tilde{x}}{|\tilde{x}||\tilde{x}-\tilde{y}|} \leq \frac{4}{|\tilde{y}|^2}\int_{B_{2|\tilde{y}|}(0)}d\tilde{x} = 16\pi.
\end{equation*}
In conclusion, from \eqref{b_V3}, we deduce
$$
\int_{B_{2|\tilde{y}|}(0)}\frac{d\tilde{x}}{|\tilde{x}||\tilde{x}-\tilde{y}|} \leq 20\pi.
$$
Finally, putting together \eqref{b_V}--\eqref{b_V3}, and substituting back $\tilde{y} = y-z$ we obtain
$$
|V(y,z)| \leq C\left(-2\pi \log \left(\frac{|y-z|}{L}\right) + 20\pi \right),
$$
which implies \eqref{imp-est}.

We now prove \eqref{below-est}. Let $\Omega'$ be an open set contained in $\Omega$ with $\delta:=\dist(\Omega',\partial\Omega)>0$
and let $y,z\in\Omega'$ with $y\neq z$. Here we use an argument by \cite{CermelliLeoni06}.  Let $\gamma_{y,z}$ be a line segment parallel to $z-y$ connecting $z$ to $\partial\Omega$, so that
$$
\gamma_{y,z}=\Big\{x\in\Omega: \ x=z+s\,\frac{z-y}{|z-y|}, \ s\in[0,\bar s_{y,z}]\Big\}.
$$
Let also $m_{y,z}$ be the unit normal vector to $\gamma_{y,z}$.
Since $\Omega\setminus\gamma_{y,z}$ is simply connected and ${\rm Curl}\,K(\cdot\,;z)=0$ in $\Omega\setminus\gamma_{y,z}$, there exists a function $v_{y,z}$ in $\Omega\setminus\gamma_{y,z}$ such that $K(\cdot\,;z)=\nabla v_{y,z}$ in $\Omega\setminus\gamma_{y,z}$ and the jump of $v_{y,z}$ across $\gamma_{y,z}$ satisfies
$[v_{y,z}]=-e_1$. We can also assume that $v_{y,z}$ has zero mean value on $\Omega\setminus\gamma_{y,z}$. By the divergence theorem we obtain
\begin{align}
V(y,z) & =  \int_{\Omega\setminus\gamma_{y,z}} \C K(x;y):K(x;z)\, dx
\nonumber
\\
& =  \int_{\Omega\setminus\gamma_{y,z}} \C K(x;y):\nabla v_{y,z}(x)\, dx
\nonumber
\\
& =  \int_{\partial\Omega} \C K(x;y)\nu(x)\cdot v_{y,z}(x)\, d\HH^1(x)
- \int_{\gamma_{y,z}} \C K(x;y)m_{y,z}(x)\cdot [v_{y,z}(x)]\, d\HH^1(x).
\label{above}
\end{align} 
We first observe that the first integral on the right-hand side is uniformly bounded
with respect to $y$ and $z$. Indeed, by \eqref{usual}
\begin{equation}\label{deltabound}
|\C K(x;y)\nu(x)|\leq \frac{C}{|x-y|}\leq \frac C\delta
\end{equation}
for every $x\in\partial\Omega$ and every $y\in\Omega'$.
Moreover, we have
\begin{eqnarray*}
\int_{\Omega\setminus\gamma_{y,z}} |K(x;z)|\, dx & \leq & C \int_\Omega \frac{1}{|x-z|}\, dx
\\
& = & C\int_{B_{\delta/2}(z)} \frac{1}{|x-z|}\, dx + C\int_{\Omega\setminus B_{\delta/2}(z)} \frac{1}{|x-z|}\, dx
\\
& \leq & \pi C\delta+\frac{2C}{\delta}.
\end{eqnarray*}
Therefore, the measures $Dv_{y,z}= K(\cdot;z){\mathcal L}^2-e_1\otimes m_{y,z}\HH^1\llcorner \gamma_{y,z}$ are uniformly bounded in $\Omega$.
By Poincar\'e-Wirtinger inequality we deduce that the functions $v_{y,z}$ are uniformly bounded
with respect to $y$ and $z$ in the $BV$ norm, hence their traces are bounded in 
$L^1(\partial\Omega;\R^2)$. In conclusion, by the previous estimate and by \eqref{deltabound} follows that the first term in \eqref{above} is bounded.

We focus on the second integral on the right-hand side of \eqref{above}. We note that for $x=z+s\frac{z-y}{|z-y|}\in\gamma_{y,z}$ we have
$$
-\C K(x;y)m_{y,z}(x)\cdot [v_{y,z}(x)]= \C K(x;y)m_{y,z}(x)\cdot e_1=
\frac{\mu(\lambda+\mu)}{\pi(\lambda+2\mu)}\frac{1}{|z-y|+s},
$$
therefore
$$
- \int_{\gamma_{y,z}} \C K(x;y)m_{y,z}(x)\cdot [v_{y,z}(x)]\, d\HH^1(x)
= \frac{\mu(\lambda+\mu)}{\pi(\lambda+2\mu)}\Big(-\log|y-z|+\log(|y-z|+\bar s_{y,z})\Big).
$$
Since $\bar s_{y,z}\geq \dist(z,\partial\Omega)\geq \delta$, we have
$$
\log(|y-z|+\bar s_{y,z})\geq \log \delta.
$$
Combining \eqref{above} with the previous estimates, we conclude that
$$
V(y,z)\geq -C -\frac{\mu(\lambda+\mu)}{\pi(\lambda+2\mu)}\log|y-z|\geq C'(1- \log|y-z|),
$$
where the last inequality is satisfied for $|y-z|<\overline R$, with $\overline R$ small enough.
This shows \eqref{below-est}.

Finally, we prove the continuity of $V$.
Let $y,z\in\Omega$ with $y\neq z$ and let $(y_n)$, $(z_n)$ be two sequences in $\Omega$ converging to $y$ and $z$, respectively.
For $n$ large enough we clearly have $y_n\neq z_n$. It is easy to see that 
$$
\C K(x;y_n):K(x,z_n) \to \C K(x;y): K(x;z) \quad \text{ for a.e.\ } x\in\Omega.
$$
Moreover, the sequence is dominated. Indeed, for $0<\delta<\frac14 |y-z|$ we have
$$
\chi_{\Omega\setminus(B_\delta(y_n)\cup B_\delta(z_n))}|\C K(x;y_n):K(x,z_n)|\leq \frac{C}{\delta^2},
$$
while
\begin{eqnarray*}
\chi_{B_\delta(y_n)\cup B_\delta(z_n)}|\C K(x;y_n):K(x,z_n)|
& \leq &
\chi_{B_\delta(y_n)\cup B_\delta(z_n)}
\frac{C}{|x-y_n|\, |x-z_n|}
\\
&\leq &
\frac{4C}{|y-z|}\Big( \frac{1}{|x-y_n|} +\frac{1}{|x-z_n|} \Big)
\end{eqnarray*}
and the function on the right-hand side is dominated since it is strongly converging in $L^1$
(by the continuity property in $L^1$). Thus, by the dominated convergence theorem we have
\begin{eqnarray*}
\lim_{(y_n,z_n)\to (y,z)}  V(y_n,z_n) & = &
\lim_{(y_n,z_n)\to (y,z)} \int_\Omega \C K(x;y_n):K(x;z_n)\, dx
\\
& = & \int_\Omega \C K(x;y):K(x;z)\, dx
=V(y,z).
\end{eqnarray*}
 
Let now $y=z\in\Omega$ and let $(y_n)$, $(z_n)$ be two sequences converging to $y$. 
Let $\Omega'$ be an open set compactly contained in $\Omega$ such that $y\in\Omega'$.
Without loss of generality we can assume that $y_n\neq z_n$ and $y_n,z_n\in\Omega'$ for every $n$.
Thus, we can use \eqref{below-est} for $n$ large enough and deduce that 
$$
\lim_{(y_n,z_n)\to (y,y)} V(y_n,z_n) =+\infty=V(y,y).
$$
This concludes the proof of the continuity of $V$. 
\end{proof}

\begin{remark}\label{rmk:bb}
Note that from \eqref{imp-est} and \eqref{below-est} follows that for every open set $\Omega'$ compactly contained in $\Omega$
there exists a constant $C_1(\Omega')>0$ such that for every $y,z\in\Omega'$
\begin{equation}\label{VboundC}
V(y,z) \geq -C_1(\Omega'). 
\end{equation}
\end{remark}

We are now ready to prove the $\Gamma$-convergence of the renormalised energy, which constitutes the main result of this section.

\begin{theorem}\label{ups:density}
Assume \eqref{hyp:enrn}. Then the functionals ${\mathcal F}_n$ $\Gamma$-converge with respect to the narrow convergence of measures,
as $n\to\infty$,  to the functional
${\mathcal F}: {\mathcal P}(\Omega)\to\R\cup\{+\infty\}$ defined as
$$
{\mathcal F}(\mu):=\frac12 \iint_{\Omega\times\Omega} V(y,z)\,d\mu(y)\,d\mu(z)
+ \min_v I_\mu(v),
$$
if $\supp\,\mu \subset \mathcal R$, while ${\mathcal F}(\mu):=+\infty$ otherwise. 
In the formula above $V$ is the function defined in \eqref{def:V} and $I_\mu$ is the functional defined in~\eqref{def:Imu}, 
$$
I_\mu(v):= \frac12\int_\Omega \C\nabla v:\nabla v\,dx+ \int_\Omega\int_{\partial\Omega} \C K(x;y)\nu(x)\cdot v(x)\,d\HH^1(x)\,d\mu(y)
$$
on the class 
\begin{equation}\label{v-class}
\Big\{ v\in H^1(\Omega;\R^2): \ \int_B v\, dx =0, \ \int_B \skw\nabla v\, dx=0
\Big\}.
\end{equation}
\end{theorem}

\begin{remark}
Since $\supp\mu\subset{\mathcal R}$ and $V$ is bounded from below on open sets compactly contained in $\Omega$
by Remark~\ref{rmk:bb}, the interaction energy term in $\mathcal F$ is always well defined, possibly equal to $+\infty$.
\end{remark}

\begin{proof}[Proof of Theorem~\ref{ups:density}]
{\bf Preliminary estimates.}
Let $(\mu_n) \subset X_n$ for every $n\in \mathbb{N}$, and let $(z^n_i)_{i=1,\dots,n}\subset \mathcal{R}$ be such that
$$
\mu_n = \frac1n\sum_{i=1}^{n} \delta_{z^n_i}.
$$
For brevity we set $\beta_n:=\beta_{\!\mu_n}$, $u_n:=u_{\mu_n}$ and simply denote
$\Omega_n(\mu_n)$ by $\Omega_n$. Finally, let 
$$
\mu_n\boxtimes\mu_n:=\frac{1}{n^2} \sum_{i=1}^{n}\sum_{j\neq i}\delta_{(z^n_i,z^n_j)}.
$$
As a preliminary step, we shall prove that
\begin{equation}\label{Bclaim}
\frac{1}{2n^2} \sum_{i=1}^{n}\sum_{j\neq i}
\int_{\Omega_n}\C K^n(x;z^n_i):K^n(x;z^n_j)\, dx
= \frac12
\iint_{\Omega\times\Omega}V(y,z)\, d(\mu_n\boxtimes\mu_n)(y,z)
+o(1)
\end{equation}
and
\begin{equation}\label{Dclaim}
\frac1{n} \sum_{i=1}^{\n}\int_{\partial\Omega_n}\C K^n(x;z_i^n)\nu(x)\cdot u_n(x)\, d\HH^1(x)
= \frac1{n} \sum_{i=1}^{\n}\int_{\partial\Omega}\C K(x;z_i^n)\nu(x)\cdot u_n(x)\, d\HH^1(x)
+o(1),
\end{equation}
where $o(1)$ denotes an infinitesimal quantity, as $n\to\infty$.
Moreover, if we assume in addition that $\mu_n\weakto\mu$ narrowly, as $n\to\infty$, then
\begin{equation}\label{Cclaim}
\lim_{n\to\infty}\frac1{2n} \sum_{i=1}^{\n}\int_{\partial\Omega}\C K(x;z_i^n)\nu(x)\cdot u_n(x)\, d\HH^1(x)
= I_\mu(v_\mu)=\min I_\mu.
\end{equation}
The proof of \eqref{Bclaim}--\eqref{Cclaim} is split into several steps.
\medskip

\noindent{\em Step 1. Proof of \eqref{Bclaim}.}
We first show that, as $n\to\infty$,
\begin{equation}\label{claim}
\frac{1}{2n^2} \sum_{i=1}^{n}\sum_{j\neq i}
\int_{\Omega_n}\C K^n(x;z^n_i):K^n(x;z^n_j)\, dx
= \frac{1}{2n^2} \sum_{i=1}^{n}\sum_{j\neq i}
\int_{\Omega}\C K(x;z^n_i):K(x;z^n_j)\, dx
+o(1).
\end{equation}
Here we follow an argument in \cite[Proof of Theorem~5.1]{CermelliLeoni06}. We recall that by \eqref{def:Ke}
we have $K^n(x;z^n_i)= K(x;z^n_i)+\e_n^2\nabla w_{n,i}$, where we set $w_{n,i}(x):=w(x-z^n_i)$.
Therefore,
\begin{eqnarray*}
\lefteqn{\int_{\Omega_n}\C K^n(x;z^n_i):K^n(x;z^n_j)\, dx
 = \int_{\Omega_n}\C K(x;z^n_i):K(x;z^n_j)\, dx}
\\
&&{}+\e_n^2\int_{\Omega_n}\C K(x;z^n_i):\nabla w_{n,j}(x)\, dx
+\e_n^2\int_{\Omega_n}\C K(x;z^n_j):\nabla w_{n,i}(x)\, dx
\\
&&{}+\e_n^4\int_{\Omega_n}\C\nabla w_{n,i}(x):\nabla w_{n,j}(x)\, dx.
\end{eqnarray*}
Applying the divergence theorem the last three integrals can be written as
\begin{multline*}
\e_n^2\int_{\partial\Omega}\big(\C K(x;z^n_i)\nu\cdot w_{n,j}(x)+
\C K(x;z^n_j)\nu\cdot w_{n,i}(x)+\e_n^2 \C\nabla w_{n,i}(x)\nu\cdot w_{n,j}(x)\big)\, d\HH^1(x)
\\
-\e_n^2\sum_{k=1}^{n} \int_{\partial B_{\e_n}(z^n_k)}
\big(\C K(x;z^n_i)\nu\cdot w_{n,j}(x)+
\C K(x;z^n_j)\nu\cdot w_{n,i}(x)+\e_n^2 \C\nabla w_{n,i}(x)\nu\cdot w_{n,j}(x)\big)\, d\HH^1(x).
\end{multline*}
Since $|K(x;x_0)|\leq C|x-x_0|^{-1}$ and $|w(x)|\leq C|x|^{-2}$, while $|\nabla w(x)|\leq C|x|^{-3}$,
we easily deduce that the first term in the formula above is of order $\e_n^2$. As for the second term, by \eqref{hyp:enrn} 
we have
\begin{align*}
\e_n^2\Big|\sum_{k\neq i,j}  \int_{\partial B_{\e_n}(z^n_k)} &
\big(\C K(x;z^n_i)\nu\cdot w_{n,j}(x)+
\C K(x;z^n_j)\nu\cdot w_{n,i}(x)+\e_n^2 \C\nabla w_{n,i}(x)\nu\cdot w_{n,j}(x)\big)\, d\HH^1(x)\Big|
\\
& \leq  C\e_n^3n \bigg(\frac{1}{r_n^3}+\frac{\e_n^2}{r_n^5} \bigg)
\leq C\frac{\e_n^3}{r_n^4} \bigg(1+\frac{\e_n^2}{r_n^2} \bigg),
\end{align*}
while
\begin{align*}
\e_n^2\Big|\sum_{k= i,j}  \int_{\partial B_{\e_n}(z^n_k)} &
\big(\C K(x;z^n_i)\nu\cdot w_{n,j}(x)+
\C K(x;z^n_j)\nu\cdot w_{n,i}(x)+\e_n^2 \C\nabla w_{n,i}(x)\nu\cdot w_{n,j}(x)\big)\, d\HH^1(x)\Big|
\\
& \leq  C\e_n^3\bigg(\frac{1}{\e_n^2r_n}+\frac{1}{\e_n r_n^2}+\frac{1}{r_n^3}\bigg)
= C\frac{\e_n}{r_n}\bigg(1+\frac{\e_n}{r_n}+\frac{\e_n^2}{r_n^2}\bigg).
\end{align*}
Combining the previous estimates, we obtain
$$
\frac{1}{2n^2} \sum_{i=1}^{n}\sum_{j\neq i}
\int_{\Omega_n}\C K^n(x;z^n_i):K^n(x;z^n_j)\, dx
= \frac{1}{2n^2} \sum_{i=1}^{n}\sum_{j\neq i}
\int_{\Omega_n}\C K(x;z^n_i):K(x;z^n_j)\, dx
+o(1).
$$
To conclude the proof of \eqref{claim}, it remains to show that
\begin{equation}\label{claim0}
\lim_{n\to\infty}\frac{1}{2n^2} \sum_{i=1}^{n}\sum_{j\neq i}
\int_{\Omega\setminus\Omega_n}\C K(x;z^n_i):K(x;z^n_j)\, dx
=0.
\end{equation}
Using again the estimate $|K(x;x_0)|\leq C|x-x_0|^{-1}$ and \eqref{hyp:enrn}, we infer
\begin{eqnarray*}
\lefteqn{\frac{1}{2n^2} \Big|\sum_{i=1}^{n}\sum_{j\neq i}
\int_{\Omega\setminus\Omega_n}\C K(x;z^n_i):K(x;z^n_j)\, dx\Big|}
\\
& = &
\frac{1}{2n^2} \Big|\sum_{i=1}^{n}\sum_{j\neq i} \sum_{k=1}^{n}
\int_{B_{\e_n}(z^n_k)}\C K(x;z^n_i):K(x;z^n_j)\, dx\Big|
\\
& \leq & C\bigg(n\frac{\e_n^2}{r_n^2}+\frac{\e_n}{r_n}\bigg)
= C\frac{\e_n}{r_n}\bigg( 1+\frac{\e_n}{r_n^2}\bigg),
\end{eqnarray*}
which proves \eqref{claim0} and, in turn, \eqref{claim}.

Identity \eqref{Bclaim} follows immediately from \eqref{claim}, since by the definitions of $V$
and of $\mu_n\boxtimes\mu_n$ we may write
\begin{align*}
\frac{1}{2n^2}\sum_{i=1}^{n}\sum_{j\neq i} \int_{\Omega}\C K(x;z^n_i):K(x;z^n_j)\, dx
& = \frac12
\iint_{\Omega\times\Omega}\Big(\int_\Omega \C K(x;y): K(x;z)\, dx\Big)\, d(\mu_n\boxtimes\mu_n)(y,z)
 \\
& = \frac12
\iint_{\Omega\times\Omega}V(y,z)\, d(\mu_n\boxtimes\mu_n)(y,z).
\end{align*}
\medskip

\noindent{\em Step 2. Proof of \eqref{Dclaim}.}
By Lemma~\ref{lemma:u-est} for every $n\in \mathbb{N}$ there exists an extension $\tilde u_n\in H^1(\Omega;\R^2)$
of $u_n$ such that
\begin{equation}\label{ue-bound3}
\|\tilde u_n\|_{H^1(\Omega)}\leq C\|E u_n\|_{L^2(\Omega_n)}\leq C,
\end{equation}
with a constant $C$ independent of $n$. This implies that, up to subsequences, $\tilde u_n$
converges weakly in $H^1(\Omega;\R^2)$ to some function $v\in H^1(\Omega;\R^2)$. 
Since $u_n$ (and thus $\tilde u_n$) satisfies the conditions \eqref{normal} for every $n$, the function $v$ belongs to the class \eqref{v-class}.
We also recall
that by \eqref{tr-Be} and \eqref{ue-bound3} we have
\begin{equation}\label{tr-Be*}
\sum_{j=1}^{n}\| u_n\|_{L^2(\partial B_{\e_n}(z_j^n))}
\leq \frac{C}{\sqrt{r_n}} \| E u_n\|_{L^2(\Omega_n)}\leq \frac{C}{\sqrt{r_n}}.
\end{equation}

Using the boundary condition in \eqref{EL-Ke}, we have
\begin{multline}\label{Mdek}
\frac{1}{n} \sum_{i=1}^{n}\int_{\partial\Omega_n}\C K^n(x;z^n_i)\nu\cdot u_n(x)\, d\HH^1(x)
\\
= \frac{1}{n} \sum_{i=1}^{n}\int_{\partial\Omega}\C K^n(x;z^n_i)\nu\cdot u_n(x)\, d\HH^1(x)
- \frac{1}{n} \sum_{i=1}^{n}\sum_{j\neq i}\int_{\partial B_{\e_n}(z^n_j)}
\C K^n(x;z^n_i)\nu\cdot u_n(x)\, d\HH^1(x).
\end{multline}
Arguing as in the proof of \eqref{claim}, it is easy to see that
\begin{equation}\label{limit-bdary}
\frac{1}{n} \sum_{i=1}^{n}\int_{\partial\Omega}\C K^n(x;z^n_i)\nu\cdot u_n(x)\, d\HH^1(x)
=\frac{1}{n} \sum_{i=1}^{n}\int_{\partial\Omega}\C K(x;z^n_i)\nu\cdot u_n(x)\, d\HH^1(x)
+ o(1).
\end{equation}
Moreover, as a consequence of \eqref{supKe-Be} and \eqref{tr-Be*}, we obtain 
\begin{equation}\label{negli-term}
\frac{1}{n}\Big| \sum_{i=1}^{n}\sum_{j\neq i}\int_{\partial B_{\e_n}(z^n_j)}
\C K^n(x;z^n_i)\nu\cdot u_n(x)\, d\HH^1(x)\Big| \leq C\frac{\sqrt\e_n}{r_n^{3/2}}\to 0.
\end{equation}
Combining \eqref{Mdek}--\eqref{negli-term}, we deduce \eqref{Dclaim}.\medskip

\noindent{\em Step 3. Convergence of the boundary energy terms.}
Assume now that $\mu_n\weakto\mu$ narrowly, as $n\to\infty$. We prove that
\begin{equation}\label{claim1}
\lim_{n\to\infty}\frac{1}{n} \sum_{i=1}^{n}\int_{\partial\Omega}\C K(x;z^n_i)\nu\cdot u_n(x)\, d\HH^1(x)
= \int_\Omega\int_{\partial\Omega}\C K(x;y)\nu(x)\cdot v(x)\, d\HH^1(x)\, d\mu(y).
\end{equation}

Since $\tilde u_n=u_n$ on $\partial\Omega$, we have
\begin{equation}\label{limit-bdaryN}
\frac{1}{n} \sum_{i=1}^{n}\int_{\partial\Omega}\C K(x;z^n_i)\nu\cdot u_n(x)\, d\HH^1(x)
= \int_\Omega\int_{\partial\Omega}\C K(x;y)\nu(x)\cdot \tilde u_n(x)\, d\HH^1(x)\, d\mu_n(y). 
\end{equation}
Let now $\Omega'$ be the open set of all points in $\Omega$ with distance from $\partial\Omega$ larger than $\ell/2$.
By the compactness of the trace operator the weak convergence of $\tilde u_n$ to $v$ in $H^1(\Omega;\R^2)$ guarantees strong convergence of the traces in $L^2(\partial\Omega;\R^2)$.
This, in turn, implies
\begin{eqnarray*}
\sup_{y\in \Omega'} \Big| \int_{\partial\Omega}\C K(x;y)\nu(x)\cdot \big( \tilde u_n(x)-v(x)\big)\, d\HH^1(x) \Big|
& \leq &  \| \tilde u_n -v \|_{L^2(\partial\Omega)} \,
\sup_{y\in\Omega'} \| K(\cdot\,;y)\|_{L^2(\partial\Omega)}
\\
& \leq & \frac{C}{\ell}\| \tilde u_n-v\|_{L^2(\partial\Omega)} \to 0.
\end{eqnarray*}
Since $\supp\,\mu_n\subset\Omega'$ for every $n\in \mathbb{N}$, the estimate above, together with the narrow convergence of $\mu_n$ to $\mu$, allows us to pass to the limit in \eqref{limit-bdaryN} and obtain \eqref{claim1}.\medskip 

\noindent
{\em Step 4. Characterization of $v$.}
We now prove that $v=v_\mu$, that is, $v$ is the unique minimizer of the functional $I_\mu$
on the class \eqref{v-class}. 

We have that $|\Omega\setminus\Omega_n|\to0$ since \eqref{hyp:enrn} implies that $nr_n^2\to 0$ as $n\to \infty$. Therefore, by lower semicontinuity
of the elastic energy with respect to weak convergence in $L^2$ we have
$$
\liminf_{n\to\infty} \frac12 \int_{\Omega_n} \C\nabla u_n:\nabla u_n\, dx
= \liminf_{n\to\infty} \frac12 \int_{\Omega_n} \C\nabla \tilde u_n:\nabla \tilde u_n\, dx
\geq \frac12 \int_{\Omega} \C\nabla v: \nabla v\, dx.
$$
Combining the inequality above with \eqref{Dclaim} and \eqref{claim1} yields
$$
\liminf_{n\to\infty} I_{n,\mu_n}(u_n)\geq I_\mu(v).
$$
On the other hand, since $u_n$ minimizes $I_{n,\mu_n}$ on $H^1(\Omega_n;\R^2)$ by Lemma~\ref{lemma:u-est}, we have
$$
 I_{n,\mu_n}(u_n)\leq  I_{n,\mu_n}( u)
$$
for every $u\in H^1(\Omega;\R^2)$. Arguing as in \eqref{Mdek}--\eqref{negli-term},
one can show that the functionals $I_{n,\mu_n}$ converge pointwise to $I_\mu$ on 
$H^1(\Omega;\R^2)$. Thus, we deduce that
$$
I_\mu(v)= \liminf_{n\to\infty} I_{n,\mu_n}(u_n)\leq \lim_{n\to\infty} I_{n,\mu_n}( u)
=I_\mu(u).
$$
We conclude that $v$ is the minimizer of $I_\mu$ on the class \eqref{v-class}, that is, $v=v_\mu$, and
$I_{n,\mu_n}(u_n)\to I_\mu(v_\mu)$. 
Together with \eqref{claim1}, this implies \eqref{Cclaim}, since 
$$
I_\mu(v_\mu)= \frac12  \int_\Omega\int_{\partial\Omega}\C K(x;y)\nu(x)\cdot v_\mu(x)\, d\HH^1(x)\, d\mu(y) ,
$$
where we used integration by parts and the Euler-Lagrange equations satisfied by $v_\mu$.
\bigskip

\noindent
{\bf Liminf inequality.} Let now $(\mu_n)$ be a sequence of measures such that $\mu_n\in X_n$ for every $n\in \mathbb{N}$
and $\mu_n\weakto\mu$ narrowly, as $n\to\infty$. 
We shall prove that
\begin{equation}\label{liminf}
\liminf_{n\to\infty}{\mathcal F}_n(\mu_n)\geq {\mathcal F}(\mu).
\end{equation}

Clearly $\mu$ satisfies $\supp\,\mu \subset \mathcal R$.
We consider the decomposition of ${\mathcal F}_n$ given in \eqref{renorm:F}.
Equation \eqref{Bclaim} guarantees that
\begin{equation}\label{pre-liminf0}
\liminf_{n\to\infty}
\frac{1}{2n^2} \sum_{i=1}^{n}\sum_{j\neq i}
\int_{\Omega_n}\C K^n(x;z^n_i):K^n(x;z^n_j)\, dx
= \liminf_{n\to\infty} \frac12
\iint_{\Omega\times\Omega}V(y,z)\, d(\mu_n\boxtimes\mu_n)(y,z). 
\end{equation}
For $M>0$ let $V_M:=V\wedge M$. We observe that $\supp\,\mu_n, \supp\,\mu\subset\Omega'$ for every $n$,
where $\Omega'$ is the open set of all points in $\Omega$ with distance from $\partial\Omega$ larger than $\ell/2$; thus,
by Lemma~\ref{lemma:Vcont} and Remark~\ref{rmk:bb} the function $V_M$ is continuous and bounded on $\Omega'\times\Omega'$.
Using the narrow convergence of
$\mu_n\boxtimes\mu_n$ to $\mu\otimes\mu$, we infer
\begin{eqnarray}
 \liminf_{n\to\infty} \frac12
\iint_{\Omega\times\Omega}V(y,z)\, d(\mu_n\boxtimes\mu_n)(y,z)
& \geq &  \liminf_{n\to\infty} \frac12
\iint_{\Omega\times\Omega}V_M(y,z)\, d(\mu_n\boxtimes\mu_n)(y,z)
\nonumber
\\
& \geq & \frac12
\iint_{\Omega\times\Omega}V_M(y,z)\, d\mu(y)\, d\mu(z).\label{liminf0}
\end{eqnarray}
Since $M$ is arbitrary, \eqref{pre-liminf0} and \eqref{liminf0} yield
$$
\liminf_{n\to\infty}
\frac{1}{2n^2} \sum_{i=1}^{n}\sum_{j\neq i}
\int_{\Omega_n}\C K^n(x;z^n_i):K^n(x;z^n_j)\, dx
\geq \frac12
\iint_{\Omega\times\Omega}V(y,z)\, d\mu(y)\, d\mu(z).
$$
Combining this inequality with \eqref{Dclaim} and \eqref{Cclaim}, we obtain \eqref{liminf}.\bigskip

\noindent{\bf Limsup inequality.}
Let $\mu \in \mathcal{P}(\Omega)$ be such that $\mathcal{F}(\mu)<+\infty$.
In particular, we have that $\supp\,\mu \subset \mathcal R$.
Since the energy estimates \eqref{Bclaim}--\eqref{Cclaim} are valid for any admissible sequence converging to $\mu$, it is sufficient to construct a sequence $\mu_n$ such that $\mu_n \in X_n$, with $\mu_n \weakto \mu$ narrowly and
$$
\limsup_{\n\to\infty} \frac12
\iint_{\Omega\times\Omega}V(y,z)\, d(\mu_n\boxtimes\mu_n)(y,z) \leq  \frac12
\iint_{\Omega\times\Omega}V(y,z)\, d\mu(y)\, d\mu(z).
$$
Also in this case the proof will be split into several steps.
\medskip

\noindent
{\em Step 1. Approximation of the measure ${\mu}$.}
Define the family of squares
$$
\widetilde{\mathcal Q}^h :=\big\{ [0,2h)^2 + 2h(m,\ell) : \ (m,\ell)\in \Z^2 \big\},
$$
and let $\{\tilde Q_k^h\}_{k=1,\dots,N_h}$ be the squares in $\widetilde{\mathcal Q}^h$ intersecting the support of $\mu$. 
For $k=1,\dots,N_h$ let $Q_k^h$ be defined as
$$
Q_k^h:=\{x \in \tilde Q_k^h: \ x +\lambda_1 e_1+\lambda_2 e_2 \in \tilde Q_k^h, \ 0\leq \lambda_i< h\}.
$$ 
We define the approximating measures $\mu^h$ as
$$
\mu^h:=\sum_{k=1}^{N_h}\frac{\mu(\tilde Q_k^h)}{h^2}\chi_{Q_k^h} 
$$
for every $h$. Without loss of generality we can assume that $\supp\,\mu^h \subset \mathcal R$ for every $h$. Clearly $\mu^h\weakto\mu$ narrowly, as $h\to0$.

We claim that
\begin{equation}\label{mest-d}
\limsup_{h\to0}\iint_{\Omega\times\Omega}V(y,z)\, d\mu^h(y)\, d\mu^h(z) \leq \iint_{\Omega\times\Omega}V(y,z)\, d\mu(y)\, d\mu(z). 
\end{equation}
For $M>0$ we consider the truncated function $V_M:=V\wedge M$
and we write $V=V_M+(V-V_M)$. 

Let $\Omega'$ be the open set of all points in $\Omega$ with distance from $\partial\Omega$ larger than $\ell/2$. 
Since $\supp\,\mu^h, \supp\,\mu\subset\Omega'$ for every $h$ and $V_M$ is continuous and bounded on $\Omega'\times\Omega'$, 
narrow convergence yields
\begin{eqnarray*}
\lim_{h\to0} \iint_{\Omega\times\Omega}V_M(y,z)\,d\mu^h(z)\, d\mu^h(y)
& = & \iint_{\Omega\times\Omega}V_M(y,z)\,d\mu(z)\, d\mu(y)  
\\
& \leq & \iint_{\Omega\times\Omega}V(y,z)\,d\mu(z)\, d\mu(y) .
\end{eqnarray*}
Therefore, \eqref{mest-d} is proved if we show that
\begin{equation}\label{mm10}
\lim_{M\to\infty}\limsup_{h\to0}\iint_{\Omega\times\Omega}(V(y,z)-V_M(y,z))\,d\mu^h(z)\, d\mu^h(y)= 0.
\end{equation}
We observe that the bound \eqref{imp-est} on $V$ implies that 
\begin{equation}\label{imp-est2}
|V(y,z)|>M \quad \Longrightarrow \quad |y-z|<R_M,
\end{equation}
for some $R_M>0$ such that $R_M \to 0$, as $M\to \infty$. Therefore,
\begin{equation}\label{mm2}
\iint_{\Omega\times\Omega}(V(y,z)-V_M(y,z))\,d\mu^h(z)\, d\mu^h(y)
\leq \int_\Omega\Big( \int_{B_{R_M}(y)}V(y,z)\,d\mu^h(z)\Big)\, d\mu^h(y).
\end{equation}
By estimate \eqref{below-est} of Lemma~\ref{lemma:Vcont} there exist $C(\Omega')>0$ and $\overline R>0$ such that
\begin{equation}\label{below-est-2}
C(\Omega')(1-\log|y-z|)\leq V(y,z)
\end{equation}
for every $y,z\in\Omega'$ with $0<|y-z|\leq \overline R$. Since $R_M \to 0$, as $M\to \infty$, we can assume that $R_M\ll\overline R$.
For every $k=1,\dots, N_h$ and every $p=1,\dots,P^h_k$, with $P^h_k=\lfloor R_M/(2h)\rfloor +1$, we consider the set of indices
$$
I^h_{k,p}:=\big\{1\leq j\leq N_h:\ Q_j^h=Q_k^h+2h(m,\ell), \ |m|\vee |\ell|=p\big\}.
$$
\begin{eqnarray*}
\lefteqn{\int_\Omega\Big(\int_{B_{R_M}(y)}V(y,z)\,d\mu^h(z)\Big)\, d\mu^h(y)}
\\
&  = & 
\sum_{k=1}^{N_h}\sum_{p=1}^{P^h_k}\sum_{j\in I^h_{k,p}}
\frac{\mu^h(\tilde Q_k^h)\mu^h(\tilde Q_j^h)}{h^4} \int_{Q_k^h}\Big(\int_{B_{R_M}(y)\cap Q_j^h}V(y,z)\,dz\Big)\, dy
\\
&&
{}+ \sum_{k=1}^{N_h} \frac{\big(\mu^h(\tilde Q_k^h)\big)^2}{h^4} \int_{Q_k^h}\Big(\int_{B_{R_M}(y)\cap Q_k^h}V(y,z)\,dz\Big)\, dy.
\end{eqnarray*}
For $k=1,\dots, N_h$ and $j\in I^h_{k,p}$ we have that $|y-z|\geq (2p-1)h$ for every $y\in Q_k^h$, $z\in Q_j^h$; thus,
estimate \eqref{imp-est} implies that
$$
V(y,z)\leq C\Big(1-\log\frac{(2p-1)h}{L}\Big) \quad \text{for every } y\in Q_k^h, z\in Q_j^h, j\in I^h_{k,p}.
$$
On the other hand, since $|y-z|\leq 2\sqrt2(p+1)h$ for every $y\in \tilde Q_k^h$, $z\in \tilde Q_j^h$
with $j\in I^h_{k,p}$, estimate \eqref{below-est-2} yields
$$
C(\Omega')\big(1-\log(2\sqrt2(p+1)h) \big)\leq V(y,z) \quad \text{for every } y\in \tilde Q_k^h, z\in \tilde Q_j^h, j\in I^h_{k,p}.
$$
Therefore, we obtain
\begin{eqnarray*}
\lefteqn{\sum_{j\in I^h_{k,p}}\frac{\mu(\tilde Q_k^h)\mu(\tilde Q_j^h)}{h^4} \int_{Q_k^h}\Big(\int_{B_{R_M}(y)\cap Q_j^h}V(y,z)\,dz\Big)\, dy}
\\
& \leq & C \sum_{j\in I^h_{k,p}} \mu(\tilde Q_k^h)\mu(\tilde Q_j^h)\Big(1-\log\frac{(2p-1)h}{L}\Big)
\\
& = & C \sum_{j\in I^h_{k,p}} \mu(\tilde Q_k^h)\mu(\tilde Q_j^h)\Big(1-\log(2\sqrt2(p+1)h) +\log\frac{2\sqrt2(p+1)L}{2p-1}\Big)
\\
& \leq & C  \sum_{j\in I^h_{k,p}} \int_{\tilde Q_k^h}\Big(\int_{\tilde Q_j^h} V(y,z)\,d\mu(z)\Big)\,d\mu(y) + C \sum_{j\in I^k_{h,p}}\mu(\tilde Q_k^h)\mu(\tilde Q_j^h).
\end{eqnarray*}
Similarly, by applying \eqref{imp-est} and \eqref{below-est-2} and using polar coordinates we have
\begin{eqnarray*}
\frac{\big(\mu(\tilde Q_k^h)\big)^2}{h^4} \int_{Q_k^h}\Big(\int_{B_{R_M}(y)\cap Q_k^h}V(y,z)\,dz\Big)\, dy
& \leq & C  \frac{\big(\mu(\tilde Q_k^h)\big)^2}{h^4}  \int_{Q_k^h}\Big(\int_{B_{\sqrt2 h}(y)\cap Q_k^h}V(y,z)\,dz\Big)\, dy
\\
& \leq & C  \big(\mu(\tilde Q_k^h)\big)^2\big(1-\log(\sqrt2h) \big)
\\
& \leq & C \int_{\tilde Q_k^h}\Big(\int_{\tilde Q_k^h} V(y,z)\,d\mu(z)\Big)\,d\mu(y)
+ C \big(\mu(\tilde Q_k^h)\big)^2.
\end{eqnarray*}
Combining together the previous estimates,
we conclude that
\begin{multline*}
\int_\Omega\Big(\int_{B_{R_M}(y)}V(y,z)\,d\mu^h(z)\Big)\, d\mu^h(y)
\\
\leq  C \int_\Omega\Big(\int_{B_{R_M}(y)}V(y,z)\,d\mu(z)\Big)\, d\mu(y) 
+ C\int_\Omega \mu(B_{R_M}(y))\, d\mu(y), 
\end{multline*}
where the right-hand side tends to zero, as $R_M\to0$, since ${\mathcal F}(\mu)<+\infty$. Owing to \eqref{mm2},
this proves \eqref{mm10}, hence the claim \eqref{mest-d}. 
In virtue of \eqref{mest-d} and of the metrizability of the narrow convergence on ${\mathcal P}(\Omega)$, 
it is sufficient to construct a recovery sequence for the approximating measures $\mu^h$.
\medskip

\noindent
{\em Step 2. Construction of the recovery sequence.}
Let us fix $h>0$. For every $n$ we denote with $\mu^h_n$ an approximation of the measure $\mu^h$ defined in the previous step, of the following form:
$$
\mu^h_n:=\sum_{k=1}^{N_h} c_{k,n}\chi_{Q_k^h},
$$
where the constants $c_{k,n}$ are chosen in such a way that $\sqrt{n \mu_n^h (Q_k^h)} \in \N_0$ for every $k$
and $\mu^h_n(Q_k^h)\to \mu^h(Q_k^h)$, as $n\to\infty$.

At this point we construct the recovery sequence $\mu_n= \frac1n \sum_{i=1}^n\delta_{z_i^n}$ by choosing the points $z_i^n\in \Omega$ in the following way:
\begin{enumerate}
\item to each square $Q^h_k$, allocate $n \mu_n^h(Q_k^h)$ dislocations;
\item in a given square $Q^h_k$, choose $\sqrt{n\mu_n^h(Q_k^h)}$ equispaced horizontal coordinates and $\sqrt{n \mu_n^h(Q_k^h)}$ equispaced vertical coordinates. Then the dislocations will be located at the nodes of this square grid, and their distance within the same square $Q_k^h$ satisfies the bound
$$
|z_i^n - z_j^n | > \frac{h}{\sqrt{n\mu_n^h(Q_k^h)}} \sim \frac{C(h)}{\sqrt{n}} \gg r_n,
$$
since, by assumption \eqref{hyp:enrn} $n r_n \to 0$, as $n\to\infty$. Such a sequence is therefore admissible.
\end{enumerate}
By construction we have that $\mu_n\llcorner{Q^h_k} \weakto \mu^h\llcorner{Q_k^h}$ narrowly, as $n\to\infty$, for every $k$.
This implies that $\mu_n \weakto \mu^h$ narrowly as $n\to\infty$.
\medskip

\noindent
{\em Step 3. Convergence of the energy.} It remains to prove that
\begin{equation}\label{ls-claim2m}
\limsup_{n\to \infty} \frac12
\iint_{\Omega\times\Omega}V(y,z)\, d(\mu_n\boxtimes\mu_n)(y,z) \leq  \frac12
\iint_{\Omega\times\Omega}V(y,z)\, d\mu^h(y)\, d\mu^h(z).
\end{equation}
For $k=1,\dots,N_h$ and $n\in\N$ we introduce the set
$$
J_k^n:=\big\{1\leq i \leq n:\ z_i^n \in Q_k^h\big\}.
$$
We then calculate
\begin{multline}\label{en:splitm}
\frac12\iint_{\Omega\times\Omega}V(y,z)\, d(\mu_n\boxtimes\mu_n)(y,z) 
= \frac{1}{2n^2}\sum_{\substack{i,j=1\\i\not=j}}^{n} V(z^n_i, z_j^n)
\\
= \frac1 {2n^2}\sum_{\substack{k,\ell=1\\k\not=\ell}}^{N_h}\sum_{\substack{i\in J_k^n\\j\in J_\ell^n}} V(z_i^n, z_j^n)
+ \frac1{2n^2}\sum_{k=1}^{N_h}\sum_{\substack{i,j\in J_k^n\\ i\neq j}}V(z_i^n, z_j^n).
\end{multline}
For the first term in the right-hand side of \eqref{en:splitm} we have that
\begin{equation}\label{different-cube}
\lim_{n\to \infty}  \frac1 {2n^2}\sum_{\substack{k,\ell=1\\k\not=\ell}}^{N_h}\sum_{\substack{i\in J_k^n\\j\in J_\ell^n}} V(z_i^n, z_j^n)
=\frac12\sum_{\substack{k,\ell=1\\k\not=\ell}}^{N_h}\iint_{Q^h_k\times Q^h_\ell} V(y,z)\,d\mu^h(y)\,d\mu^h(z),
\end{equation}
since $V$ is continuous and bounded on $Q^h_k\times Q^h_\ell$ for $k\neq\ell$. For the second term in \eqref{en:splitm} we consider, for fixed $M>0$, the 
decomposition $V=V_M+(V-V_M)$, where $V_M:=V\wedge M$.
Since the truncated function $V_M$ is continuous and bounded on an open set containing $\mathcal R\times\mathcal{R}$, we have
\begin{equation}\label{same-cubem}
\lim_{n\to \infty}\frac1{2n^2}\sum_{k=1}^{N_h}\sum_{\substack{i,j\in J_k^n\\ i\neq j}}V_M(z_i^n, z_j^n)
 = \frac12\sum_{k=1}^{N_h} \iint_{Q^h_k\times Q^h_k} V_M(y,z)\,d\mu^h(y)\,d\mu^h(z).
\end{equation}
Combining together \eqref{en:splitm}--\eqref{same-cubem}, since $V_M\leq V$, the claim \eqref{ls-claim2m} follows if we prove that
\begin{equation}\label{reduced-claim-m}
\lim_{M\to\infty}\limsup_{n\to\infty}\frac{1}{2n^2}
\sum_{k=1}^{N_h}\sum_{\substack{i,j\in J_k^n\\ i\neq j}} (V-V_M)(z_i^n, z_j^n) = 0.
\end{equation}

From \eqref{imp-est2}, for large enough $M$ (so that, in particular, $R_M < h$) we have that, for every $k=1,\dots,N_h$,
\begin{align}\label{trunc-estm}
\frac1{2n^2} \sum_{\substack{i,j\in J_k^n\\ i\neq j}} (V-V_M)(z_i^n, z_j^n)
\leq \frac1{2n^2} \sum_{\substack{i,j\in J_k^n \\ |z_i^n-z_j^n|<R_M \\ i\neq j}} V(z_i^n, z_j^n).
\end{align}
Now we denote with $L_k^n$ the size of the square grid in $Q_k^h$ formed by the dislocations locations, i.e.,
$$
L_k^n := \frac{h}{\sqrt{n \mu_n^h(Q_k^h)}}.
$$
Then we can rewrite the last term in \eqref{trunc-estm} as
\begin{equation}\label{est-VM}
\frac1{2n^2} \sum_{\substack{i,j\in J_k^n \\ |z_i^n-z_j^n|<R_M \\ i\neq j}} V(z_i^n, z_j^n)
= \frac1{2n^2 }\sum_{i\in J_k^n}\sum_{p=2}^{P_k^n}\sum_{j\in J_k^n(i,p)} V(z_i^n, z_j^n),
\end{equation}
where $P_k^n$ is the integer part of $R_M/L_k^n$ and, for $p=2,\dots,P_k^n$,
$$
J_k^n(i,p):=\big\{j\in J_k^n: \ (p-1) L_k^n\leq  |z_j^n-z_i^n|< pL_k^n\big\}.
$$
The cardinality of the set $J_k^n(i,p)$ can be estimated as
$$
\# (J_k^n(i,p)) \leq C(p^2 - (p-1)^2) = C(2p +1);
$$
moreover, for every $j\in J_k^n(i,p)$ we have, in virtue of \eqref{imp-est},
$$
|V(z_i^n, z_j^n)| \leq C\Big(1-\log\frac{(p-1)L_k^n}{L}\Big).
$$
Hence, we deduce
\begin{align}\label{est-V2-2}
&\Big|\frac1{2n^2 }\sum_{i\in J_k^n}\sum_{p=2}^{P_k^n}\sum_{j\in J_k^n(i,p)} V(z_i^n, z_j^n) \Big| 
\leq 
\frac C{2n^2 }\sum_{i\in J_k^n} \sum_{p=2}^{P_k^n} (2p +1)\Big(1-\log\frac{(p-1)L_k^n}{L}\Big)\nonumber\\
&\qquad\qquad \qquad\leq \frac{C \# (J_k^n)}{n^2} \sum_{p=2}^{P_k^n} \Big(2p+1 -2(p-1)\log\frac{(p-1)L_k^n}{L}
-3\log \frac{L_k^n}{L} \Big).
\end{align}
It is easy to show that, for every $p=2,\dots,P_k^n$,
$$
-2(p-1)\log\frac{(p-1)L_k^n}{L} \leq \frac{C}{L_k^n};
$$
therefore, the sum in the right-hand side of \eqref{est-V2-2} can be estimated as follows:
$$
\sum_{p=2}^{P_k^n} \Big(2p+1 + \frac{C}{L_k^n} - 3 \log \frac{L_k^n}{L} \Big) \leq CP_k^n\Big(1+\frac{1}{L_k^n}\Big) + CP_k^n(P_k^n+1).
$$
Combining this estimate with \eqref{est-V2-2} and the definition of $P_k^n$ and $L_k^n$, we deduce
\begin{align*}
\Big|\frac1{2n^2 }\sum_{i\in J_k^n}\sum_{p=2}^{P_k^n}\sum_{j\in J_k^n(i,p)} V(z_i^n, z_j^n) \Big| 
&\leq C\, \frac{\mu_n^h(Q_k^h)}{n}\Big(\frac{R_M}{L_k^n} + \frac{R_M}{(L_k^n)^2} + \frac{R_M^2}{(L_k^n)^2} \Big)
\\
&\leq C R_M \frac{\big(\mu_n^h(Q_k^h)\big)^{3/2}}{h\sqrt n} + CR_M \frac{\big(\mu_n^h(Q_k^h)\big)^2}{h^2} 
+ CR_M^2 \frac{\big(\mu_n^h(Q_k^h)\big)^2}{h^2}.
\end{align*}
From the last estimate it follows that
\begin{equation}\label{estVM-last}
\limsup_{n\to\infty} \Big|\frac1{2n^2 }\sum_{i\in J_k^n}\sum_{p=2}^{P_k^n}\sum_{j\in J_k^n(i,p)} V(z_i^n, z_j^n) \Big|  
\leq CR_M \frac{\big(\mu^h(Q_k^h)\big)^2}{h^2} 
+ CR_M^2 \frac{\big(\mu^h(Q_k^h)\big)^2}{h^2}.
\end{equation}
Since $R_M \to 0$, as $M\to \infty$, \eqref{trunc-estm}, \eqref{est-VM}, and \eqref{estVM-last} prove the claim \eqref{reduced-claim-m}. 
This concludes the proof of \eqref{ls-claim2m} and of the theorem. 
\end{proof}

\begin{remark}
We observe that there exists a positive constant $C$ such that 
\begin{equation}\label{lower-bound:calFn}
\mathcal{F}(\mu) \geq -C
\end{equation}
for every $\mu\in\mathcal P(\Omega)$.
Indeed, from \eqref{VboundC} follows directly that the interaction energy term is uniformly bounded from below. Moreover, considering
$v=0$ as test function for $I_\mu$ in the class \eqref{v-class}, we obtain
$$
\frac12\int_\Omega\C\nabla v_\mu:\nabla v_\mu\, dx\leq - \int_\Omega\int_{\partial\Omega} \C K(x;y)\nu(x)\cdot v(x)\,d\HH^1(x)\,d\mu(y),
$$
hence
$$
\int_\Omega|E v_\mu|^2 \, dx \leq \frac{C}{\ell}\|v_\mu\|_{L^2(\partial\Omega)}.
$$
By Korn's inequality and the definition of the class \eqref{v-class} we deduce that the functions $v_\mu$ are uniformly bounded in $H^1(\Omega;\R^2)$
with respect to $\mu$,
hence their traces are uniformly bounded in $L^2(\partial\Omega;\R^2)$
with respect to $\mu$. Finally, this implies that the minimum values $I_\mu(v_\mu)$ are uniformly bounded from below, as well.
\end{remark}

We conclude this section with a characterization of the class where the $\Gamma$-limit functional $\mathcal F$ is finite.

\begin{prop}
Let $\mu\in\P(\Omega)$. Then ${\mathcal F}(\mu)<+\infty$ if and only if $\mu\in H^{-1}(\Omega)$.
\end{prop}

\begin{proof}
Let $\mu\in\P(\Omega)$ and let \begin{equation}\label{beta}
\beta^*_\mu(x):=\nabla v_\mu(x)+\int_\Omega K(x;y)\,d\mu(y),
\end{equation}
where $v_\mu$ is the function associated to $\mu$ according to Section~\ref{subsec:aux-functions}, that is, the unique minimizer of $I_\mu$ on the class
\eqref{def:Iclass}. By Fubini Theorem we have that $\beta_\mu\in L^1(\Omega;\mtwo)$.
Moreover, $\beta_\mu$ is the unique (up to skew-symmetric matrices) solution of the system
\begin{equation}\label{beta-sys}
\begin{cases}
\Div\C\beta_\mu=0 & \text{in } \Omega,
\\
\C\beta_\mu\nu=0 & \text{on }\partial\Omega,
\\
\curl\beta_\mu=\mu\, e_1 & \text{in } \Omega.
\end{cases}
\end{equation}

Let now $\mu\in\P(\Omega)\cap H^{-1}(\Omega)$. Then the problem
$$
\min \Big\{ \frac12 \int_\Omega \C\beta(x):\beta(x)\, dx : \ \beta\in L^2(\Omega;\mtwo), \ \curl\beta=\mu\, e_1 \Big\}
$$
has a solution $\tilde\beta$ (indeed, the class where the minimum is taken is not empty:
for instance, we can solve $\Delta u=\mu\, e_1$, $u\in H^1_0(\Omega;\R^2)$, and take 
as $\beta$ a $\pi/2$ rotation of $\nabla u$). Since $\tilde\beta$ is a solution to \eqref{beta-sys}, 
we deduce that $\beta^*_\mu\in L^2(\Omega;\mtwo)$. Now, by Fubini Theorem we may write
\begin{eqnarray*}
\lefteqn{\frac12 \int_\Omega \C \beta^*_\mu{\,:\,}\beta^*_\mu\, dx}
\\
& = & \frac12 \int_\Omega \C \nabla v_\mu{\,:\,}\nabla v_\mu\, dx
+ \int_\Omega \int_\Omega  \C K(x;y){\,:\,}\nabla v_\mu\, dx\,d\mu(y)
+  \frac12 \int_{\Omega\times\Omega} V(y;z)\, d\mu(y)\,d\mu(z) 
\\
& = & \vphantom{\int}{\mathcal F}(\mu),
\end{eqnarray*}
hence we conclude that ${\mathcal F}(\mu)<+\infty$.

Conversely, let $\mu\in\P(\Omega)$ be such that ${\mathcal F}(\mu)<+\infty$. By Step~1 in the proof of the limsup inequality
of Theorem~\ref{ups:density} there exists a sequence $(\mu^h)\subset L^\infty(\Omega)$ such that $\mu^h\rightharpoonup\mu$ narrowly and
\begin{equation}\label{en1}
\limsup_{h\to0} {\mathcal F}(\mu^h)\leq {\mathcal F}(\mu)<+\infty.
\end{equation}
Let $\beta^*_h$ be defined as in \eqref{beta} with $\mu$ replaced by $\mu^h$. 
Since $(\mu^h)\subset L^\infty(\Omega)$ (hence, in particular, in $H^{-1}(\Omega)$),
the previous argument guarantees that $\beta^*_h\in L^2(\Omega;\mtwo)$ and by Fubini Theorem 
\begin{equation}\label{en2}
\frac12 \int_\Omega \C \beta^*_h {\,:\,}\beta^*_h \, dx
={\mathcal F}(\mu_h)
\end{equation}
for every $h$.
Combining \eqref{en1} and \eqref{en2}, we obtain that the sequence $(\sym\beta^*_h)$ is uniformly bounded in $L^2(\Omega;\mtwo)$.

On the other hand, since $\mu^h\rightharpoonup\mu$, we have that $\nabla v_{\mu^h}\rightharpoonup \nabla v_\mu$ weakly in
$L^2(\Omega;\mtwo)$ and that
$$
\int_\Omega K(\cdot;y)\,d\mu^h(y) \rightharpoonup \int_\Omega K(\cdot;y)\,d\mu(y)
$$
in the sense of distributions. Thus, $\beta^*_h\rightharpoonup \beta^*_\mu$ in the sense of distributions.
We conclude that $\sym\beta^*_\mu\in L^2(\Omega;\mtwo)$. By the generalized Korn inequality \cite[Theorem~11]{GarroniLeoniPonsiglione10}
we deduce that
$\beta^*_\mu\in L^2(\Omega;\mtwo)$, hence $\mu\in H^{-1}(\Omega)$.
\end{proof}


\section{Quasi-static evolution of dislocation energies}\label{Sect:quasistatic} 
In this section we prove the existence of a quasi-static evolution for the energy-dissipation system given by the renormalised dislocation energy $\mathcal{F}_n$ defined in \eqref{def:curly-F-intro}, and the modified, slip-plane-confined, Wasserstein distance $d$ in~\eqref{def:d} as dissipation distance. 

\subsection{Wasserstein-type dissipation}\label{subsec:d}

The dissipation term is defined in terms of a Wasserstein-type distance $d$~\eqref{def:d}--\eqref{def:Gamma}, which is finite only on pairs of measures with the same 
vertical marginals. This implies that during the evolution only the horizontal positions of the dislocations will be allowed to vary, i.e., each dislocation will be only allowed to move within its slip plane. 

The next lemma summarises a number of properties of the distance $d$ and its relation with the standard Wasserstein distance $d_1$ and with narrow convergence. 
Recall that $d_1$ is defined by
\begin{equation}\label{def:d1}
d_1(\mu,\nu) := \inf_{\gamma\in \Gamma_1(\mu,\nu)} \iint_{\Omega\times\Omega} |x-y|\, d\gamma(x,y),
\end{equation}
where $\Gamma_1(\mu,\nu)$ is the following set of couplings of $\mu$ and $\nu$:
\begin{equation}\label{def:Gamma1}
\Gamma_1(\mu,\nu) := \bigl\{\gamma\in \P(\Omega\times\Omega): \gamma(A\times \Omega) = \mu(A),\ 
  \gamma(\Omega\times A) = \nu(A)\text{ for all Borel sets }A\subset \Omega \bigr\}.
\end{equation}
Since $\Omega$ is bounded, $d_1$ generates the topology of narrow convergence, i.e., convergence against continuous and bounded functions
on $\Omega$.

\begin{lemma}
\label{lemma:props-d}
The following properties are satisfied.
\begin{enumerate}
\item \label{lemma:props-d:1}
Convergence in $d$ implies narrow convergence. 
\item \label{lemma:props-d:12} 
The $d_1$-distance of measures $\mu,\nu \in \mathcal{P}(\Omega)$ is bounded from below by the $d_1$-distance of their horizontal marginals, namely
$$
d_1((\pi_1)_\#\mu, (\pi_1)_\#\nu) \leq d_1(\mu,\nu)\leq d(\mu,\nu).
$$
\item \label{lemma:props-d:2}
The metric $d$ is lower semicontinuous with respect to narrow convergence: if $\mu_k\weakto \mu$ and $\nu_k\weakto \nu$, then $d(\mu,\nu)\leq \liminf_{k\to\infty} d(\mu_k,\nu_k)$. 
\item \label{lemma:props-d:3}
The dissipation $\mathcal D$, introduced in \eqref{e:var-diss}, is lower semicontinuous with respect to pointwise narrow convergence: if $\mu_k,\mu:[0,T]\to\P(\Omega)$ are such that $\mu_k(t)\weakto \mu(t)$ for all $t\in [0,T]$, then $
\mathcal D(\mu,[0,T])\leq \liminf_{k\to\infty} \mathcal D(\mu_k,[0,T]).
$
\item \label{lemma:props-d:dual}
We have the dual characterization
\begin{multline}
\label{lemma:props-d:dual-eq}
d(\mu,\nu) = \sup\biggl\{ \int_\Omega \phi\,(d\mu-d\nu): \ \phi\in C_c^\infty(\R^2),\\ |\phi(x_1,x_2)-\phi(x_1',x_2)|\leq |x_1-x_1'| \text{ for all }x_1,x_1',x_2 \biggr\}.
\end{multline}
\item \label{lemma:props-d:char-1-eps} 
For $\e>0$ let $d_{1,\e}$ be the Monge-Kantorovich transport distance with cost function $c_\e(x,y) = |x_1-y_1| + \e^{-1}|x_2-y_2|$;
then
\begin{equation}
\label{lemma:props-d:char-1-eps-eq}
d(\mu,\nu) = \lim_{\e\to0^+} d_{1,\e}(\mu,\nu).
\end{equation}
\item 
\label{lemma:props-d:5}
For any measurable curve $t\mapsto \mu(t,\cdot)$, define
\begin{multline}
\mathscr{D}(\mu,[0,T]):=\inf\biggl\{\int_0^T \int_\Omega  |\phi(t,x)|\,d\mu(t)\,dt: 
  \quad \phi:(0,T)\times\R^2\to \R \text{ Borel measurable,}\\
  \partial_t \mu +\partial_{x_1}(\phi\mu)=0 \text{ as distributions on }(0,T)\times\R^2\biggr\},
\label{s:dissipation}
\end{multline}
with the convention that $\inf \emptyset = +\infty$.
If $\mathscr D(\mu,[0,T])<\infty$, then $\mathscr D(\mu,[0,T]) = \mathcal D(\mu,[0,T])$.\end{enumerate}
\end{lemma}

\begin{remark}
The condition that $\mathscr D <\infty$ in Part~\ref{lemma:props-d:5} is necessary: for instance, for two points $x,\ \tilde x\in \Omega$ with $x_2=\tilde x_2$,  the curve $\mu(t) = t\delta_x + (1-t)\delta_{\tilde x}$ has $\mathcal D(\mu,[0,1]) = |x_1-\tilde x_1|$. However, $\mathscr D(\mu,[0,1]) = +\infty$, since the `jumping' of the mass from $x$ to $\tilde x$ can not be represented by the continuity equation $\partial_t \mu + \partial_{x_1}(\phi\mu) =0$.
\end{remark}

\begin{proof}
By definition $d_1(\mu,\nu)\leq d(\mu,\nu)$, which implies Part~\ref{lemma:props-d:1}. For Part~\ref{lemma:props-d:12} we observe that 
$$
\gamma \in \Gamma_1(\mu,\nu) \,\, \Rightarrow \,\,  (\pi_1\times \pi_1)_\#\gamma \in \Gamma_1( (\pi_1)_\#\mu,(\pi_1)_\#\nu).
$$
Thus, using the fact that $\pi_1(x)=x_1$, we have
\begin{align*}
d_1((\pi_1)_\#\mu, (\pi_1)_\#\nu) &= \inf_{\hat \gamma\in \Gamma_1((\pi_1)_\#\mu, (\pi_1)_\#\nu)} \iint_{\mathbb{R}\times\mathbb{R}} |x_1-y_1|\, d\hat \gamma(x_1,y_1)\\
 &\leq  \inf_{\gamma\in \Gamma_1(\mu, \nu)} \iint_{\mathbb{R}\times\mathbb{R}} |x_1-y_1|\, d\big( (\pi_1\times \pi_1)_\#\gamma\big)(x_1,y_1)\\
 &= \inf_{\gamma\in \Gamma_1(\mu, \nu)} \iint_{\mathbb{R}^2\times\mathbb{R}^2} |\pi_1 (x)-\pi_1 (y)|\, d \gamma (x,y)\\
 &\leq \inf_{\gamma\in \Gamma_1(\mu, \nu)} \iint_{\mathbb{R}^2\times\mathbb{R}^2} |x-y|\, d \gamma (x,y) = d_1(\mu,\nu).
\end{align*}

We next prove Part~\ref{lemma:props-d:char-1-eps}. Note that $d_{1,\e}(\mu,\nu)\leq d(\mu,\nu)$. To prove $d(\mu,\nu) = \lim_{\e\to0^+} d_{1,\e}(\mu,\nu)$, first assume that $d(\mu,\nu)<\infty$ and let $\gamma_{1,\e}$ be an optimal transport plan for $d_{1,\e}(\mu,\nu)$; since $\mu$ and $\nu$ are fixed, $\gamma_{1,\e}$ is tight, and as $\e\to0^+$ we can assume that $\gamma_{1,\e}\weakto \gamma_{1,0}$ in $\P(\Omega\times\Omega)$. The inequality
$$
\iint_{\Omega\times\Omega} \Big(|x_1-y_1| + \frac1\e |x_2-y_2|\Big) \, d\gamma_{1,\e}(x,y) = d_{1,\e}(\mu,\nu)\leq d(\mu,\nu)
$$
then implies that the limit satisfies $\gamma_{1,0}\in \Gamma(\mu,\nu)$. Therefore $\gamma_{1,0}$ is admissible for $d(\mu,\nu)$, and 
$$
\liminf_{\e\to0^+} d_{1,\e}(\mu,\nu) \geq  \iint_{\Omega\times\Omega}|x_1-y_1|  \, d\gamma_{1,0}(x,y) \geq d(\mu,\nu).
$$
This, together with the inequality $d_{1,\e}(\mu,\nu)\leq d(\mu,\nu)$, proves~\eqref{lemma:props-d:char-1-eps-eq}. On the other hand, if $d(\mu,\nu)=\infty$, then this argument implies that  $d_{1,\e}(\mu,\nu)$ can not remain bounded as $\e\to0^+$, which again implies~\eqref{lemma:props-d:char-1-eps-eq}.

From~\eqref{lemma:props-d:char-1-eps-eq} directly follows Part~\ref{lemma:props-d:2}, since $d_{1,\e}$ is lower semicontinuous with respect to narrow convergence for each $\e>0$ and $d_{1,\e}(\mu,\nu)\leq d(\mu,\nu)$.
Part~\ref{lemma:props-d:3} follows from Part~\ref{lemma:props-d:2} and the definition~\eqref{e:var-diss}. 

We next prove Part~\ref{lemma:props-d:dual}. For $\alpha>0$ define the norm $|(x_1,x_2)|_{\alpha} = |x_1| + \alpha^{-1} |x_2|$, so that $d_{1,\e}$ is the transport distance with respect to the norm $|\cdot|_\e$. Fix $\eta>0$, and choose $\e>0$ such that $d(\mu,\nu)\leq d_{1,\e}(\mu,\nu) + \eta$; using the dual characterization for the Monge-Kantorovich transport distance (e.g.~\cite[Theorem~1.14]{Villani03}), there exists a function $\phi$ with $|\phi(x)-\phi(y)|\leq |x-y|_\e$, and such that $d_{1,\e}(\mu,\nu) = \int_{\Omega}\phi\,(d\mu-d\nu)$. At the cost of another $\eta$ we can also assume that $\phi\in C_c^\infty(\R^2)$, so that 
$$
d(\mu,\nu)\leq 2\eta + \int_{\Omega} \phi\, (d\mu-d\nu).
$$
Since $|\phi(x_1,x_2)-\phi(x_1',x_2)|\leq |(x_1,x_2)-(x_1',x_2)|_\e = |x_1-x_1'|$, $\phi$ is admissible in the supremum in~\eqref{lemma:props-d:dual-eq}. Since $\eta$ is arbitrary, it follows that $d(\mu,\nu)$ is bounded from above by the right-hand side in~\eqref{lemma:props-d:dual-eq}. For the opposite inequality, let $\gamma\in \Gamma(\mu,\nu)$. By applying the disintegration theorem to $\mu$ and $\nu$, we can write $\mu(x_1,x_2) = \zeta(x_2)\tilde \mu_{x_2}(x_1)$ and  $\nu(y_1,y_2) = \zeta(y_2)\tilde\nu_{y_2}(y_1)$, where $\zeta := (\pi_2)_\#\mu = (\pi_2)_\#\nu$ is the common vertical marginal of $\mu$ and $\nu$. Since $\gamma$ only transports from points $x$ to $y$ with $x_2=y_2$, we have $\gamma(x_1,x_2,y_1,y_2) = \zeta(x_2)\delta_{x_2}(y_2)\tilde \gamma_{x_2}(x_1,y_1)$, where the marginal transport plan $\tilde \gamma_{x_2}$ is  an element of $\Gamma_1(\tilde\mu_{x_2},\tilde \nu_{x_2})$
for $\zeta$-a.e.\ $x_2$. We can then write
\begin{align*} 
\iint_{\Omega\times\Omega} |x-y|\, d\gamma(x,y)&= \int_\R \int_{\R^2} |x_1-y_1|\,d\tilde\gamma_{x_2}(x_1,y_1) \,d\zeta(x_2)\\
&\geq \int_\R d_1(\tilde\mu_{x_2},\tilde \nu_{x_2}) \,d\zeta(x_2)\\
&\geq \int_\R \int_\R \phi(x_1,x_2)\, (d\tilde\mu_{x_2}-d\tilde \nu_{x_2})\, d\zeta(x_2)
\end{align*}
for any $\phi$ $1$-Lipschitz in $x_1$, from which follows the assertion.

For Part~\ref{lemma:props-d:5}, we first show the inequality $\mathcal D\leq \mathscr D$. Since we assume $\mathscr D(\mu,[0,T])<\infty$, we may choose an admissible $\phi$. For any $\xi\in C_c^\infty(\R^2)$ with $|\partial_{x_1}\xi|\leq 1$, write 
$$
f(t) := \int_{\R^2} \xi(x)\,d\mu(t)\qquad\text{and}\qquad g(t) := \int_{\R^2} \phi(x,t)\, \partial_{x_1} \xi(x)\, d\mu(t),
$$
and note that $g\in L^1(0,T)$ since $\mathscr D(\mu,[0,T])<\infty$. From the continuity equation it follows that $\partial_t f = g$ in the sense of distributions on $(0,T)$, which implies that $f$ is absolutely continuous. From the fundamental theorem of calculus and the bound $|\partial_{x_1}\xi|\leq 1$ we then deduce 
$$
\int_{\Omega} \xi(x)\,(d\mu(\sigma)-d\mu(s)) = \int_s^\sigma \int_{\Omega} \phi(t,x)\,  \partial_{x_1} \xi(x) \,d\mu(t)\, dt 
\leq \int_s^\sigma \int_{\Omega} |\phi(t,x)|\, d\mu(t)\, dt.
$$
In combination with the dual characterization (part~\ref{lemma:props-d:dual}) this implies the inequality $\mathcal D\leq \mathscr D$.

To show the opposite inequality $\mathscr D\leq \mathcal D$, by the same argument as above we have, for any partition $\{t_i\}$ of $[0,T]$, and $\xi_i\in C_c^\infty(\R^2)$ with $|\partial_{x_1}\xi|\leq 1$, 
\begin{equation}
\label{ineq:D=D-1}
\sum_i \int_{t_i}^{t_{i+1}} \!\!\int_\Omega \phi(t,x)\,\partial_{x_1}\xi_i(x) \, d\mu(t)\, dt 
=\sum_i \int_\Omega  \xi_i(x)\, (d\mu(t_{i+1})-d\mu(t_i))\leq \mathcal D(\mu, [0,T]).
\end{equation}
By approximation it follows that 
$$
\int_0^T \int_{\R^2} \phi(t,x)\, \partial_{x_1}\zeta(x) \, d\mu(t)\, dt \leq \mathcal D(\mu, [0,T])
\qquad \text{for all $\zeta\in C_c^\infty([0,T]\times\R^2)$, $|\partial_{x_1}\zeta|\leq 1$,}
$$
and by defining $\partial_{x_1}\zeta =: \psi$, up to a modification outside of $\Omega$ to obtain compact support, we find
$$
\int_0^T \int_{\R^2} \phi(t,x) \,\psi(t,x) \, d\mu(t)\, dt \leq \mathcal D(\mu, [0,T])
\qquad \text{for all $\psi\in C_c^\infty([0,T]\times\R^2)$, $|\psi|\leq 1$.}
$$
The inequality $\mathscr D\leq \mathcal D$ then follows from the dual characterization of the total variation.
\end{proof}

\subsection{Loading term}
Let $[0,T]$ be the interval of time in which we observe our process.
The evolution of the system is driven by a forcing term of the form 
$$
\int_\Omega f(t,x)\,d\mu(x),
$$
where $f\in W^{1,1}(0,T;C(\Omega))$, which leads to the time-dependent total energy $\widetilde {\mathcal F}_n(\mu,t)$ defined as 
$$
\widetilde {\mathcal F}_n(\mu,t):= {\mathcal F}_n(\mu) - \int_\Omega f(t,x)\,d\mu(x).
$$
The function $f$ in the force term can be interpreted as the potential of a force, in the sense that $\partial_{x_1} f$ can be seen as a shear stress imposed on the system in the horizontal direction. 
This is clear in the case $f(t,x) = \sigma(t)x_1$, where $\sigma$ is a time dependent but spatially uniform shear stress acting on the body.

\subsection{Minimisation problem for the semi-discrete energy} 
In order to prove existence of solutions for the continuous-time quasi-static problem (Definition~\ref{def:quasistat}) we follow the common strategy of time discretization, and we first study the minimisation problem that defines a single time step. Given the current value $\mu^0$ of the dislocation density and the updated value of the force potential $f$ (for this discussion the time $t$ is fixed), the updated value of the dislocation density is obtained by solving the minimisation problem 
\begin{equation}\label{m_p}
\min_{\mu\in X_n} \left\{\mathcal F_n(\mu) + d(\mu, \mu^0) - \int_\Omega f\,d\mu\right\}.
\end{equation}

\begin{lemma}[Continuity of $\mathcal F_n$ and existence of solutions of the incremental problem]
\label{lemma:lscQe-ex}
Let $n\in\N$ and let $\mu^k,\mu\in X_n$ be such that $\mu^k\weakto\mu$ narrowly. Then
$$
\lim_{k\to\infty}\mathcal F_n(\mu^k)=\mathcal F_n(\mu).
$$
For $\mu^0\in X_n$ and $f\in C(\Omega)$ the minimisation problem \eqref{m_p} has a solution. 
\end{lemma}

\begin{proof}
Since $\mu^k\rightharpoonup\mu$, we can choose a numbering of the $n$ dislocation points $z_i^k$ of $\mu^k$ and $z_i$ of $\mu$ such that $z_i^k\to z_i$. By Lemma~\ref{lemma:u-est}, the extended $\tilde u_{\mu^k}$ are uniformly bounded in $H^1(\Omega;\R^2)$, and a subsequence (not relabeled) converges weakly to some $u\in H^1(\Omega;\R^2)$, as $k\to\infty$. The trace of $\tilde u_{\mu^k}$ converges strongly in 
$L^2(\partial\Omega;\R^2)$, as does the trace of $\tilde u_{\mu^k}$ on the moving boundaries $\partial B_{\e_n}(z_i^k)$ (indeed, this follows from the convergence $z_i^k\to z_i$, and the boundedness in $H^1(B_{\e_n}(0);\R^2)$ of translates $x\mapsto \tilde u_{\mu^k}(x-z_i^k)$). 
From the inequalities
$$
I_{n,\mu}(u)\leq \liminf_{k\to\infty} I_{n,\mu^k}(u_{\mu^k})
\leq \liminf_{k\to\infty}I_{n,\mu^k}(v)= I_{n,\mu}(v)
 \qquad
\text{for all } v\in H^1(\Omega;\R^2),
$$
it follows that $u = u_\mu$ on $\Omega_n(\mu)$ and that the whole sequence converges.  

To prove the continuity of $\mathcal F_n$, consider the decomposition in~\eqref{renorm:F} for $\mathcal F_n(\mu^k)$. The first term converges as $k\to\infty$ since $K^n(\cdot;z_i^k)$ is uniformly bounded on $\Omega_n(\mu^k)$. For the second term in~\eqref{renorm:F} we use the strong convergence of $u_{\mu_k}$ on the boundaries $\partial\Omega(\mu^k)$.

Finally, to prove the existence of a minimizer in~\eqref{m_p}, let $(\mu^k) \subset X_n$ be a minimising sequence, and assume without loss of generality that $\mu^k \weakto \mu$ narrowly and $u_{\mu^k}\weakto u_\mu$ in $H^1(\Omega;\R^2)$. Moreover, because of the separation condition in $X_n$, $\mu\in X_n$.  
The forcing term clearly converges; by the lower semicontinuity of $d$ (Lemma~\ref{lemma:props-d}) and the continuity of $\mathcal F_n$, 
$$
\liminf_{k\to \infty}\left\{\mathcal F_n(\mu^k)+ d(\mu^k, \mu^0) - \int_\Omega f \,d\mu^k \right\} \geq \mathcal F_n(\mu) + d(\mu, \mu^0) - \int_\Omega f\,d\mu,
$$
which proves the minimality of the limit measure $\mu$ and concludes the proof of the lemma.
\end{proof}


\subsection{Quasi-static evolution}
\label{subsec:existence-finite-n}
In this section we prove the existence of a quasi-static evolution in the sense of Definition~\ref{def:quasistat} for a time-dependent force potential $f\in W^{1,1}(0,T;C(\Omega))$.

\begin{theorem}
\label{th:existence-finite-n}
Let $n\in\N$ and let $\mu_0\in X_n$ satisfy the stability condition
$$
\mathcal F_n(\mu_0) - \int_\Omega f(0) \,d\mu_{0} \leq \mathcal F_n(\nu) + d(\nu, \mu_0) - \int_\Omega f(0) \,d\nu,
$$
for every $\nu\in X_n$. Then there exists a quasi-static evolution $t\mapsto \mu(t)$ from $[0,T]$ into
$\mathcal P(\Omega)$ such that $\mu(0)=\mu_0$. In other words, the following two conditions hold:
\begin{itemize}
\item[(qs1)$_n$] global stability: for every $t\in [0,T]$ we have $\mu(t)\in X_n$ and
\begin{equation}\label{e-stab}
\mathcal{F}_n(\mu(t)) - \int_\Omega f(t)\,d\mu(t) \leq \mathcal{F}_n(\nu) +
d(\nu, \mu(t)) - \int_\Omega f(t)\,d\nu,
\end{equation}
for every $\nu \in X_n$;
\item[(qs2)$_n$] energy balance: the map $s\mapsto \int_\Omega \dot{f}(s)\,d\mu(s)$ is integrable on $[0,T]$
and for every $t\in [0,T]$
\begin{equation}\label{e-bal}
\mathcal{F}_n(\mu(t)) +  \mathcal{D}(\mu,[0,t])- \int_\Omega f(t)\,d\mu(t)
= \mathcal{F}_n(\mu(0))  - \int_\Omega f(0)\,d\mu(0) -
\int_0^{t}\int_\Omega \dot{f}(s)\,d\mu(s)\,ds.
\end{equation}
\end{itemize}
\end{theorem}

\begin{remark}\label{rmk:proj}
Note that if $t\mapsto\mu(t)$ is a quasi-static evolution, then
\begin{equation}\label{cmar0}
(\pi_2)_\#\mu(t)=(\pi_2)_\#\mu(0) \quad \text{for every } t\in[0,T].
\end{equation}
Indeed, owing to the regularity assumptions on $f$, 
condition (qs2)$_n$ implies that $\mathcal{D}(\mu,[0,T])<+\infty$, which in turn gives \eqref{cmar0} by the definition of the dissipation distance $d$.
\end{remark}

\begin{proof}[Proof of Theorem~\ref{th:existence-finite-n}]
The theorem is a direct application of \cite[Theorem~4.5]{MainikMielke}.
We sketch the proof for the reader's convenience.

For every $k\in\N$ let $(t_k^i)_{i=0,\dots,k}$ be a partition of the interval $[0,T]$ such that
$$
0=t_k^0 < t_k^1 < \dots <t_k^k=T \quad \text{and}\quad \max_{i=1,\dots,k}|t_k^i - t_k^{i-1}|\to 0, \text{ as }k\to \infty.
$$
For every $k\in\N$ we set $\mu_k^0:=\mu_0$ and for $i\geq 1$ we define $\mu_k^i$ as a solution of the following minimisation problem (well-defined by Lemma~\ref{lemma:lscQe-ex}):
$$
\min_{\nu \in X_n} \Big\{ \mathcal F_n(\nu) +  d(\nu, \mu_k^{i-1}) - \int_\Omega f(t^i_k) \, d\nu \Big\}.
$$
Since $\mu_k^{i-1}$ is a competitor for this problem with finite energy, the minimum is finite. This implies that
$d(\mu_k^i, \mu_k^{i-1})<+\infty$, so that $(\pi_2)_\#\mu_k^i=(\pi_2)_\#\mu_k^{i-1}$ for every $i$ and $k$. In other words,
$(\pi_2)_\#\mu_k^i=(\pi_2)_\#\mu_0$
for every $i$ and $k$.
By the minimality of $\mu_k^i$ and the triangle inequality
we infer that
\begin{equation}\label{sqss}
\mathcal F_n(\mu_k^i) - \int_{\Omega}f(t^i_k) \,d\mu_k^i \leq \mathcal F_n(\nu) + d(\nu, \mu_k^{i}) - \int_{\Omega}f(t^i_k) \,d\nu
\end{equation}
for every $\nu\in X_n$.
Moreover, using again the minimality, the following discrete energy inequality can be proved:
for all $1\leq j\leq k$
\begin{equation}\label{e:minimality}
\mathcal F_n(\mu_k^j) + \sum_{i=1}^jd(\mu_k^i, \mu_k^{i-1}) \leq \mathcal F_n(\mu_0) + \sum_{i=1}^j\int_{\Omega}f(t^i_k) \, d(\mu_k^i - \mu_k^{i-1}).
\end{equation}
Now we let $\mu_k(t)$ be the piecewise constant right-continuous interpolation of $(\mu_k^i)_{i=1,\dots,k}$. In particular, we have
$(\pi_2)_\#\mu_k(t)=(\pi_2)_\#\mu_0$ for every $t\in[0,T]$ and every $k$.

By the definition \eqref{def:curly-F-intro} it is clear that the functional $\mathcal F_n$ is bounded from below by a fixed constant (since $F_n$ is nonnegative, while the self-energy is bounded by $C(n\e_n)^{-1}$). Using this fact and \eqref{e:minimality} we deduce that
$\mathcal{D}(\mu_k, [0,T])\leq C$
for every $k$. Therefore, by Theorem~\ref{thm:Helly} in the Appendix (or by applying the standard Helly's Theorem directly to the sequences (in $k$) of maps $t\mapsto z^k_j(t)$, $j=1,\dots,n$, where $\mu_k(t)=\frac1n\sum_{j=1}^n\delta_{z^k_j(t)}$)
there exist a subsequence (not relabelled) and a function $t\mapsto \mu(t)$ from $[0,T]$ into $\mathcal P(\Omega)$,
with $\mathcal{D}(\mu, [0,T])<+\infty$, such that
$\mu_k(t)\weakto \mu(t)$ narrowly, as $k\to+\infty$, for every $t\in[0,T]$. It is also easy to check that $\mu(t)\in X_n$ for every $t\in[0,T]$ and
$(\pi_2)_\#\mu(t)=(\pi_2)_\#\mu_0$ for every $t\in[0,T]$.

Condition (qs1)$_n$ follows now by passing to the limit in \eqref{sqss}. This is possible since $\mathcal F_n$ and the force term are continuous with respect to narrow convergence, and convergence in $d$ is equivalent to convergence of the dislocation points for sequences of measures in $X_n$.

To prove condition (qs2)$_n$, we can pass to the limit in \eqref{e:minimality} and deduce an energy inequality.
The converse inequality can be proved using the global stability (qs1)$_n$ and an approximation of $t\mapsto\mu(t)$
in terms of piecewise constant maps. This is possible since the condition $\mathcal{D}(\mu, [0,T])<+\infty$ is equivalent to saying that the maps 
$t\mapsto z_j(t)$ are $BV$ functions from $[0,T]$ into $\R^2$ for every $j=1,\dots,n$, where $\mu(t)=\frac1n\sum_{j=1}^n\delta_{z_j(t)}$.
Therefore, in particular, $t\mapsto z_j(t)$ is continuous on $[0,T]$, except 
possibly on a countable set of times. 
\end{proof}

\section{Convergence of the evolutions} \label{conv:evolution}

We now prove the convergence of the evolutions associated to $\mathcal{F}_n$ to an evolution associated to the $\Gamma$-limit functional~$\mathcal{F}$, i.e.,  Theorem~\ref{th:plasticlimit} in the introduction (as Theorem~\ref{th:conv_evolutions} below).
We will consider evolutions for $\mathcal{F}_n$ with initial values $\mu_n^0$ belonging to a subset $Y_n(\gamma,c)$ of the admissible class $X_n$, defined as follows.

\begin{defin}\label{def:classYng}
Let $n\in\N$ and let $\mu = \frac1n\sum_{i=1}^{n}\,\delta_{z_i}$, for some $z_i\in \Omega$. Let $S_\mu$ be the set of slip planes of $\mu$, and for $s\in S_\mu$ let $m_{\mu,s}$  be the number of dislocations on the slip plane $s$, i.e.,
\begin{equation}\label{classS}
S_\mu:= \supp\, [(\pi_2)_{\#}\mu], \qquad m_{\mu,s}:=\#\big\{j:\ \pi_2(z_j)=s\big\}.
\end{equation}
For $-\frac12<\gamma\leq\frac12$ and $c>0$ we introduce the following class of measures:
\begin{equation}\label{def:Yng}
Y_n(\gamma,c) := \Bigl\{ \mu \in X_n:\ \min_{s,s'\in S_\mu, s\neq s'}|s-s'|\geq cn^{-\frac12+\gamma}, \ \max_{s\in S_\mu} m_{\mu,s}\leq \frac1c\,n^{\frac12+\gamma}\Bigr\}.
\end{equation}
\end{defin}%
In words, measures in $Y_n(\gamma,c)$ have slip-plane spacing larger than $cn^{-\frac12 +\gamma}$ and no more than $c^{-1}n^{\frac12+\gamma}$ dislocations per slip plane.
Note that the two properties required in the definition of the class $Y_n(\gamma,c)$ depend only on the vertical marginal of a measure; therefore if $\mu\in Y_n(\gamma,c)$ for some $\gamma$ and $c$ and $\nu\in X_n$ is such that
$(\pi_2)_\#\nu=(\pi_2)_\#\mu$, then $\nu\in Y_n(\gamma,c)$.\medskip

Although the class $Y_n(\gamma,c)$ imposes restrictions on the structure of admissible measures, we can nevertheless recover a large class of measures in the limit, as shown in the following lemma. 

\begin{lemma}\label{lemma:char-Pinfty}
Assume \eqref{hyp:enrn} and
let $-\frac12<\gamma\leq\frac12$ and $c>0$. Let $\mathcal{P}_{\gamma,c}^\infty(\Omega)$ be the class of measures $\mu\in\mathcal P(\Omega)$ such that there exists a sequence $\mu_n\weakto\mu$ narrowly, as $n\to\infty$, with $\mu_n\in Y_n(\gamma,c)$ for every $n$. 
\begin{itemize}
\item If $-\frac12<\gamma<\frac12$, then
\begin{equation}\label{<12}
\mathcal{P}_{\gamma,c}^\infty(\Omega)=\Big\{ \mu\in \mathcal{P}(\Omega): \ \supp\,\mu \subset \mathcal R, \
(\pi_2)_{\#}\mu\leq \frac{1}{c^2} \mathcal{L}^1 \Big\}.
\end{equation}
\item If $\gamma=\frac12$, then $\mathcal{P}_{\gamma,c}^\infty(\Omega)$ is the set of all measures $\mu \in \mathcal{P}(\Omega)$ such that
$\supp\,\mu \subset \mathcal R$,
$S_\mu:=\supp\,[(\pi_2)_{\#}\mu]$ is finite, 
\begin{equation}\label{12a}
\min\big\{|s-s'|: s,s'\in S_\mu,\ s\neq s'\big\}\geq c ,
\end{equation}
and
\begin{equation}\label{12b}
\mu(\Omega\cap(\R\times\{s\})\big) \leq \frac{1}{c} \quad\text{ for every }s\in S_\mu .
\end{equation}
\end{itemize}
\end{lemma}

\begin{remark}
We note that for $-\frac12<\gamma<\frac12$ the class $\mathcal{P}_{\gamma,c}^\infty(\Omega)$ is nonempty if and only if 
$c^2\leq \mathcal L^1(\pi_2(\Omega))$. This follows from the second condition in \eqref{<12} and the fact that $\mu$ is a probability measure. Analogously, for $\gamma=\frac12$ the class $\mathcal{P}_{\gamma,c}^\infty(\Omega)$ is nonempty if and only if 
$c^2\leq \max\{ c, \mathcal L^1(\pi_2(\Omega))\}$. This follows from \eqref{12b} and the fact that $\mu$ is a probability measure,
taking into account that $\#S_\mu\leq  \max\{ 1, \mathcal L^1(\pi_2(\Omega))/c\}$ by \eqref{12a}.
\end{remark}

\begin{proof}[Proof of Lemma~\ref{lemma:char-Pinfty}]
In what follows all the measures that appear are extended to zero outside~$\Omega$.
Let $\mu \in \mathcal{P}_{\gamma,c}^{\infty}(\Omega)$ and let $\mu_n\weakto\mu$ narrowly, as $n\to\infty$, with $\mu_n\in Y_n(\gamma,c)$ for every $n$. Clearly $\supp\,\mu \subset \mathcal R$. Let $I_r$ be an open interval of length $r>0$. 

If $-\frac12<\gamma<\frac12$, then the maximum number of dislocations of $\mu_n$ contained in the horizontal strip $\R\times I_r$ is $rn/c^2$. Therefore,
$$
(\pi_2)_\#\mu(I_r)=\mu(\R\times I_r)\leq \liminf_{n\to\infty} \mu_n(\R\times I_r) \leq \frac{1}{c^2}\mathcal L^1(I_r).
$$
This inequality implies that $(\pi_2)_\#\mu\leq \frac{1}{c^2}\mathcal L^1$ on every open set of $\R$ and thus, by approximation
on every measurable set. 

If $\gamma=\frac12$, then the maximum number of dislocations of $\mu_n$ in $\R\times I_r$ is $\max\{n/c, rn/c^2\}$. This implies~\eqref{12b}, since 
$S_\mu:=\supp\,(\pi_2)_{\#}\mu$ clearly has finite cardinality and satisfies \eqref{12a}.\smallskip

Let now $-\frac12<\gamma<\frac12$ and let $\mu \in \mathcal{P}(\Omega)$ be such that $\supp\,\mu \subset \mathcal R$ and $(\pi_2)_\#\mu\leq \frac{1}{c^2}\mathcal L^1$. In particular, for every interval $I_r$ of length $r>0$ we have
\begin{equation}\label{mudensity}
\mu(\R\times I_r)\leq \frac{r}{c^2}.
\end{equation}
As a first step we assume that $\supp\,\mu \subset {\rm int}\,{\mathcal R}$, where ${\rm int}\,{\mathcal R}$ denotes the interior of $\mathcal R$. 

For every $n\in\N$ we consider a covering $\{Q_i^n\}$ of $\Omega$ made of a grid of squares of side $h_n$ with $h_n\to0$ chosen such that 
\begin{equation}\label{hyphn}
n^{\frac12-\gamma}h_n\to\infty, \quad n^{\frac12+\gamma}h_n\to\infty, \quad \textrm{as } n\to \infty.
\end{equation}
We construct the approximating measures $\mu_n$ as follows. Let us consider the straight lines $\ell_j^n:=\R\times\{cjn^{-\frac12+\gamma}\}$, $j\in\Z$.
For every $i$ and $j$ such that $Q_i^n\cap\ell_j^n\neq\emptyset$ we allocate $m_i^n$ equidistant dislocations on $Q_i^n\cap\ell_j^n$ at a distance $r_n$ from $\partial Q_i^{n}\cap\ell_j^n$, where
$$
m_i^n:=\Big\lfloor \frac{\mu(Q_i^n)}{h_n}cn^{\frac12+\gamma}\Big\rfloor.
$$
Note that $h_n/r_n\to+\infty$, as $n\to\infty$, by \eqref{hyp:enrn} and \eqref{hyphn}. Moreover, by
\eqref{mudensity} the distance between two dislocations in $Q_i^n$ and on the same line is larger than
$$
\frac{h_n-2r_n}{m_i^n}\geq \frac1{2c}\frac{h_n^2}{\mu(Q_i^n)n^{\frac12+\gamma}}
\geq \frac{c}{2}\frac{h_n}{n^{\frac12+\gamma}} \gg r_n,
$$
where the last inequality is satisfied in view of \eqref{hyp:enrn} and \eqref{hyphn}.

Let $I_j^n$ be the set of indices $i$ for which $Q_i^n$ intersects the line $\ell_j^n$. Then the total number of dislocations on $\ell_j^n$ can be estimated by
$$
\sum_{i\in I_j^n} m_i^n\leq \sum_{i\in I_j^n}\frac{\mu(Q_i^n)}{h_n}cn^{\frac12+\gamma} \leq \frac{\mu(S_j^n)}{h_n}cn^{\frac12+\gamma}\leq 
 \frac{1}{c}n^{\frac12+\gamma},
$$
where $S_j^n:=\cup_{i\in I_j^n}Q_i^n$ is a strip of height $h_n$ and the last inequality follows from \eqref{mudensity}.
Since $\supp\,\mu_n \subset {\mathcal R}$ for $n$ large enough, we conclude that $\mu_n\in Y_n(\gamma,c)$ 
(if $\mu_n(\Omega)<1$, we can add the dislocations needed in order for $\mu_n$ to be a probability measure
in the activated slip planes, at distance larger than $r_n$ from other dislocations).

We now check that $\mu_n\rightharpoonup\mu$ narrowly. Note that the number of lines $\ell_j^n$ intersecting a square $Q_i^n$ is 
$\lfloor h_nn^{\frac12-\gamma}/c\rfloor$. Therefore, the total number of dislocations on each square $Q_i^n$ is
$$
p_i^n:=\Big\lfloor\frac{h_n n^{\frac12-\gamma}}{c}\Big\rfloor m_i^n.
$$
Denoting by $\{z_j^{i,n}\}_{j=1,\dots,p_i^n}$ the dislocations of the measures $\mu_n$ in $Q_i^n$, we may write
$$
\mu_n=\frac1n\sum_i \sum_{j=1}^{p_i^n} \delta_{z_j^{i,n}}.
$$
Let now $\phi\in C(\Omega)$. We have
\begin{eqnarray*}
\int_\Omega \phi\,d\mu_n & = & \frac1n\sum_i \sum_{j=1}^{p_i^n} \phi(z_j^{i,n}) 
\\
& = &  \frac1n\sum_i p_i^n \phi(z_1^{i,n})
+ \frac1n\sum_i \sum_{j=1}^{p_i^n} \big(\phi(z_j^{i,n})-\phi(z_1^{i,n})\big),
\end{eqnarray*}
where the last term tends to $0$, as $n\to\infty$, by the uniform continuity of $\phi$ and the fact that $|z_j^{i,n}-z_1^{i,n}|\leq\sqrt2h_n\to0$.
Moreover,
$$
\frac1n\sum_i p_i^n \phi(z_1^{i,n}) =
\frac1n \sum_i \big(p_i^n- n\mu(Q_i^n)\big)  \phi(z_1^{i,n}) 
+ \sum_i \int_{Q_i^n} ( \phi(z_1^{i,n})-\phi )\, d\mu
+ \int_\Omega \phi\, d\mu,
$$
where again by the absolute continuity of $\phi$ the second term on the right-hand side tends to $0$, as $n\to\infty$.
To conclude, it is enough to observe that
\begin{eqnarray*}
\frac1n \sum_i \Big|\big(p_i^n- n\mu(Q_i^n)\big)  \phi(z_1^{i,n}) \Big|
& \leq & \frac1n \|\phi\|_\infty \sum_i\Big( \frac{h_nn^{\frac12-\gamma}}{c} +\frac{\mu(Q_i^n)}{h_n}c n^{\frac12+\gamma}-1\Big)
\\
& \leq & C\Big( \frac{1}{n^{\frac12-\gamma}h_n}+\frac{1}{n^{\frac12+\gamma}h_n} +\frac{1}{nh_n^2} \Big)
\end{eqnarray*}
which goes to $0$ by \eqref{hyphn}. In the last inequality we have used the fact that the number of squares $Q_i^n$ covering $\Omega$ is of order $1/h_n^2$.\smallskip

In the general case, one can proceed as follows. 

For every $k\in\N$ let $\mu^k$ be a push-forward  of $\mu$ by a $1/k$-retraction of its support; clearly $\supp\,\mu^k \subset {\rm int}\,{\mathcal R}$. 
Then, we can construct an approximation $(\mu^k_n)$ of $\mu^k$ as above. In particular, for every $k\in\N$ there exists $N^k\in\N$ such that
$d(\mu_n^k,\mu^k)<1/k$ for every $n\geq N^k$ and we can assume that $N_k$ is strictly increasing with respect to $k$. The required approximation $\mu_n$ of $\mu$ is finally given by $\mu_n:= \mu_n^k$
for $N^k\leq n<N^{k+1}$.\smallskip

Assume now $\gamma=\frac12$ and let $\mu \in \mathcal{P}(\Omega)$ be such that
$\supp\,\mu \subset \mathcal R$, $S_\mu:=\supp\,[(\pi_2)_{\#}\mu]$ has finite cardinality, and
\eqref{12a}--\eqref{12b} hold.
In this case, to construct the approximating sequence, we locate the dislocations on the slip planes of $\mu$ (which by assumption are finite in number and spaced at least $c$ from each other) by following a one-dimensional procedure on every slip plane. More precisely, for every $s\in S_\mu$ we subdivide the straight line $\R\times\{s\}$ into a family of segments $\{I_i^{s,n}\}$ of length $h_n$, where $h_n\to 0$ is chosen so that $nh_n\to\infty$ and
$h_n/(nr_n)\to\infty$. In every segment $I_i^{s,n}$ we allocate $m_i^{s,n}$ equispaced dislocations at a distance $r_n$ from $\partial I_i^{s,n}$, where
$$
m_i^{s,n}:=\big\lfloor\mu(I_i^{s,n})n\big\rfloor.
$$
Note that the distance between two dislocations in $I_i^{s,n}$ is larger than
$$
\frac{h_n-2r_n}{m_i^{s,n}}\geq \frac12\frac{h_n}{n\mu(I_i^{s,n})}
\geq C\frac{h_n}{n} \gg r_n,
$$
where the last two inequalities follow from \eqref{12b} and from our choice of $h_n$.
Moreover, again by \eqref{12b} the number of dislocations on every slip plane is less than $n/c$.
Therefore, the sequence $\mu_n$ supported on the dislocations positions defined above is in $Y_n(\gamma,c)$
and similarly as before we can show that it converges narrowly to $\mu$.
\end{proof}

\begin{remark}
We note that the case $\gamma=-\frac12$ is irrelevant for our analysis since the corresponding class of limit measures is contained in the class $\mathcal P^\infty_{\gamma,c}(\Omega)$ with $-\frac12<\gamma<\frac12$, as it can be easily deduced from the proof of Lemma~\ref{lemma:char-Pinfty}. 
\end{remark}

Let now $f\in W^{1,1}(0,T;C(\Omega))$ and let $(\mu^0_n)$ be a sequence of initial data with $\mu^0_n\in Y_n(\gamma,c)$ for every $n\in\N$. By Theorem~\ref{th:existence-finite-n} for every $n$ there exists a quasi-static evolution
$t\mapsto\mu_n(t)$ with initial value $\mu^0_n$ and potential $f$, according to Definition~\ref{def:quasistat}. 
Note that $(\pi_2)_\#\mu_n(t)=(\pi_2)_\#\mu^0_n$ for every $t\in[0,T]$ and every $n$ by Remark~\ref{rmk:proj}.
Therefore, $\mu_n(t)\in Y_n(\gamma,c)$ for every $t\in[0,T]$ and every $n\in\N$.
Now, in order to pass to the limit in (qs1)$_n$ and (qs2)$_n$ we need to construct a so-called joint recovery sequence for $\mathcal{F}_n$ relative to the distance $d$ (as defined in \cite{MielkeRoubicekStefanelli08}). This construction is done in the next theorem and will also be used to guarantee the existence of admissible sequences of initial data (see Remark~\ref{rmk:aid}).

\begin{theorem}\label{thm:joint}
Assume \eqref{hyp:enrn} and fix $-\frac12<\gamma\leq \frac12$ and $c>0$.
Let $\mu\in \mathcal{P}_{\gamma,c}^{\infty}(\Omega)$ be such that $\mathcal{F}(\mu)<+\infty$
and let $\mu_n\weakto\mu$ narrowly, as $n\to \infty$, with
$\mu_n \in Y_n(\gamma,c)$ for every $n$.
Then for every $\nu\in\mathcal P(\Omega)$ with $(\pi_2)_{\#} \mu = (\pi_2)_{\#}\nu$ and $\mathcal{F}(\nu)<+\infty$, there exists a sequence $(\nu_n)$ such that
\begin{itemize}
\smallskip
\item[(i)] $\nu_n\in Y_n(\gamma,c)$ and $(\pi_2)_{\#}\mu_n = (\pi_2)_{\#}\nu_n$ for every $n$;
\medskip
\item[(ii)] $\nu_n \rightharpoonup \nu$ narrowly, as $n\to \infty$;
\medskip
\item[(iii)] $\displaystyle \limsup_{n\to \infty} d(\mu_n,\nu_n) \leq d(\mu,\nu)$;
\medskip
\item[(iv)] $\displaystyle \limsup_{n\to \infty} \mathcal{F}_n(\nu_n) \leq \mathcal{F}(\nu)$.
\smallskip
\end{itemize}
\end{theorem}

The latter two conditions in Theorem~\ref{thm:joint} motivate the name \textit{joint recovery sequence} for the sequence $\nu_n$. 
The construction of $(\nu_n)$ relies on some technical steps which are summarised in the following Lemmas~\ref{Approx1} and~\ref{Approx2}.

\begin{lemma}\label{Approx1}
Assume \eqref{hyp:enrn}.
For $n\in \N$ let $\{z_i^n:\ i=1,\dots,n\}$ be a family of points in $\Omega$ with $\{z_i^n\}_i\subset \mathcal R$, and let $\mu_n:= \frac{1}{n}\sum_{i=1}^n\delta_{z_i^n}$ denote the corresponding probability measure. Assume that $\mu_n\weakto\mu$ narrowly, as $n\to \infty$. 
Then for every $\nu \in \mathcal{P}(\Omega)$ with ${\rm supp}\, \nu \subset \mathcal R$ and $(\pi_2)_{\#} \mu = (\pi_2)_{\#}\nu$, there exists a sequence $(\hat{\nu}_n) \subset \mathcal{P}(\Omega)$, 
$\hat{\nu}_n= \frac{1}{n}\sum_{i=1}^n\delta_{\hat z_i^n}$,  such that ${\rm supp}\, \hat\nu_n \subset \mathcal R$, $(\pi_2)_{\#} \hat \nu_n = (\pi_2)_{\#} \mu_n$ for every $n$, 
$\hat{\nu}_n\weakto\nu$ narrowly, as $n\to\infty$, and 
$$
\limsup_{n\to\infty} d(\mu_n,\hat{\nu}_n) \leq d(\mu,\nu).
$$
\end{lemma}

\begin{proof}
Let $\nu \in \mathcal{P}(\Omega)$ be such that ${\rm supp}\, \nu \subset \mathcal R$ and $(\pi_2)_{\#} \mu = (\pi_2)_{\#}\nu$.
Fix $\delta>0$. For every $m\in\Z$ we set $S_m:=\Omega\cap (\R\times[m\delta,(m+1)\delta))$ and we define
$\mu^m:=\mu\llcorner S_m$ and $\nu^m:=\nu\llcorner S_m$.
From the assumption $(\pi_2)_{\#} \mu = (\pi_2)_{\#}\nu$ it follows immediately that $(\pi_2)_{\#} \mu^m = (\pi_2)_{\#}\nu^m$
and $\mu^m(\Omega)=\nu^m(\Omega)$.
Moreover, from the definition of the distance $d$ we have 
\begin{equation}\label{lm99}
d(\mu,\nu)= \sum_{m\in\Z} d(\mu^m, \nu^m).
\end{equation}

\noindent
We can construct a family $\{I^m_n\}_{m\in\Z}$ of disjoint subsets of $\{1,\dots, n\}$ with the following properties:
\begin{eqnarray}
&\displaystyle \mu^m_n:=\frac1n \sum_{i\in I^m_n}\delta_{z^i_n} \weakto \mu^m \quad \textrm{as }n\to \infty, \quad \text{for every } m\in\Z,
\label{lm100}
\\
&\displaystyle \dist(z^i_n, S_m)\leq\delta \quad \text{for every } i\in I^m_n, m\in\Z,n\in\N,
\label{lm101}
\\
&\displaystyle \label{lm102}
\frac1n \# (I_n^*)\to 0,
\end{eqnarray}
where $I_n^*:=\{1,\dots,n\}\setminus\cup_{m\in\Z} I^m_n$.

For simplicity of notation we denote the Cartesian coordinates of $z^i_n$ by $(\xi^i_n,\zeta^i_n)$.
For every $n\in\N$, $m\in\Z$, and $i\in I^m_n$ let $\tilde \xi^{m,i}_n\in\R$ be chosen so that 
\begin{equation}\label{lm103}
\lambda^m_n:=\frac1n \sum_{i\in I^m_n}\delta_{\tilde \xi^{m,i}_n} \weakto 
(\pi_1)_\#\nu^m,
\end{equation}
as $n\to\infty$.
Let $T^m_n$ be the non-decreasing map from $\{\xi^i_n\}_{i\in I^m_n}$ to $\{\tilde \xi^{m,i}_n\}_{i\in I^m_n}$.
We define
$$
\hat\nu^m_n:=\frac1n \sum_{i\in I^m_n}\delta_{(T^m_n(\xi^i_n), \zeta^i_n)} \quad \text{for }m\in\Z.
$$
Up to modifying slightly the choice of $\tilde \xi^{m,i}_n$, if needed, we can assume that $\textrm{supp}\,\hat\nu^m_n\subset \mathcal R$.
We note that \eqref{lm103} and the definition of $\hat\nu^m_n$ imply
\begin{equation}\label{lm104}
(\pi_1)_\#\hat\nu^m_n= \lambda^m_n \weakto(\pi_1)_\#\nu^m,
\end{equation}
as $n\to\infty$. Clearly $(\pi_2)_{\#} \mu^m_n = (\pi_2)_{\#} \hat \nu^m_n$. Moreover, using the fact that $T^m_n$ is non-decreasing, we have
$$
d(\mu^m_n,\hat\nu^m_n)= \frac1n \sum_{i\in I^m_n} |\xi^i_n-T^m_n(\xi^i_n)|
=d_1((\pi_1)_\#\mu^m_n,(\pi_1)_\#\hat\nu^m_n),
$$
so that by \eqref{lm100} and \eqref{lm104} we obtain
\begin{equation}\label{lm105}
\lim_{n\to\infty} d(\mu^m_n,\hat\nu^m_n)=d_1((\pi_1)_\#\mu^m, (\pi_1)_\#\nu^m)
\end{equation}
for every $m\in\Z$.

Now we define
$$
\hat\nu_n:=\sum_{m\in\Z} \hat\nu^m_n + \frac1n\sum_{i\in I^*_n}\delta_{z^i_n}.
$$ 
Clearly $\textrm{supp}\,\hat\nu_n \subset \mathcal R$ since both $\hat \nu_n^m$ and $\mu_n$ satisfy the same condition. We will now prove that the sequence $(\hat\nu_n)$ satisfies 
\begin{itemize}
\item[(1)] $(\pi_2)_{\#} \hat \nu_n = (\pi_2)_{\#} \mu_n$ for every $n$;
\smallskip
\item[(2)] $\displaystyle\limsup_{n\to\infty} d_1(\hat{\nu}_n,\nu)\leq 3\delta$;
\smallskip
\item[(3)] $\displaystyle\limsup_{n\to\infty} d(\mu_n,\hat{\nu}_n) \leq d(\mu,\nu)$.
\end{itemize}
Property~(1) holds by construction. Moreover, by Part \ref{lemma:props-d:12} of Lemma~\ref{lemma:props-d}, \eqref{lm99}, and \eqref{lm105}, we have
\begin{eqnarray*}
\limsup_{n\to\infty} d(\mu_n,\hat\nu_n) & \leq & \sum_{m\in\Z} \limsup_{n\to\infty} d(\mu^m_n,\hat\nu^m_n)
=  \sum_{m\in\Z} d_1((\pi_1)_\#\mu^m, (\pi_1)_\#\nu^m)
\\
& \leq &  \sum_{m\in\Z} d(\mu^m,\nu^m)
= d(\mu,\nu),
\end{eqnarray*}
which proves condition (3).

To show (2), we introduce the auxiliary measures
$$
\bar\nu^m_n:= \begin{cases}
\dfrac{\hat\nu^m_n(\Omega)}{\nu^m(\Omega)}\,\nu^m & \text{ if } \nu^m(\Omega)>0
\smallskip\\
\hat\nu^m_n & \text{ if } \nu^m(\Omega)=0
\end{cases}
$$
for every $n\in\N$, $m\in\Z$, and
$$
\bar\nu_n:= \sum_{m\in\Z}\bar\nu^m_n + \frac1n\sum_{i\in I^*_n}\delta_{z^i_n}
$$
for every $n\in\N$. Since $\hat\nu^m_n(\Omega)=\mu^m_n(\Omega)\to \mu^m(\Omega)$, as $n\to\infty$, and \eqref{lm102} holds,
it is immediate to see that $\bar\nu_n\weakto\nu$, as $n\to\infty$, and thus, $d_1(\bar\nu_n, \nu)\to0$,
as $n\to\infty$. In particular, we have
\begin{equation}\label{lm106}
(\pi_1)_\#\bar\nu^m_n \weakto (\pi_1)_\#\nu^m,
\end{equation}
as $n\to\infty$.

Since $\lim_{n\to\infty}d_1(\bar\nu_n, \nu)=0$, to prove condition (2) it is enough to show that
\begin{equation}\label{claim100}
\limsup_{n\to\infty}d_1(\hat\nu_n,\bar\nu_n)\leq 3\delta.
\end{equation}
By \eqref{lm101} the support of $\hat\nu^m_n$ and $\bar\nu^m_n$ is contained in $\R\times[(m-1)\delta,(m+2)\delta]$
for every $n\in\N$ and $m\in\Z$. Therefore, for $m\in\Z$ such that $\nu^m(\Omega)>0$, we have, by \eqref{def:d1} and \eqref{def:Gamma1}:
\begin{eqnarray*}
d_1(\hat\nu^m_n, \bar\nu^m_n) 
& \leq & \inf\Big\{ \iint_{\Omega\times\Omega} \big(|x_1-y_1|+3\delta\big)\, d\gamma(x,y): \ 
\gamma\in \Gamma_1(\hat\nu^m_n, \bar\nu^m_n) \Big\}
\\
& = & \inf\Big\{ \iint_{\pi_1(\Omega)\times \pi_1(\Omega)} |x_1-y_1|\, d\gamma(x_1,y_1): \ 
\gamma\in \Gamma((\pi_1)_\#\hat\nu^m_n, (\pi_1)_\#\bar\nu^m_n) \Big\}
+3\delta 
\\
& = & d_1((\pi_1)_\#\hat\nu^m_n, (\pi_1)_\#\bar\nu^m_n) +3\delta\, \hat\nu^m_n(\Omega).
\end{eqnarray*}

\noindent
Combining this inequality with \eqref{lm104} and \eqref{lm106}, we deduce that
$$
\limsup_{n\to\infty} d_1(\hat\nu^m_n, \bar\nu^m_n) \leq 3\delta\, \hat\nu^m(\Omega)
$$
for every $m\in\Z$ such that $\nu^m(\Omega)>0$.
Since $\bar\nu^m_n=\hat\nu^m_n$ for every $m\in\Z$ such that $\nu^m(\Omega)=0$, we conclude that
$$
\limsup_{n\to\infty} d_1(\hat\nu_n, \bar\nu_n)\leq \sum_{m\in\Z} \limsup_{n\to\infty} d_1(\hat\nu^m_n, \bar\nu^m_n)  
\leq 3\delta \sum_{m\in\Z}\nu^m(\Omega)=3\delta,
$$
which proves the claim \eqref{claim100}, hence condition~(2). 

To conclude the proof of the lemma it is enough to consider a sequence $\delta_k\searrow0$
and apply a diagonal argument.
\end{proof}

\begin{lemma}\label{Approx2}
Let $\hat z_1,\dots, \hat z_n\in\Omega$, $n\in\N$, be such that $\{\hat z_i\}_i\subset \mathcal R$, and let $\hat\nu:= \frac{1}{n}\sum_{i=1}^n\delta_{\hat z_i}$ denote the corresponding probability measure. 
Then there exists $\eta_0>0$, depending only on $\mathcal R$, such that for every $\eta\in(0, \eta_0)$
there exists a modification $\nu$ of $\hat\nu$, $\nu= \frac{1}{n}\sum_{i=1}^n\delta_{z_i}\in\mathcal{P}(\Omega)$, satisfying the following properties:
\begin{itemize}
\item[(a)] $(\pi_2)_{\#} \nu = (\pi_2)_{\#}\hat{\nu}$;
\smallskip
\item[(b)] $d(\hat{\nu},\nu)\leq\eta$;
\smallskip
\item[(c)] $\supp\nu \subset \mathcal R$;
\smallskip
\item[(d)] $|\pi_1(z_i - z_j)|\geq \frac{\eta}{m}$ for every $z_i, z_j \in \supp\,\nu$ such that $\pi_2(z_i) = \pi_2(z_j)$, $i\neq j$,
\smallskip
\end{itemize}
where $m$ is the maximum number of dislocations per slip plane of $\hat\nu$, that is,
$$
m:=\max\#\big\{ i:\ \pi_2(\hat z_i)=s, \ s\in  \supp\, [(\pi_2)_{\#}\hat\nu]\big\}.
$$
\end{lemma}

\begin{proof}
Let $S:=\supp\, [(\pi_2)_{\#}\hat\nu]$ and for $s\in S$ let $m_s$  be the number of dislocations on the $s$-th slip plane, as in \eqref{classS}. By definition $m=\max_{s\in S}m_s$.
We can rewrite the measure $\hat \nu$ by splitting its vertical and horizontal marginals as
$$
\hat \nu(x,y) = \sum_{s\in S}\frac{m_s}{n}\delta_s(y)\hat \nu_s(x),
$$
where $\hat \nu_s \in \mathcal{P}(\mathbb{R})$ is the normalised dislocation density on the $s$-th slip plane. More explicitly, we write
$$
\hat \nu_s = \frac{1}{m_s}\sum_{i=1}^{m_s}\delta_{\hat x_i^s} \quad\text{with } 
\{\hat x_i^s:\ i=1\dots,m_s\} := \{\pi_1(\hat z_i): \ \hat z_i\in \supp\, \hat{\nu},\ \pi_2(\hat z_i)=s\},
$$
and we assume without loss of generality that the points $\{\hat x_i\}$ of the support of $\hat \nu_{s}$ are ordered increasingly.

To construct the required modification $\nu$ of $\hat \nu$, since we want to preserve the vertical marginals, we will only modify the \textit{horizontal} marginals $\hat \nu_s$, for every $s\in S$. More precisely, for every $s\in S$ we consider a grid of size $\eta/m_s$ and define $G_s$ as
$$
G_s:=\Big\{x\in \frac{\eta}{m_s}\,\Z: \  (x,s)\in \ \mathcal{R}\Big\}.
$$
Now let $\nu_s:=  \frac{1}{m_s}\sum_{i=1}^{m_s}\delta_{x^s_i}$ be a minimiser of the following problem:
$$
\min \Big\{d_1(\hat\nu_s,\mu): \ \mu = \frac{1}{m_s}\sum_{i=1}^{m_s}\delta_{x_i}, \ x_i \in G_s \text{ for every } i ,\  x_i\neq x_j \text{ if }  i\neq j\Big\},
$$
where the points $\{x_i^s\}$ are ordered increasingly. Note that to guarantee that the above minimum is taken over a class that is nonempty, $\eta$ has to be smaller than the width of $\mathcal R$ (that we denote by $\eta_0$). 
By construction we have that
\begin{equation}\label{propertiescd}
|x_i^s- x_j^s| \geq \frac{\eta}{m_s} \quad \text{if } i\neq j \quad \text{and} \quad (x_i^s,s) \in \mathcal R
\quad \text{for every } i.
\end{equation}
Moreover, from the minimality it follows that 
\begin{equation}\label{propertyb}
d_1(\hat \nu_s,\nu_s)\leq \eta.
\end{equation}
We now define the measure $\nu\in \mathcal{P}(\Omega)$ as
$$
\nu(x,y):= \sum_{s\in S}\frac{m_{s}}{n}\delta_s(y) \nu_{s}(x).
$$
From \eqref{propertiescd} it follows immediately that $\nu$ satisfies (a), (c), and (d). As for property (b), we have
$$
d(\hat{\nu},\nu) = \sum_{s\in S} \frac{m_{s}}{n} d_1(\hat\nu_{s},  \nu_{s}) \leq \eta,
$$
where the last inequality follows from \eqref{propertyb}.
\end{proof}

We are now ready for the proof of Theorem~\ref{thm:joint}. Roughly speaking, given $\mu\in \mathcal{P}^{\infty}_{\gamma,c}(\Omega)$, a competitor $\nu$ with $(\pi_2)_\#\mu=(\pi_2)_\#\nu$, and $\mu_n\in Y_n(\gamma,c)$ for every $n$ with $\mu_n\rightharpoonup \mu$ narrowly, Lemmas~\ref{Approx1} and~\ref{Approx2} will provide a joint recovery sequence $\nu_n$ by \textit{moving} the horizontal coordinates of the points in the support of $\mu_n$ so that the modified sequence converges to the given measure~$\nu$. 

\begin{proof}[Proof of Theorem~\ref{thm:joint}]
Let $\nu\in\mathcal P(\Omega)$ be such that $(\pi_2)_{\#} \mu = (\pi_2)_{\#}\nu$ and $\mathcal{F}(\nu)<+\infty$.
We consider the cases $\gamma \in (-\frac12,\frac12)$ and $\gamma=\frac12$ separately. 
\smallskip

Let $\gamma \in(-\frac12,\frac12)$. We apply Lemma~\ref{Approx1} to the sequence $(\mu_n)$ and to the measure $\nu$ 
and construct a sequence $(\hat\nu_n)$. Then
we apply Lemma~\ref{Approx2} to each measure $\hat\nu_n$ with $\eta= n^{-\frac12+\gamma}$.
In this way we obtain a new sequence $(\nu_n)$ such that $\textrm{supp}\,\nu_n \subset \mathcal R$, 
$(\pi_2)_{\#} \nu_n = (\pi_2)_{\#} \mu_n$ for every $n$,
$\nu_n$ converges to $\nu$ narrowly (in view of Part~\ref{lemma:props-d:1} of Lemma~\ref{lemma:props-d}
and (b) of Lemma~\ref{Approx2}), and
$$
\limsup_{n\to\infty} d(\mu_n, \nu_n) \leq d(\mu,\nu).
$$

We shall now prove that $(\nu_n)$ is the required joint recovery sequence. 
Taking into account that $(\pi_2)_{\#} \nu_n = (\pi_2)_{\#} \mu_n$, 
by property (d) of Lemma~\ref{Approx2} and by our choice of $\eta$
we have that for every $z^n_i, z^n_j \in \supp\, \nu_n$, with $i\neq j$, 
\begin{equation}\label{n-mindist}
|z^n_i - z^n_j|\geq  cn^{-1} \wedge c n^{-\frac12+\gamma}\geq  cn^{-1} \stackrel{\eqref{hyp:enrn}}\gg r_n.
\end{equation}
Therefore, $\nu_n\in X_n$. Since $\mu_n\in Y_n(\gamma,c)$ and 
$(\pi_2)_{\#} \nu_n = (\pi_2)_{\#} \mu_n$, we conclude that $\nu_n\in Y_n(\gamma,c)$.
In what follows we set $c=1$ for notational simplicity.

To complete the proof of the theorem we show that 
\begin{equation}\label{ls-claimjoint-a}
\limsup_{n\to\infty} {\mathcal F}_n(\nu_n) \leq {\mathcal F}(\nu).
\end{equation}
Applying the preliminary estimates \eqref{Bclaim}--\eqref{Cclaim} in the proof of Theorem~\ref{ups:density}
and the decomposition \eqref{renorm:F},
inequality \eqref{ls-claimjoint-a} is proved if we show that
\begin{equation}\label{ls-claimjoint}
\limsup_{n\to\infty} \frac12
\iint_{\Omega\times\Omega}V(y,z)\, d(\nu_n\boxtimes\nu_n)(y,z) \leq  \frac12
\iint_{\Omega\times\Omega}V(y,z)\, d\nu(y)\, d\nu(z).
\end{equation}
To prove \eqref{ls-claimjoint} we proceed similarly to the proof of the limsup inequality of Theorem~\ref{ups:density}. Let $M>0$ be fixed, and set $V_M:=V\wedge M$. By definition of $\nu_n$ we have
\begin{align}\label{en:splitjoint}
\frac12\iint_{\Omega\times\Omega}V(y,z)\, d(\nu_n\boxtimes\nu_n)(y,z) &= \frac1{2n^2}\sum_{\substack{i,j=1\\i\not=j}}^{n} V_M(z_i^n, z_j^n) + \frac1{2n^2}\sum_{\substack{i,j=1\\i\not=j}}^{n}(V-V_M)(z_i^n, z_j^n)\nonumber\\
&\leq \frac1{2n^2}\sum_{\substack{i,j=1\\i\not=j}}^{n} V_M(z_i^n, z_j^n) + \frac1{2n^2}\sum_{i=1}^{n}\sum_{\substack{j\neq i\\|z_i^n-z_j^n|<R_M}} V(z_i^n, z_j^n),
\end{align}
where $R_M \to 0$ as $M\to \infty$. Since the truncated function $V_M$ is continuous and bounded 
on an open set containing $\mathcal R\times\mathcal{R}$, for the first term in the right-hand side of \eqref{en:splitjoint} we have that
$$
\lim_{n\to \infty}  \frac1 {2n^2}\sum_{\substack{i,j=1\\i\not=j}}^{n} V_M(z_i^n, z_j^n)=\frac12 \iint_{\Omega\times \Omega} V_M(y,z)\,d\nu(y)\,d\nu(z).
$$
Since $V_M\leq V$, claim \eqref{ls-claimjoint} follows if we prove that for every $i=1,\dots,n$
\begin{equation}\label{reduced-claimjoint}
\lim_{M\to\infty}\limsup_{n\to \infty}\, \frac1{n^2}\sum_{i=1}^{n}\sum_{\substack{j\neq i\\|z_i^n-z_j^n|<R_M}} |V(z_i^n, z_j^n)|
= 0.
\end{equation}
Let $Q_{R_M}(z)$ denote the open square centred at $z$ with side length $2R_M$. Clearly $B_{R_M}(z_i^n)\subset Q_{R_M}(z_i^n)$
for every $i$; therefore, if we prove  
\begin{equation}\label{reduced-claimjoint2}
\lim_{M\to\infty}\limsup_{n\to \infty}\, \frac1{n^2}\sum_{i=1}^{n}\sum_{\substack{j\neq i\\ z^n_j \in Q_{R_M}(z_i^n)}} |V(z_i^n, z_j^n)|
= 0,
\end{equation}
then \eqref{reduced-claimjoint} will follow. First of all, by \eqref{imp-est}, we have the bound
\begin{equation}\label{en:bound_log}
\sum_{i=1}^{n}\sum_{\substack{j\neq i\\z^n_j \in Q_{R_M}(z_i^n)}} |V(z_i^n, z_j^n)| \leq C\hspace{-.02cm}\sum_{i=1}^{n}\hspace{-.1cm}\sum_{\substack{j\neq i\\z^n_j \in Q_{R_M}(z_i^n)}}\Big(1-\log\frac{|z_i^n - z_j^n|}{L} \Big).
\end{equation}
Since the right-hand side of \eqref{en:bound_log} is a decreasing function of the distances between the dislocation locations,  the energy of any dislocation arrangement can be estimated from above in terms of 
the energy of the most densely packed configuration. More precisely, 
taking into account that $\nu_n\in Y_n(\gamma,1)$ and \eqref{n-mindist} holds,
an upper bound in \eqref{en:bound_log} is generated by the arrangement where the slip planes are all at the minimum distance $n^{-\frac12+\gamma}$ from each other, and are packed with dislocations located at the minimum distance $1/n$ from one another. Without loss of generality we therefore limit our analysis to this special arrangement, and we denote by $\tilde{z}_j^n$ the corresponding dislocation positions. Since replacing $\{z_i^n\}$ with $\{\tilde{z}_i^n\}$ increases the right-hand side of \eqref{en:bound_log}, we have
\begin{equation}\label{e:estimate}
\frac1{n^2}\sum_{i=1}^{n}\sum_{\substack{j\neq i\\z^n_j \in Q_{R_M}(z_i^n)}} |V(z_i^n, z_j^n)| \leq \frac{C}{n^2}
\sum_{i=1}^{n}\sum_{\substack{j\neq i\\\tilde z^n_j \in Q_{R_M}(\tilde z_i^n)}}\Big(1-\log\frac{|\tilde z_i^n - \tilde z_j^n|}{L}\Big).
\end{equation}

Let $i=1,\dots, n$ be fixed.
To estimate the right-hand side of \eqref{e:estimate}, we first consider the region $A_i^n \subset Q_{R_M}(\tilde z_i^n)$ defined as the following union of concentric square annuli:
$$
A_i^n:=\bigcup_{p=2}^{P_1^n}Q_{\frac{p}{n}}(\tilde z_i^n)\setminus Q_{\frac{p-1}{n}}(\tilde z_i^n), \quad P_1^n:= \lfloor n^{\frac12 + \gamma}\rfloor.
$$
We note that each annulus in the set $A_i^n$ intersects only the slip plane containing $\tilde z_i^n$, since the largest annulus of the family has outer side length $2 P_1^n/n$, which by definition is smaller than $2 n^{-\frac12 + \gamma}$. Since, by assumption, two active consecutive slip planes are at a distance $n^{-\frac12 + \gamma}$, every annulus in the set $A_i^n$ only contains dislocations in the same slip plane as $\tilde z^n_i$. Moreover, there are at most two dislocations in each annulus since, by assumption, the distance between dislocations is $1/n$, which is exactly the width of each annulus. 

By these arguments, for $p=2,\dots,P_1^n$ we have 
\begin{equation}\label{est:V_annuli}
|\tilde z_i^n- \tilde z_j^n| = \frac{p-1}{n}, \quad \forall\, \tilde z_j^n \in Q_{\frac{p}{n}}(\tilde z_i^n)\setminus Q_{\frac{p-1}{n}}(\tilde z_i^n).
\end{equation}
Therefore we can estimate the energy contribution in the sets $A_i^n$ by using \eqref{est:V_annuli}, and we obtain
\begin{eqnarray}
\frac1{n^2}\sum_{i=1}^n\sum_{z^n_j \in A_i^n} \Big(1-\log\frac{|\tilde z_i^n - \tilde z_j^n|}{L}\Big)
& \leq & \frac{2}{n}\sum_{p=2}^{P^n_1}\Big(1-\log\frac{p-1}{nL} \Big)
\nonumber\\
&\leq & \frac{2}{n}\Big( P_1^n - \log( P^n_1!) + P^n_1\log (nL)\Big). \label{enA1n}
\end{eqnarray}
Since $P_1^n\to\infty$, as $n\to \infty$, we use the Stirling approximation
$$
\log(P!) = P\log P - P + O(\log P) \quad \textrm{as } P\to \infty. 
$$
For the last term in \eqref{enA1n} we have
\begin{equation}\label{StirlingP1}
\frac2n \Big( P_1^n - \log(P^n_1!) + P^n_1\log (nL) \Big) =
\frac{2P_1^n}{n} \Big(2-\log\frac{P_1^n}{nL}\Big) + O\Big(\frac{\log P_1^n}{n}\Big).
\end{equation}
Since $P_1^n \simeq n^{\frac12+\gamma}$, $P_1^n/n \simeq n^{-\frac12+\gamma}\to0$ as $n\to\infty$. 
Therefore, the right-hand side of \eqref{StirlingP1} tends to zero, and by \eqref{enA1n} we deduce that the energy contribution in $A_i^n$ is negligible as $n\to \infty$.

It remains now to estimate the energy contribution in the complement of the sets $A_i^n$, i.e., in
$$
B_i^n:= Q_{R_M}(\tilde z_i^n)\setminus A_i^n.
$$
Note that $B_i^n$ contains only dislocations located on different slip planes than $\tilde z_i^n$. This is because the horizontal length of the region containing dislocations is less than $n^{-\frac12 + \gamma}$, while the outer side of the largest annulus in $A_i^n$ is $2\lfloor n^{\frac12 + \gamma}\rfloor/n$, which is larger that $n^{-\frac12 + \gamma}$ for $n$ large enough. 

We now write the region $B_i^n$ as the union of horizontal strips of height $n^{-\frac12 + \gamma}$, as 
$$
B_i^n = \bigcup_{p=1}^{P_2^n} S_p^{i,n}, \qquad 
S_p^{i,n}:= \Big\{x\in B_i^n: pn^{-\frac12 +\gamma}\leq |\pi_2(x- \tilde z_i^n)| < (p+1)n^{-\frac12 +\gamma}\Big\},
$$ 
where $P_2^n:= \lfloor R_M n^{\frac12 - \gamma}\rfloor$. Then we have
$$
\frac1{n^2}\sum_{i=1}^n\sum_{z^n_j \in B_i^n} \Big(1-\log\frac{|\tilde z_i^n - \tilde z_j^n|}{L}\Big)
\leq 
\frac1{n^2}\sum_{i=1}^n \sum_{p=1}^{P_2^n} \, \sum_{\tilde z^n_j \in S_p^{i,n}} \Big(1-\log\frac{|\tilde z_i^n - \tilde z_j^n|}{L}\Big).
$$
Since the maximum number of dislocations per slip plane is $n^{\frac12+\gamma}$,  
\begin{equation}\label{eq:hor2}
\frac1{n^2}\sum_{i=1}^n \sum_{p=1}^{P_2^n} \, \sum_{\tilde z^n_j \in S_p^{i,n}} \Big(1-\log\frac{|\tilde z_i^n - \tilde z_j^n|}{L}\Big)
\leq  
\frac{1}{n}\sum_{p=1}^{P_2^n} n^{\frac12+\gamma} \Big(1-\log\frac{pn^{-\frac12+\gamma}}{L}\Big).
\end{equation}
Now, the right-hand side of \eqref{eq:hor2} can be estimated as follows:
$$
\frac{1}{n}\sum_{p=1}^{P_2^n} n^{\frac12+\gamma} \Big(1-\log\frac{pn^{-\frac12+\gamma}}{L}\Big)
\leq C\int_0^{R_M/L}(1-\log x)\,dx \leq C R_M(1-\log R_M);
$$
thus, it tends to zero, as $M\to \infty$.
This concludes the estimate of the energy contribution in the sets $B_i^n$. In view of
\eqref{e:estimate}, claim \eqref{reduced-claimjoint2} is proved.
\smallskip

We now consider the case $\gamma=\frac12$. 
For simplicity we assume that there is only one active slip plane and $c=1$; 
moreover, we assume that $\supp\,[(\pi_2)_{\#}\mu_n] = \supp\, [(\pi_2)_{\#}\mu] = \supp\,[(\pi_2)_{\#}\nu] =\{s\}$. 
In this case the problem becomes one-dimensional and we can construct the recovery sequence 
by hand, using the same ideas as in the proof of Theorem~\ref{ups:density}. 
More precisely, following the steps of the limsup inequality in Theorem~\ref{ups:density} 
we first approximate the measure $\nu$ by means of a sequence $(\nu^h)$ supported on a subset of $h\mathcal I\times\{s\}$, where
$$
\mathcal I:=  \bigcup_{m\in \mathbb{Z}}I_m, \qquad
I_m:=[2m,2m+1] .
$$
Up to a further approximation, we can assume that $n\nu^h(hI_m)\in \mathbb{N}_0$ for every $h$, $n$, and $m$.
This can be done without increasing the interaction energy. Moreover, since $\supp\,(\pi_2)_{\#}\nu^h =\{s\}$ and
$\nu^h\weakto\nu$ narrowly, as $h\to0$, we have $d(\nu^h,\nu)\to0$, as $h\to0$. In other words,
it is enough to construct a joint recovery sequence for the measures $\nu^h$.

Let $h>0$ be fixed. The recovery sequence $\nu^h_n$ is obtained by arranging 
$n\nu^h(hI_m)$ equispaced dislocations $\{z_i^n\}$ in every interval $hI_m$.  
We observe that $\nu_n^h\in Y_n(\frac12,1)$ and $(\pi_2)_\#\nu_n^h=(\pi_2)_\#\mu_n$ for every $n$; in particular, for every $i\neq j$
we have
\begin{equation}\label{smallestdistance}
|z_i^n - z_j^n| \geq \frac{C(h)}{n} \geq r_n.
\end{equation}
Moreover, by construction, $\nu_n^h \weakto \nu^h$ narrowly. Since $\supp\,(\pi_2)_{\#}\nu_n^h =\{s\}$, this implies that $d(\mu_n,\nu_n^h)\to d(\mu,\nu^h)$.

To prove that $\nu^h_n$ is a recovery sequence it remains to show \eqref{ls-claimjoint-a} and, as for the case $\gamma\neq\frac12$,
this reduces to proving \eqref{reduced-claimjoint}. By \eqref{imp-est} we have that
\begin{eqnarray}
\frac1{n^2}\sum_{i=1}^n\sum_{\substack{j\neq i\\|z_i^n-z^n_j|\leq R_M}} |V(z_i^n, z_j^n)|
& \leq & \frac C{n^2}\sum_{i=1}^n\sum_{\substack{j\neq i\\|z_i^n-z^n_j|\leq R_M}}\Big(1-\log\frac{|z_i^n - z_j^n|}{L}\Big)
\nonumber
\\
\label{estimategamma12}
& \leq & \frac Cn\sum_{p=1}^{\lfloor nR_M/C(h)\rfloor}\Big(1-\log\frac{C(h)p}{nL}\Big), 
\end{eqnarray}
where the last inequality follows by estimating the energy of the distribution of dislocations of $\nu^h_n$ with the most densely packed configuration, from \eqref{smallestdistance}.
By taking the limit as $n\to \infty$ in \eqref{estimategamma12} we have 
$$
\lim_{n\to \infty}\frac1{n^2}\sum_{i=1}^n\sum_{\substack{j\neq i\\|z_i^n-z^n_j|\leq R_M}} |V(z_i^n, z_j^n)| \leq C\int_{0}^{R_M/L} (1-\log x)\,dx,
$$
and the right-hand side converges to zero, as $M\to \infty$. Therefore \eqref{reduced-claimjoint} is proved also for $\gamma=\frac12$. 
\smallskip

This concludes the proof of the theorem. 
\end{proof}

\subsection{Convergence of the evolutions}

In this section we prove the convergence of the quasi-static evolution for the renormalized energy defined as in \eqref{e-stab} and \eqref{e-bal}.

\begin{theorem}\label{th:conv_evolutions}
Assume \eqref{hyp:enrn}. Let $f\in W^{1,1}(0,T;C(\Omega))$ and let $-\frac12<\gamma\leq \frac12$ and $c>0$.
Let $\mu^0\in\P(\Omega)$ with $\mathcal F(\mu^0)<+\infty$ and let $(\mu_n^0)$ be such that $\mu_n^0\weakto\mu^0$ narrowly, and 
\begin{equation}\label{e-data}
\mathcal{F}_n(\mu^0_n)\to \mathcal F(\mu^0),
\end{equation} 
as $n\to\infty$.
For every $n\in\N$ assume also that $\mu_n^0\in Y_n(\gamma,c)$
and $\mu_n^0$ satisfies the stability condition:
\begin{equation}\label{0sc}
\mathcal F_n(\mu_n^0) - \int_\Omega f(0) \,d\mu_n^{0} \leq   \mathcal F_n(\nu) + d(\nu, \mu_n^0) - \int_\Omega f(0) \,d\nu
\end{equation}
for every $\nu\in X_n$. For every $n\in\N$ let $t\mapsto\mu_n(t)$ be a quasi-static evolution on $[0,T]$ with initial value $\mu^0_n$ and potential force~$f$, according to Theorem~\ref{th:existence-finite-n}. Then 
\begin{enumerate}
\item \textbf{Compactness:} There exists a subsequence $\mu_n$ (without change in notation) and a limit curve $\mu:[0,T]\to\P^\infty_{\gamma,c}(\Omega)$ such that $\mu_n(t)\weakto \mu(t)$ for all $t\in [0,T]$; 
\item \textbf{Convergence:} The curve $\mu$ is a quasi-static evolution with initial value $\mu^0$ and force $f$ for the limit energy $\mathcal{F}$ and the dissipation $\mathcal D$. More precisely, the following two conditions are satisfied:
\begin{itemize}
\item[(qs1)$_\infty$] global stability: for every $t\in [0,T]$ and for every $\nu\in \mathcal{P}(\Omega)$,
\begin{equation}\label{0min-1}
\mathcal F(\mu(t))- \int_{\Omega}f(t)\,d\mu(t) \leq
\mathcal F(\nu) + d( \nu, \mu(t))- \int_{\Omega}f(t)\,d\nu;
\end{equation}
\item[(qs2)$_\infty$] energy balance: the map $t\mapsto \int_{\Omega}\dot{f}(t)\,d\mu(t)$ is integrable on $[0,T]$ and 
for every $t\in [0,T]$
\begin{multline}\label{0e-bal-1}
\mathcal F(\mu(t))+ \mathcal{D}(\mu,[0,t])  - \int_{\Omega}f(t)\,d\mu(t) 
= \mathcal F(\mu^0) - \int_{\Omega}f(0)\,d\mu(0) -
\int_0^{t}\int_{\Omega}\dot{f}(s)\,d\mu(s)\,ds.
\end{multline}
\end{itemize}
\end{enumerate}
\end{theorem}

\begin{remark}\label{rmk:aid}
The existence of an admissible sequence of initial data satisfying all the assumptions of Theorem~\ref{th:conv_evolutions}
is guaranteed by Theorems~\ref{ups:density} and~\ref{thm:joint}. 
It can be constructed as follows.
Let $\hat\mu^0\in \P^\infty_{\gamma,c}(\Omega)$ be such that
$\mathcal F(\hat\mu^0)<+\infty$. By definition of $\P^\infty_{\gamma,c}(\Omega)$ there exist 
$\hat\mu^0_n\weakto\hat\mu^0$ narrowly, as $n\to\infty$, such that $\hat\mu^0_n\in Y_n(\gamma,c)$ for every $n$.
By applying Theorem~\ref{thm:joint} to $\mu=\nu=\hat\mu^0$ and $\mu_n=\hat\mu^0_n$
we can replace the sequence $(\hat\mu^0_n)$ with a new sequence on which the renormalized energies $\mathcal F_n$ converge to
$\mathcal F(\hat\mu^0)$. In other words, up to replacing $(\hat\mu^0_n)$ with this new sequence,
we can assume that $\mathcal F_n(\hat\mu^0_n)\leq C$ for every $n$.
By Lemma~\ref{lemma:lscQe-ex} the minimum problems
$$
\min_{\mu\in X_n} \left\{\mathcal F_n(\mu)  + d(\mu, \hat\mu^0_n) - \int_\Omega f(0)\,d\mu\right\}
$$
have a solution $\mu^0_n$ with finite energy. In particular, $d(\mu^0_n, \hat\mu^0_n)<\infty$ for every $n$,
which implies $(\pi_2)_\#\mu^0_n=(\pi_2)_\#\hat\mu^0_n$ for every $n$. Thus, $\mu^0_n\in Y_n(\gamma,c)$ for every $n$.
By the triangle inequality $\mu_n^0$ satisfies the stability condition \eqref{0sc}. Moreover, up to subsequences, 
there exists $\mu^0\in{\mathcal P}(\Omega)$ such that $\mu_n^0\weakto\mu^0$
narrowly, as $n\to\infty$. From the minimality we also deduce that
$$
\mathcal F_n(\mu_n^0)  + d(\mu_n^0, \hat\mu^0_n) - \int_\Omega f(0)\,d\mu_n^0
\leq \mathcal F_n(\hat \mu_n^0)  - \int_\Omega f(0)\,d\hat\mu_n^0,
$$
hence $\mathcal F_n(\mu^0_n)\leq C$ for every $n$ and from the liminf inequality of Theorem~\ref{ups:density}
we deduce that
$\mathcal F(\mu^0)<+\infty$.
It remains to show \eqref{e-data}. By Theorem~\ref{thm:joint} there exists a sequence $(\nu_n^0)$ 
such that $\nu^0_n\in X_n$ for every $n$, $\nu_n^0\weakto\mu^0$ narrowly, as $n\to\infty$, and
\begin{equation}\label{app-data}
\lim_{n\to\infty} d(\mu^0_n,\nu^0_n)=0, \qquad \limsup_{n\to\infty} \mathcal F_n(\nu_n^0)\leq \mathcal F(\mu^0).
\end{equation}
From the stability condition \eqref{0sc} we have
$$
\mathcal F_n(\mu_n^0) - \int_{\Omega}f(0)\,d\mu_n^0 
\leq \mathcal F_n(\nu_n^0)+d\big(\nu_n^0, \mu_n^0\big) - \int_{\Omega}f(0)\,d\nu_n^0.
$$
Owing to the liminf inequality of Theorem~\ref{ups:density} and to \eqref{app-data},
we can pass to the limit in the above expression and get \eqref{e-data}.
\end{remark}

\begin{remark}\label{rmk:projL}
Note that if $t\mapsto\mu(t)$ is a quasi-static evolution, then
\begin{equation}\label{cmar}
(\pi_2)_\#\mu(t)=(\pi_2)_\#\mu(0) \quad \text{for every } t\in[0,T].
\end{equation}
Indeed, owing to the regularity assumptions on $f$ and to \eqref{lower-bound:calFn}, 
condition (qs2)$_\infty$ implies that $\mathcal{D}(\mu,[0,T])<+\infty$, which in turn gives \eqref{cmar} by the definition of the dissipation distance $d$.
\end{remark}

\begin{proof}[Proof of Theorem~\ref{th:conv_evolutions}]
\textbf{Compactness.} 
Using the regularity assumption on $f$ and assumption \eqref{e-data} on the initial data,
we infer from the energy balance (qs2)$_n$
\begin{equation}\label{almob}
\mathcal{F}_n(\mu_n(t)) + \mathcal{D}(\mu_n,[0,t])
\leq  C
\end{equation}
for every $n$ and every $t\in[0,T]$.
On the other hand, the energy estimates \eqref{Bclaim} and \eqref{Dclaim} in the proof of Theorem~\ref{ups:density} yield
\begin{multline*}
\mathcal{F}_n(\mu_n(t)) = \frac12
\iint_{\Omega\times\Omega}V(y,z)\, d(\mu_n(t)\boxtimes\mu_n(t))(y,z)
\\
+\frac{1}{2n} \sum_{i=1}^{n}\int_{\partial\Omega}\C K^n(x;z_i^n(t))\nu\cdot  u_{\mu_n}(t) \, d\HH^1(x)
+o(1),
\end{multline*}
where $o(1)$ is a quantity tending to $0$, as $n\to\infty$, uniformly in $t$. By Lemma~\ref{lemma:u-est} and Remark~\ref{rmk:bb}
we deduce that
$$
\mathcal{F}_n(\mu_n(t)) \geq -C
$$
for every $n$ and every $t$. Combining this inequality with \eqref{almob}, we obtain
\begin{equation}\label{MG-est2}
\mathcal{D}(\mu_n,[0,T])\leq C.
\end{equation}
In particular, this implies that $d(\mu_n(t),\mu_n^0)<\infty$ for every $t\in[0,T]$ and every $n$
and, in turn, that $(\pi_2)_\#\mu_n(t)=(\pi_2)_\#\mu_n^0$ for every $t\in[0,T]$ and every $n$.
Since $\mu_n^0\in Y_n(\gamma,c)$ by assumption, we conclude that $\mu_n(t)\in Y_n(\gamma,c)$ for every $t\in[0,T]$ and every $n$.

By Theorem~\ref{thm:Helly} and \eqref{MG-est2} we can guarantee that
there exists a map $t\mapsto\mu(t)$ from $[0,T]$ into $\mathcal{P}(\Omega)$
such that, up to a subsequence,
$$
\mu_n(t)\weakto \mu(t)\quad \text{ narrowly,}
$$
as $n\to\infty$, for every $t\in[0,T]$. By definition of $\mathcal{P}_{\gamma,c}^{\infty}(\Omega)$ we clearly have that $\mu(t)\in \mathcal{P}_{\gamma,c}^{\infty}(\Omega)$ for every $t\in[0,T]$.
Moreover, by \eqref{almob} and Theorem~\ref{ups:density} we have
\begin{equation}\label{inflim}
\mathcal F(\mu(t))\leq \liminf_{n\to\infty}\mathcal F_n(\mu_n(t))\leq C
\end{equation}
for every $t\in[0,T]$; in other words, $\mathcal F(\mu(t))<\infty$ for every $t\in[0,T]$.
\medskip

\noindent
\textbf{Convergence.} We now prove condition (qs1)$_\infty$.
Let us fix $t\in[0,T]$ and let $\nu\in \P(\Omega)$. Clearly, it is enough to prove \eqref{0min-1} for
$\nu\in \P(\Omega)$ such that $\mathcal F(\nu)<\infty$ and $d(\nu,\mu(t))<\infty$. This last condition implies, in particular, that
$(\pi_2)_\#\nu = (\pi_2)_\#\mu(t)$. By Theorem~\ref{thm:joint} there exists a sequence $(\nu_n)$ such that 
$\nu_n\in Y_n(\gamma,c)$ for every $n$, $\nu_n\weakto \nu$ narrowly, as $n\to \infty$, and
\begin{equation}\label{rq1}
\limsup_{n\to\infty}d(\nu_n,\mu_n(t)) \leq d(\nu,\mu(t)) \quad\textrm{and} \quad \limsup_{n\to\infty}\mathcal F_n(\nu_n) \leq \mathcal F(\nu).
\end{equation}
Since $\nu_n\in X_n$ for every $n$, the minimality condition (qs1)$_n$ implies
$$
\mathcal F_n(\mu_n(t)) - \int_{\Omega}f(t)\,d\mu_n(t) 
\leq d\big(\nu_n, \mu_n(t)\big) +\mathcal F_n(\nu_n) - \int_{\Omega}f(t)\,d\nu_n.
$$
Combining \eqref{inflim} and \eqref{rq1}, we can pass to the limit in this inequality and obtain \eqref{0min-1}.

By \eqref{inflim}, the lower semicontinuity of the dissipation, and \eqref{e-data}, we can now pass to the limit in
the energy balance (qs2)$_n$ and prove an energy inequality. The converse energy inequality can be deduced 
from the global stability.
\end{proof}


\section{Strong formulation of the evolution problems}
\label{sec:strong-solutions}
In this section we derive the strong formulation of the quasi-static evolution for the discrete renormalised energy~\eqref{def:curly-F-intro} and for the limit energy \eqref{def:limitF}, under the assumption that they are sufficiently regular for the various arguments to be valid.

\subsection{Discrete energy}\label{subs:dis}
Our starting point is the formulation \eqref{e-stab}--\eqref{e-bal} of the quasi-static evolution for the renormalised energy.
For simplicity of notation, since $\mu(t)=\frac1n\sum_{i=1}^n\delta_{z_i(t)}$ we write
$$
\mathcal{F}_n(\mu(t))=\frac1n\int_\Omega W_n(x, z_1(t), \dots, z_n(t))\, dx.
$$
We note that the force term reduces to
$$
\int_{\Omega}f(t) \,d\mu(t) = \frac1n\sum_{i=1}^nf(t,z_i(t)).
$$
We assume for simplicity that $\supp\,\mu(t) \subset \textrm{int}\,{\mathcal R}$ and $|z_i(t)-z_j(t)|>r_n$, $i\neq j$, for every $t\in[0,T]$.
In (qs1)$_n$ we consider the measure $\frac1n\sum_{i=1}^n \delta_{z_i(t)+\eta_ie_1}$ as a competitor of $\mu(t)$, with $|\eta_i|$ small enough,
and we obtain
\begin{multline*}
\frac1n\int_\Omega W_n(x,z_1(t), \dots, z_n(t))\, dx - \frac1n\sum_{i=1}^n f(t,z_i(t)) 
\\
\leq \frac1n\int_\Omega W_n(x,z_1(t)+\eta_1e_1, \dots, z_n(t)+\eta_ne_1)\, dx+\frac1n\sum_{i=1}^n|\eta_i| - \frac1n\sum_{i=1}^n f(t,z_i(t)+\eta_ie_1).
\end{multline*}
Choosing $\eta_i>0$ and $\eta_j=0$ for $i\neq j$, we can divide by $\eta_i$ and pass to the limit as $\eta_i\to0$, and obtain
\begin{equation}
\label{ineq:strong1}
-\int_\Omega \partial_{z_i}\!W_n(x,z_1(t),\dots,z_n(t))\cdot e_1\, dx
+ \partial_{z_i}f(t,z_i(t))\cdot e_1 \leq 1.
\end{equation}
Repeating the same argument for $\eta_i<0$ we have
$$ 
-\int_\Omega \partial_{z_i}\!W_n(x,z_1(t),\dots,z_n(t))\cdot e_1\, dx
+ \partial_{z_i}f(t,z_i(t))\cdot e_1 \geq -1.
$$
Therefore, we deduce
\begin{equation}\label{stab}
\Big|-\int_\Omega \partial_{z_i}\!W_n(x,z_1(t),\dots,z_n(t))\cdot e_1\, dx
+ \partial_{z_i}f(t,z_i(t))\cdot e_1 \Big|
\leq 1
\end{equation} 
for every $i=1,\dots,n$.

Assume now that $\mu(t)$ varies smoothly with respect to time, so that
\begin{equation}
\label{eq:D-discrete-strong}
\mathcal{D}(\mu,[0,t])=\frac1n\sum_{i=1}^n\int_0^t  |\dot z_i(s)\cdot e_1|\, ds.
\end{equation}
By \eqref{cmar0} we also have that
\begin{equation}\label{smoothcmar0}
\dot z_i(t)\cdot e_2=0
\end{equation}
for every $i=1,\dots,n$.
Differentiating (qs2)$_n$ with respect to time, we obtain
$$
\sum_{i=1}^n \dot z_i(t) \cdot \left(-\int_\Omega\partial_{z_i}\!W_n(x,z_1(t),\dots,z_n(t))\, dx  + \partial_{z_i}f(t,z_i(t)) \right) =
\sum_{i=1}^n |\dot z_i(t)\cdot e_1|.
$$
In view of \eqref{stab} and \eqref{smoothcmar0}, this is equivalent to
\begin{equation}\label{flow}
\left(-\int_\Omega\partial_{z_i}\!W_n(x,z_1(t),\dots,z_n(t))\cdot e_1\, dx  + 
\partial_{z_1}f(t,z_i(t))\cdot e_1 \right) (\dot z_i(t)\cdot e_1) = |\dot z_i(t)\cdot e_1|
\end{equation} 
for every $i=1,\dots,n$.

The flow rule \eqref{flow} is rate-independent. Moreover, together with \eqref{stab}, it implies that
when the inequality in \eqref{stab} is strict for some $i$, then it must be $\dot z_i\cdot e_1=0$.
When \eqref{stab} holds with the equality for some $i$, then $\dot z_i\cdot e_1$ may be different from zero and
satisfies
$$
\dot z_i(t)\cdot e_1 =\lambda   \left(-\int_\Omega\partial_{z_i}\!W_n(x,z_1(t),\dots,z_n(t))\cdot e_1\, dx  + 
\partial_{z_1}f(t,z_i(t))\cdot e_1 \right)
$$
for some $\lambda\geq0$.

\begin{remark}
The formal derivation performed above can be justified rigorously under  higher regularity of the forcing term, namely if $f\in W^{1,1}(0,T;C^1(\Omega))$, and if $W_n$ is continuously differentiable 
with respect to the points in the support of $\mu$. This latter condition depends on the regularity of the functions $z\mapsto K_z$.
\end{remark}

\subsection{Limiting energy}
The limiting quasi-static evolution is in this case given by \eqref{0min-1}--\eqref{0e-bal-1}, where we set
$$
\widetilde{\mathcal F}(\mu,t):= {\mathcal F}(\mu)-\int_\Omega f(t)\, d\mu.
$$
We now follow the same arguments as in Section~\ref{subs:dis}, but for the energy $\widetilde{\mathcal F}$.
We assume for simplicity that $\supp\,\mu(t) \subset \textrm{int}\,{\mathcal R}$ for every $t\in[0,T]$.
We choose $\varphi\in C_c^\infty(\Omega)$, and for $\eta>0$ small enough we define perturbations $\mu^\eta$ of $\mu(t)$ by (writing $x=(x_1,x_2)$ for an element of $\Omega$)
$$
\mu^\eta = T^\eta _{\#} \mu(t), 
  \qquad T^\eta(x) := \bigl(x_1+\eta\varphi(x),x_2\bigr).
$$
Note that if $\eta$ is small enough, then $T^\eta$ is a smooth one-to-one map from $\Omega$ to $\Omega$. Since the map $x \mapsto T^\eta(x)$ is an admissible transport map, we have 
$$
d(\mu^\eta,\mu(t))\leq \int_\Omega |x_1-(x_1+\eta\varphi(x))| \, d\mu(t)(x)
 = \eta \int_\Omega |\varphi(x)| \,d\mu(t)(x).
$$
To estimate the effect of this perturbation on $\widetilde{\mathcal F}$, we note that $\partial_\eta\mu^\eta = -\partial_{x_1} (\varphi\mu^\eta)$ at $\eta=0$, and calculate
\begin{equation}
\label{eq:calc-deriv-transport}
\lim_{\eta\to0} \frac1\eta\bigl[ 
  \widetilde{\mathcal F}(\mu^\eta,t)-\widetilde{\mathcal F}(\mu(t),t)\bigr]\Big|_{\eta=0}
  = \int_\Omega \frac{\delta \widetilde{\mathcal F}}{\delta \mu}(\mu(t),t) \, \partial_\eta \mu^\eta\big|_{\eta=0} 
  = \int_\Omega\varphi \; \partial_{x_1}\frac{\delta \widetilde{\mathcal F}}{\delta \mu}(\mu(t),t) \,d\mu(t),
\end{equation}
where 
\begin{equation}\label{var:derivative}
\frac{\delta \widetilde{\mathcal F}}{\delta \mu}(\mu,t)(x) = \int_\Omega V(x,y)\, d\mu(y) 
+ \frac12 \int_{\partial\Omega } \C K(x;y)\nu(y)\cdot v_\mu(y)\,d\HH^1(y)
- f(x,t).
\end{equation}
The second term follows from~\eqref{def:limitF} using the general principle that if $h(s) = \inf_v g(v,s)$, with unique minimizer $v_s$ depending smoothly on~$s$, then $h'(s) = (\partial_s g)(v_s,s)$. 
\medskip

Now using $\mu^\eta$ in (qs1)$_\infty$, dividing by $\eta$ and taking the limit $\eta\to0$, we find
$$
-\int_\Omega\varphi \; \partial_{x_1}\frac{\delta \widetilde{\mathcal F}}{\delta \mu}(\mu(t),t) \,d\mu(t) 
\leq \int_\Omega |\varphi|\, d\mu(t),
$$
which is the equivalent of~\eqref{ineq:strong1}. Since $\varphi$ is arbitrary, we similarly deduce that 
\begin{equation}
\label{stab-ct}
\Big|\partial_{x_1}\frac{\delta \widetilde{\mathcal F}}{\delta \mu}(\mu(t),t)(x)\Big|
\leq 1
\qquad \text{for $\mu(t)$-a.e.\ $x\in \Omega$.}
\end{equation}
Continuing the argument from the discrete case, we assume that $\mu$ is smooth in space and time, and we write $d\mu(t)= \mu(t,x)\,dx$. By~\eqref{s:dissipation} there exists a Borel measurable $\phi$ satisfying the equation 
\[
\partial_t \mu(t,x) +\partial_{x_1}(\phi(t,x)\mu(t,x))=0 \qquad \text{for all $t$ and for $\mu(t)$-a.e.\ $x$},
\]
such that 
\[
\mathcal{D}(\mu,[0,t])=\int_0^t \int_\Omega  |\phi(s,x)|\,\mu(s,x)\,dx\,ds.
\]
%
%
Differentiating (qs2)$_\infty$ in time and calculating as in~\eqref{eq:calc-deriv-transport} we find
\[
\int_\Omega \phi(t,x)\, \partial_{x_1}\frac{\delta\widetilde{\mathcal F}}{\delta\mu}(\mu(t),t)\, \mu(t,x)\,dx 
+ \int_\Omega |\phi(t,x)|\, \mu(t,x)\,dx = 0.
\]

Again, in combination with~\eqref{stab-ct}, this implies
\[
- \phi(t,x)\,\partial_{x_1} \frac{\delta\widetilde{\mathcal F}}{\delta\mu}(\mu(t),t)
 = |\phi(t,x)| \qquad
 \text{for $\mu(t)$-a.e.\ $x$,}
\]
which is the counterpart of~\eqref{flow}. At $\mu(t)$-a.e.\ $x$, therefore, \emph{either} $\phi(x,t)=0$, \emph{or} the total force satisfies $-[\partial_{x_1}\delta\widetilde{\mathcal F}/{\delta\mu}](\mu(t),t)(x) = \pm 1$; in the latter case, $\phi$ points in the direction of the force.


\section{Appendix}
In this appendix we state and prove an extension result (Theorem~\ref{thm:extension}), and a Helly-type theorem. 

\begin{theorem}\label{thm:extension}
Let $\Omega$ be an open set in $\R^2$.
Let $\delta$ and let $z_1,\dots,z_n$, $n\in\N$, be such that $|z_j-z_k|\geq 4\delta$ for every $j\neq k$
and $\dist(z_i,\partial\Omega)\geq 2\delta$ for every $i$. Then there exists a positive constant $C$, independent of $\delta$, of the $z_i$'s and of $n$, with the following property: for every $u\in H^1(\Omega\setminus \cup_{i=1}^n \overline B_\delta(z_i);\R^2)$ there exists an extension $\tilde u\in H^1(\Omega;\R^2)$ of $u$ such that
\begin{equation}\label{app:ext}
\|\sym\nabla \tilde u\|_{L^2(\Omega)}\leq C \|\sym\nabla u\|_{L^2(\Omega\setminus \cup_{i=1}^n \overline B_\delta(z_i))}.
\end{equation}
\end{theorem}

\begin{proof}
Let $u\in H^1(\Omega\setminus \cup_{i=1}^n \overline B_\delta(z_i);\R^2)$.
By \cite[Lemma~4.1]{OSY} for every $i=1,\dots, n$ there exists a function $\tilde u^i\in H^1(B_{2\delta}(z_i))$ such that $\tilde u^i=u$ on $B_{2\delta}(z_i)\setminus \overline B_\delta(z_i)$ and
\begin{equation}\label{ext1}
\|E\tilde u^i\|_{L^2(B_{2\delta}(z_i))}\leq C_1 \|Eu\|_{L^2(B_{2\delta}(z_i)\setminus \overline B_\delta(z_i))},
\end{equation}
where the constant $C_1$ is independent of $u$. By a scaling argument it is easy to see that $C_1$ is also independent of $\delta$ and of $z_i$.

It is now enough to define
$$
\tilde u:=\begin{cases}
\tilde u^i & \text{ in } B_\delta(z_i),
\\
u & \text{ otherwise.}
\end{cases}
$$
Clearly $\tilde u$ is an extension of $u$ and $\tilde u\in H^1(\Omega;\R^2)$. Moreover, \eqref{ext1} yields
\begin{eqnarray*}
\|E\tilde u\|_{L^2(\Omega)}^2 & = & \|Eu\|_{L^2(\Omega\setminus \cup_{i=1}^n \overline B_\delta(z_i))}^2
+ \sum_{i=1}^n \|E\tilde u_i\|_{L^2(B_\delta(z_i))}^2
\\
& \leq &
\|Eu\|_{L^2(\Omega\setminus \cup_{i=1}^n \overline B_\delta(z_i))}^2
+ C_1^2\sum_{i=1}^n \|E u\|_{L^2(B_{2\delta}(z_i)\setminus \overline B_\delta(z_i))}^2 
\\
& \leq & (1+C_1^2) \|Eu\|_{L^2(\Omega\setminus \cup_{i=1}^n \overline B_\delta(z_i))}^2,
\end{eqnarray*}
which proves \eqref{app:ext}.
\end{proof}

\begin{theorem}\label{thm:Helly}
Let $t\mapsto\mu_k(t)$ be a sequence of maps from $[0,T]$ into the space $\P(\Omega)$.  We assume that there exist an open set $\Omega'\subset\subset\Omega$ and  $C>0$ such that 
\begin{equation}\label{const-mass}
\supp\,\mu_k(t)\subset\Omega' 
\end{equation}
for every $t\in[0,T]$ and every $k$, and
\begin{equation}\label{bdd-var}
{\mathcal D}(\mu_k,[0,T])\leq C
\end{equation}
for every $k$. Then there exists a subsequence (not relabelled) and a map $t\mapsto\mu(t)$ from $[0,T]$
into $\P(\Omega)$, with ${\mathcal D}(\mu,[0,T])\leq C$,
such that
\begin{equation}\label{app-th}
\mu_k(t)\weakto \mu(t) \quad \text{ narrowly}
\end{equation}
for every $t\in[0,T]$.
\end{theorem}

\begin{proof}
Let $Q$ be a countable and dense subset of $[0,T]$. Using a diagonal argument,
assumption \eqref{const-mass} implies the existence of a subsequence $k_j\to\infty$ and a map $t\mapsto\mu(t)$
from $Q$ into $\mathcal P(\Omega)$ such that 
$$
\mu_{k_j}(t) \weakto \mu(t) \quad \text{ narrowly} 
$$
for every $t\in Q$.
We now define
$$
D_{k_j}(t):= \mathcal{D}(\mu_{k_j},[0,t]).
$$ 
The functions $D_{k_j}$ are non-decreasing with respect to time; moreover, $D_{k_j}(T)\leq C$ and $D_{k_j}(0)=0$ for every $j$. 
By Helly's Theorem we deduce that, up to subsequences, 
$$
\lim_{j\to\infty}D_{k_j}(t)=D(t) \quad \forall\, t\in [0,T],
$$
and the limit function $D(t) $ is also non decreasing in $t$ and bounded by the same constant $C$. We now define the set 
$$
\Theta:= \{0\} \cup \Big\{t\in (0,T]: \lim_{s\to t^{-}}D(s) = D(t)\Big\}.
$$
Note that, by definition, for every $0\leq s<t\leq T$ we have
\begin{equation}\label{Helly}
d(\mu_{k_j}(t),\mu_{k_j}(s)) \leq D_{k_j}(t) - D_{k_j}(s).
\end{equation}
Let $t\in \Theta$. Then, for every $\e>0$ there exists $s\in Q$ such that $s<t$ and $D(t)-D(s)<\e$. Then, by \eqref{Helly}, for large enough $j$ (note that $j=j(t,s)$)
$$
d(\mu_{k_j}(t),\mu_{k_j}(s)) < \e.
$$
Then, for every $i,j \in \N$ we have the bound
$$
d(\mu_{k_i}(t),\mu_{k_j}(t)) < 2\e + d(\mu_{k_i}(s),\mu_{k_j}(s)),
$$ 
entailing that $(\mu_{k_j}(t))$ is a Cauchy sequence in $d$, and hence in $d_1$ by Part~\ref{lemma:props-d:12} of Lemma~\ref{lemma:props-d}. Therefore, $\mu_{k_j}(t)$ converges narrowly to some limit measure $\mu(t)$ for $t\in\Theta$. Since the complement of $\Theta$ in $[0,T]$ is at most countable, up to extracting a further subsequence we have by a diagonal argument that
$\mu_{k_j}(t)$ converges narrowly to a measure $\mu(t)$ for every $t\in[0,T]$.
\end{proof}

\bigskip

\noindent
\textbf{Acknowledgements.}
M.G.M.\ acknowledges support from the European Research Council under Grant No.\ 290888
``Quasistatic and Dynamic Evolution Problems in Plasticity and Fracture'' and from GNAMPA--INdAM. 
The work of M.A.P. is supported by NWO Complexity grant 645.000.012 and NWO VICI grant 639.033.008. 
L.S. acknowledges the support of The Carnegie Trust. All authors are grateful for various enlightening discussions with Adriana Garroni, Patrick van Meurs, and Marcello Ponsiglione.
\bigskip

\bibliographystyle{alpha} 
\bibliography{Quasistatic}

\end{document}